\documentclass[reqno, 12pt]{amsart}
\usepackage{amsfonts}
\usepackage{mathrsfs}
\usepackage{xcolor}
\usepackage{amsmath}
\usepackage{graphicx, amsmath, amsfonts, amssymb, amscd, amsthm}
\usepackage{CJK}
\makeatletter
\def\blfootnote{\gdef\@thefnmark{}\@footnotetext}
\makeatother

\newtheorem{theorem}{Theorem}
\newtheorem{lemma}{Lemma}[section]
\newtheorem{prop}{Proposition}[section]

\newcommand{\be}{\begin{eqnarray}}
\newcommand{\ba}{\begin{array}}
\newcommand{\ben}{\begin{eqnarray*}}
\newcommand{\ee}{\end{eqnarray}}
\newcommand{\ea}{\end{array}}
\newcommand{\een}{\end{eqnarray*}}
\newcommand{\ms}{\medskip}
\newcommand{\NN}{\hbox{I\kern-.2em\hbox{N}}}
\def\beq{\begin{equation}}
\def\eeq{\end{equation}}
\def\beq{\begin{equation*}}
\def\eeq{\end{equation*}}

\newcommand{\bl}{\be\label}\newcommand{\va}{\theta}\newcommand{\p}{\partial}
\newcommand{\al}{\alpha}
\newcommand{\dsp}{\displaystyle}
\numberwithin{equation}{section}\newcommand{\tin}{\to\infty}
\newcommand{\mrn}{\mathbb{R}^n}
\newcommand{\lm}{\lambda}

\newcommand{\equ}[1]{(\ref{#1})}
\newcommand{\clb}{\color{black}}\newcommand{\clr}{\color{black}}

\begin{document}
\title[Uniqueness of bound states]
{Uniqueness of bound states to \\
$\Delta u-u+|u|^{p-1}u= 0$ in $\mrn$, $n\ge 3$}

\author[M. Tang]{Moxun Tang}
\address{Department of Mathematics, Michigan State University,
East Lansing, Michigan, 48824, USA} \email{tangm@msu.edu}

\blfootnote{\textup{2020} \textit{Mathematics Subject Classification}:
35B05, 34B40, 35J66, 35Q51}

\begin{abstract} We give a positive answer to a conjecture of Berestycki and Lions
in 1983 on the uniqueness of bound states to $\Delta u +f(u)=0$ in $\mrn$, $u\in H^1(\mrn)$,  $u\not\equiv 0$, $n\ge 3$. For the model nonlinearity $f(u)=-u+|u|^{p-1}u$, $1<p<(n+2)/(n-2)$, {\clr arising} from finding standing waves of Klein-Gordon equation or nonlinear Schr\"odinger equation, we show that, for each integer $k\ge 1$, the problem has a unique solution $u=u(|x|)$, $x\in \mrn$, up to translation and reflection, that has precisely $k$ zeros for $|x|>0$.
\end{abstract}

\maketitle
\section {Introduction}
In their celebrated studies \cite{bl1, bl2}, Berestycki and Lions considered solutions of the semi-linear elliptic problem
\bl{ble} \Delta u +f(u)=0\,\,\,{\rm in}\,\,\, \mrn,\quad u\in H^1(\mrn), \quad u\not\equiv 0, \quad n\ge 3,\ee
where $f: \mathbb{R}\to \mathbb{R}$ is assumed to be a continuous and odd function satisfying

(H1) $\,\, -\infty<\liminf_{s\to 0}f(s)/s\le \limsup_{s\to 0}f(s)/s=-m<0.$

(H2) $\,\, -\infty\le \limsup_{s\tin}f(s)/s^{(n+2)/(n-2)}\le 0.$

(H3) $\,\,$ There exists $\xi>0$ such that $F(\xi)=\int_0^\xi f(s)\, ds>0$.

By using variational methods, they proved the following classic result:

\noindent
{\bf The existence theorem of Berestycki and Lions \cite{bl1, bl2}.} {\it Let $f: \mathbb{R}\to \mathbb{R}$ be a continuous and odd function satisfying (H1), (H2) and (H3). Then problem \equ{ble} possesses an infinite sequence of distinct solutions $(u_k)_{k\ge 0}$ such that

i) $u_k\in C^2(\mrn)$ and is radial: $u(x)=u(r)$, where $r=|x|$.

ii) $u_k$ together with its derivatives up to order 2 have exponential decay at infinity.

iii) $u_0>0$ in $\mrn$ and decreases with respect to $r$.}

A positive solution of \equ{ble} is called a {\it ground state}. Under conditions (H1)-(H3), a ground state is necessarily radial \cite{gnn}. A sign-changing radial solution of \equ{ble} is called a {\it bound state} (or {\it excited state} \cite{cls, rr}). The existence theorem of Berestycki and Lions identifies sufficient and nearly necessary conditions for existence of ground states and bound states of \equ{ble}, and covers the most important model case
\begin{equation}\label{ble1}
\begin{cases} \Delta u +f(u)=0\,\,\,{\rm in}\,\,\, \mrn,
\quad u\in H^1(\mrn), \quad u\not\equiv 0,
\\
f(u)=-u+|u|^{p-1}u,\,\,\,\quad 1<p<\frac{n+2}{n-2}, \quad n\ge 3.
\end{cases}
\end{equation}
For the critical or supercritical exponent $p\ge (n+2)/(n-2)$, one can use the identity of Pohozaev \cite{poh} to show that \equ{ble1} admits no radial solutions.

In practice, a solution of $\Delta u + f(u) = 0$ can be interpreted as a steady-state of the reaction-diffusion equation $u_t=\Delta u + f(u)$.
It arises, for instance, in the study of phase transitions \cite{serrin81}, ecological modeling \cite{ln}, {\clb the elimination of dengue and chikungunya viruses \cite{hty}, and in the analysis} of Turing instability and pattern formation \cite{bas, nt, nit}.
The study of \equ{ble1} is motivated in particular by the search for standing waves, a special type of solitary waves, in nonlinear Klein-Gordon equations or Schr\"odinger equations \cite{bl1, st77, tao}. Consider the Klein-Gordon equation
\bl{kg} \phi_{tt}-\Delta \phi +m^2 \phi= |\phi|^{p-1}\phi, \quad m>0.\ee
If $\phi$ is a standing wave, that is, $\phi(x, t)=e^{i\omega t}u(x)$, $\omega\in \mathbb{R}$, $|\omega|<m$, and $u: \mrn\to \mathbb{R}$, then one is led to the equation
\bl{kg1} \Delta u + (\omega^2-m^2)u + |u|^{p-1}u= 0.\ee
Similarly, for the nonlinear Schr\"odinger equation
\bl{sde} i\phi_t+\Delta \phi+|\phi|^{p-1}\phi=0,\ee
the real function $u(x)$ in the standing wave solution $\phi(x, t)=e^{i\omega t}u(x)$ satisfies
\bl{sde1} \Delta u-\omega u+|u|^{p-1}u= 0.\ee
Both equations \equ{kg1} and \equ{sde1} can be transformed to \equ{ble1} by scaling \cite{tao}.

Back {\clb in} 1951, Finkelstein, LeLevier, and Ruderman \cite{flr} introduced a scalar field equation, which is a special case of \equ{kg}, to investigate certain properties in elementary particle theory. The real part of the standing wave solves (after scaling)
\bl{ble2}
\Delta u -u+u^3=0\,\,\,{\rm in}\,\,\, \mathbb{R}^3,
\quad u\in H^1(\mathbb{R}^3), \quad u\not\equiv 0.\ee
Through a phase-plane analysis, they demonstrated that \equ{ble2} admits a ground state and at
least one particle-like solution (bound state) with any given number of nodes (simple zeros).
One year later, Rosen and Rosenstock \cite{rr} approximated the interaction
between two far apart identical particles by using the ground state of \equ{ble2}, and showed that the interaction is described by the Yukawa potential. Of course, \equ{ble2} is a special case of \equ{ble1} with $n=p=3$. It can also be viewed as the equation for $u(x)$ in the standing wave for the cubic Schr\"odinger equation in $\mathbb{R}^3$ \cite{mps, sch}.

 A rigorous proof for the existence of ground states of \equ{ble2} was
 given by Nehari \cite{ne2}, who also guided Ryder \cite{ryd} to prove the existence of bound states of \equ{ble2} by applying the ODE techniques from \cite{ne1}. The existence of infinitely many radial solutions to \equ{ble} for a large class of $f(u)$, including the model case \equ{ble1}, was first proved by Strauss \cite{st77}, and {\clb was later brought to prominence in the celebrated studies of Berestycki and Lions \cite{bl1, bl2}}. Existence of radial solutions of \equ{ble1} with any prescribed number of zeros was proved by Jones and K\"upper \cite{jk}
 by using a dynamical systems approach; a simpler proof by a direct ODE method was given later by McLeod, Troy and Weissler \cite{mtw}.

In the seminal work \cite{we2}, Weinstein observed that one could use solutions of \equ{ble1}
to construct a large family of solutions to the Schr\"odinger equation \equ{sde}, or
to characterize the solutions of \equ{sde} that develop singularities. In \cite{we1}, Weinstein also noticed an interesting relation between finding a ground state of \equ{ble1} and estimating the ``best constant" in the Gagliardo-Nirenberg inequality
\ben \frac {\int_{\mrn} |u|^{p+1}}
{\left(\int_{\mrn} |\nabla u|^2 \right)^{n(p-1)/4} \left(\int_{\mrn} |u|^2 \right)^{1-\frac{(n-2)(p-1)}4} }\le C_{n, p},\qquad 0\not\equiv u\in H^1(\mrn). \een
If the ratio in the left is maximized at a function $u\in H^1(\mrn)$, then
$u$ solves \equ{ble1} in the sense of distributions (after scaling); see also
Theorems 1--2 in Frank \cite{fra}, and Lemma B.1 in
Tao \cite{tao}. In this respect, {\clb the} uniqueness of ground states is essential, as further evidenced by the important work of
Del Pino and Dolbeault \cite{del}.

In the concluding section, Berestycki and Lions  \cite{bl1, bl2} proposed a conjecture on the characterization of bound states, {\clr namely the sign-changing radial solutions of \equ{ble}}.

\noindent
{\bf The conjecture of Berestycki and Lions \cite{bl1, bl2}.} {\it We conjecture that, at least for some classes of $f$'s{\footnote {In \cite{bl1, bl2}, the nonlinearity in \equ{ble} is {\clb denoted by $g$; in this work, we relabel it as} $f$. }}, there is one and exactly one solution of \equ{ble} that has precisely
$k-1$ nodes (i.e. simple zeroes) for $r\in (0, \infty)$}.

{\clb For the model system \equ{ble1}, proving this conjecture is equivalent to showing that there
 exists} exactly one $\al$ such that the solution $u(r)$ to the initial value problem
\begin{equation}\label{eu}
\begin{cases}
u''+\frac {n-1}ru'+f(u)=0,\quad
u(0)=\alpha>0,\quad u'(0)=0, \quad r>0.\\
f(u)=-u+|u|^{p-1}u,\qquad 1<p<\frac{n+2}{n-2}, \quad  {\clr n\in \mathbb{N},} \quad n\ge 3,  
\end{cases}
\end{equation}
has precisely $k\ge 0$ simple zeros and $|u(r)|$ decays exponentially at infinity. {\clb Despite
 its singularity} at $r=0$, \equ{eu} admits a unique solution $u\in C^2 ([0,\infty))$ that depends smoothly on $\al$ \cite{bl1, blp, tao}. {\clb By the uniqueness theorem for ODEs, any zero of $u$ is simple. If $u_k(r)$ is a $k$-node bound state of \equ{eu}, then the functions $u_k(|x-x_0|)$, with $x_0\in \mrn$, together with} $-u_k(|x-x_0|)$, form {\clr a} family of $k$-node bound states {\clb of \equ{ble1} by virtue of translational and reflection invariance}.

For the special case \equ{ble2}, the same question was also asked by Hastings and McLeod in \cite{ham}, Chapter 19: {\it Three unsolved problems}.

\noindent
{\bf The open problem of Hastings and McLeod \cite{ham}.} {\it Prove that for each positive integer k the boundary value problem
\bl{hm} u''+\frac 2ru'-u+u^3=0,\quad u'(0)=0,\quad \lim_{r\tin}u(r)=0,\ee
has at most one solution which is initially positive and has exactly $k$ zeros
in $(0, \infty)$.}

{\clb Significant progress has been achieved} on these open problems. In \cite{cgy},
Cort\'azar, Garc\'ia-Huidobro and Yarur proved the uniqueness of {\clb $1$-node} bound states of \equ{ble} for a {\clb more general} nonlinearity $f(u)=-|u|^{q-1}u+|u|^{p-1}u$, $0<q<p$,
$n=2, 3, 4$, and $p+q\le 2/(n-2)$. {\clb Remarkably, when viewing the dimension $n$ as a real variable in $(1, \infty)$, their uniqueness theorem requires no further restrictions on $p$ and $q$ for $1<n\le 2$.} In \cite{awy}, Ao, Wei and Yao proved the uniqueness of bound states of \equ{ble1} {\clb for $p$ sufficiently} close to $(n+2)/(n-2)$, with their proof {\clb focusing on $1$-node} bound states.
{\clb An exceptional advance was recently achieved} by Cohen, Li and Schlag \cite{cls}, {\clb who} nearly solved the open problem of Hastings and McLeod \cite{ham}: By {\clb means of} a rigorous computer-assisted analytical {\clb method, they established}

\noindent
{\bf The uniqueness theorem of Cohen, Li and Schlag \cite{cls}.} {\it The first twenty bound states of {\clb \equ{ble2} (or \equ{hm})} are unique}.

In this work, we prove
\begin{theorem}\label{hh1} For each positive integer $k$, there exists a unique bound state {\clb of \equ{ble1} with} precisely $k$ nodes for $r\in (0,\infty)$, up to translation and reflection.

{\clb Moreover, for each finite ball $B \subset \mrn$, there exists a unique radial solution to
\ben \Delta u-u+|u|^{p-1}u=0 \,\,{\rm in}\,\, B, \quad u=0\,\,{\rm on}\,\,\partial B, \quad 1<p<\frac{n+2}{n-2}, \quad n\ge 3,\een
that changes sign exactly $k$ times in the interior of the ball, up to reflection.}
\end{theorem}

{\clb The first part of} Theorem \ref{hh1} is an immediate consequence of the following result that {\clb characterizes} all solutions of \equ{eu}. To state the result, {\clb we introduce
the unique positive zero $\al_*$ of $F(u){\clr =\int_0^uf(s)\, ds}$ and the unique positive zero $\al^*$ of $2nF(u)-(n-2)uf(u)$: }
\bl{als} \al^*=\left(\frac{2(p+1)}{(n+2)-p(n-2)}\right)^{1/(p-1)},\quad
\alpha_*=\left(\frac{p+1}2\right)^{1/(p-1)}>1.\ee
{\clb Given $p>1$, we have $\al^*>\al_*$ if and only if $n>2$.}

\begin{theorem}\label{hh2} Let $u(r)$ be the solution of \equ{eu}. There exists a sequence of initial data $\al_0<\al_1<\al_2<\cdots$ with
$\al_0>\al^*$ and $\lim_{k\tin} \al_k=\infty$
such that, for $u_k(r)$ denoting the solution with $u(0)=\al_k$,

(i) $u_0(r)$ is the unique ground state of \equ{ble1};

(ii) $u_k(r)$, $k\ge 1$, is the unique bound state to \equ{ble1} that has precisely
$k$ (simple) zeros $z_1<\cdots<z_k$ in $(0, \infty)$. Between each pair of its two consecutive zeros $z_i$ and $z_{i+1}$, $i\in \{1, 2, \cdots, k-1\}$, $u_k(r)$ has exactly one critical point
$c_i$. Behind its last zero $z_k$, $u_k(r)$ has exactly one critical point $c_k$, after which $|u_k(r)|$ decreases with
\bl{lim0} \lim_{r\to \infty}\frac {u'_k(r)}{u_k(r)}=-1, \quad {\rm and} \,\,\,
\limsup_{r\to\infty} |u_k(r)|e^{(1-\epsilon)\, r}<\infty \,\,\, {\rm for \,\, any}\, \, \epsilon\in (0, 1). \ee
At each critical point, $|u_k|>\al_*$.

(iii) If $\al=1$, then $u(r)\equiv 1$. If $\al<\al_0$ and $\al\ne 1$, then $u(r)>0$ in $(0, \infty)$ with $\inf u>0$, and $u(r)$ oscillates about $u\equiv 1$.

(iv) If $\al\in (\al_k, \al_{k+1})$, $k\ge 0$, then $u(r)$ is a {\clb nodal (i.e., sign-changing)} solution with precisely $k+1$ zeros, and oscillates about $u\equiv 1$ or $u\equiv -1$ behind its last zero.
\end{theorem}

We mention that (i) and (iii) are well-known; see \cite{fra, k, mc, tao}. We include them in Theorem \ref{hh2} for a complete {\clb characterization} of \equ{eu}. {\clb Part (i) confirms the uniqueness of ground states of \equ{ble1}}. In 1972, Coffman \cite{coffman} proved {\clb (i) for the special case \equ{ble2} by analyzing the variation of $u$ with respect to $\al$.
McLeod and Serrin \cite{ms81, ms87} proved (i) for \equ{ble1} with $1<p\le n/(n-2)$ for $n=3,\, 4$; $1<p<8/n$ for $4<n<8$,  and Kwong \cite{k} extended it all ranges of $p$ and $n$ in \equ{ble1}. The proof was later simplified and further generalized in \cite{chl, cj, cef, frl, frls, kl, kz, mc, ous, yan}.} A different approach, {\clb known as} the separation technique, was {\clb introduced} by Peletier and Serrin \cite{ps}. This method has been influential in proving the uniqueness of ground states and {\clb extends} naturally to quasilinear elliptic equations \cite{cef, erbe, fls, pus, st}. In \cite{st}, Serrin and Tang  proved that \equ{ble} admits at most one ground state if, for some $b > 0$,

$f$ is continuous in $(0,\infty)$, with $f(u)\le 0$ in $(0,b]$ and $f(u) > 0$ for $u>b$;

$f\in C^1(b, \infty)$ with $uf'(u)/f(u)$ non-increasing on $(b, \infty)$.

\noindent
{\clb These conditions are satisfied by the model nonlinearity in \equ{ble1}} and its extension
\bl{fqp} f(u)=-|u|^{q-1}u+|u|^{p-1}u, \quad 0<q<p. \ee
By a slight modification, the uniqueness was also proved in \cite{st} when $f(u)$ has one extra zero in $(b, \infty)$, which covers
\bl{fge} f(u)=u(u-b)(c-u),\quad 0<b<c,\ee
arisen from population genetics \cite{arw, fisher}, and the important double-power nonlinearity
\bl{stf} f(u)=-u-\lm |u|^{q-1}u+\mu |u|^{p-1}u, \quad 1<p\ne q<\frac{n+2}{n-2}, \quad \lm, \,\mu>0,\ee
often arisen from the corresponding nonlinear Schr\"odinger equations \cite{cof, kit, len, st77}.

{\clb The uniqueness of bound states of $\Delta u+f(u)=0$ is, in general, exceedingly difficult to establish. This difficulty is reflected both in the rare successes achieved in \cite{awy, cls, cgy} and in the early counterexamples demonstrating non-uniqueness of ground states \cite{nn, ps}. Unraveling the intricate influence of $f$ on the solution structure is often far from straightforward. For instance, if $f$ is given by \equ{stf}, then the uniqueness of ground states is ensured by \cite{st}. In contrast, by replacing $\lm>0$  in \equ{stf} with a sufficient negative number, D\'avila, Del Pino and Guerra \cite{ddg} constructed three ground states in $\mathbb{R}^3$ for $\mu=1$, $1<q<\min\{3, p\}$, and $p_0<p<5$ for some $p_0\in (1, 5)$. Nevertheless, we conjecture that Theorems \ref{hh1} and \ref{hh2} for \equ{ble1} can be extended to $n=2$ and $p>1$ for which our proof does not apply. Moreover, we conjecture that if $f$ is given by one of \equ{fqp}-\equ{stf}, then $\Delta u+f(u)=0$ admits at most one $k$-node bound state with its absolute maximum located at the origin of $\mrn$, for any given $k\ge 1$ and $n\ge 2$.}

In Section 2, we {\clb investigate basic properties of radial solutions. To facilitate potential extensions of the present study, these properties are presented for $\Delta u+f(u)=0$ with a general nonlinearity $f$. In Section 3, we introduce two technical conditions sufficient for establishing Theorems \ref{hh1} and \ref{hh2}, while Sections 4 and 5 are devoted to verifying these conditions. Finally, the auxiliary functions and identities employed in our proof, along with their derivations, are provided in the appendix.}

 \section{Basic properties of radial solutions}

{\clb In this section, we recall some known properties of radial solutions and establish the positivity of several auxiliary functions. For conceptual clarity and to support potential extensions, we present these properties for the solution $u$ of
\bl{ble3} u''+\frac {n-1}ru'+f(u)=0,\quad
u(0)=\alpha>0,\quad u'(0)=0, \quad r>0,\ee
with a general nonlinearity $f(u)$. The dimension $n$
may be treated as a real variable ranging over $(1, \infty)$. We begin with the following assumptions on $f$:

(C1) $f\in C^1(\mathbb{R})$, $f(u) < 0$ for $u\in (0,1)$, $f(u) > 0$ for $u\in (1,\infty)$,
 $f(-u) = -f(u)$;

(C2) there is some $\al_*>1$ such that $F(\al_*)=\int_0^{\al_*} f(s)\, ds=0$;

(C3) $\zeta:=-f'(0)>0$;

(C4) $f'(1)>0$.

\noindent
From (C1) and (C2) we see that $f(0)=f(1)=0$ and $F$ has a unique positive zero at $\al_*>1$.
Specifying the unique positive zero of $f$ at $u=1$ is not essential;
the following proof would proceed the same way if ``$1$" is replaced by another positive constant.}

\subsection{{\clb The properties derived from energy estimates.}}
Given $f(1)=0$, the unique solution of {\clb \equ{ble3}} with $\al=1$ is $u\equiv 1$.
We first present a {\it monotone property} of $u(r)\not \equiv 1$ {\clb and estimate its $L^\infty(0, \infty)$ norm} by analyzing the {\it energy function}
\bl{ef} E(r):=\frac {u^{\prime 2}(r)}2+F(u(r)). \ee

\begin{prop}\label{basic1} Let $u(r)\not\equiv 1$ be the solution of {\clb \equ{ble3}. If (C1) and (C2) hold,} then

(i) $E(r)$ decreases strictly in $(0, \infty)$.

(ii) If $\,\widehat r>\overline r>0$ and $|u(\widehat r)|=|u(\overline r)|$, then $|u'( \widehat r)|<|u'(\overline r)|$.

(iii) $\Vert u\Vert_\infty=\sup\{|u(r)|; r\ge 0\}=\al$ if $\al>1$, and
$\al<\Vert u\Vert_\infty<\al_*$ if $\al<1$.
\end{prop}

\begin{proof} (i) {\clb Since $E'(r)=-(n-1)u^{\prime 2}(r)/r\le 0$,} $E(r)$ is non-increasing. If (i) were false, then there would be $\widetilde r_2>\widetilde r_1>0$ such that $E(\widetilde r_2)=E(\widetilde r_1)$, which forces
$E(r)\equiv E(\widetilde r_1)$ and $u'\equiv 0$  in $(\widetilde r_1, \widetilde r_2)$.
Consequently, $u(r)$ is a constant on $[\widetilde r_1, \widetilde r_2]$ and is extended
to $u(r)\equiv u(0)=\alpha$ in $[0, \infty)$ by the uniqueness theorem of ODE. It follows from {\clb \equ{ble3} and (C1)} that $f(\alpha)=0$ and $\al=1$, leading to a contradiction of $u\not\equiv 1$.

(ii) It follows from (i) immediately.

(iii) Note that $F(u)$ is an even function that increases in $u>1$ and decreases in $u\in (0, 1)$. For any $r>0$, we find from (i) that $F(|u(r)|)<E(0)=F(\alpha)$. If $\al>1$, then $|u(r)|<\al$ for all $r>0$ and so $\Vert u\Vert_\infty=\al$. Now we assume that $\al\in (0, 1)$.
By taking the limit in {\clb \equ{ble3}}, we find $u''(0)=-f(\al)/n>0$. Hence $u(r)$ increases initially and $\Vert u\Vert_\infty>\al$. Take the unique $\widetilde \al\in (1, \al_*)$ with
$F(\widetilde \al)=F(\alpha)<0$. Then $F(|u(r)|)<F(\widetilde\alpha)$ and $|u(r)|<\widetilde \al$
for all $r>0$. Consequently, $\Vert u\Vert_\infty\le \widetilde \al<\al_*$. \end{proof}

{\clb We next characterize critical points of nodal solutions, as well as the limiting behavior of ground states and bound states.}
Note that a nodal solution cannot have a double zero, since $u(r)=u'(r)=0$ at any $r>0$ would imply $u\equiv 0$ on $[0, \infty)$.

\begin{prop}\label{basic11} {\clb Assume that (C1)--(C2) hold and let $u$ be the solution of \equ{ble3}.}

(i) If $u$ is a ground state or a nodal solution, then $\alpha>\alpha_*$.

(ii) If $u$ is a ground state, then $u'(r)<0$ in $(0, \infty)$.

(iii) If $u$ is a nodal solution with exactly $k\ge 1$ zeros $z_1<z_2<\cdots <z_k$, then $u$ has $k$ critical points in $[0, z_k]$, labelled as $0=c_0<c_1<c_2<\cdots<c_{k-1}$, with $c_i\in (z_i, z_{i+1})$, $1\le i\le k-1$, and $u(0)>|u(c_1)|>\cdots>|u(c_{k-1})|>\alpha_*$. In addition, if
$u$ is a bound state, then there is a unique critical point $c_k>z_k$ with $|u(c_{k-1})|>|u(c_{k})|>\alpha_*$.

(iv) If $u$ is a ground state or a bound state and (C3) holds, then \\
$\dsp  \lim_{r\to \infty}\frac {u'(r)}{u(r)}=-{\clb \sqrt\zeta}$, and {\clb $\dsp \limsup_{r\to\infty} |u(r)|e^{\sqrt{\zeta-\epsilon}\, r}<\infty$ for any $\epsilon \in (0, \zeta)$}.
\end{prop}
\begin{proof} (i) If $u$ is a ground state, then $F(\alpha)>E_\infty:=\lim_{r\to \infty}E(r)=0$; if $u$ is a nodal solution with a zero $z_1$, then $F(\alpha)>E(z_1)>0$. Thus $F(\alpha)>0$ and $\alpha>\alpha_*$.

(ii) Let $u$ be a ground state. Then $\alpha>\alpha_*>1$ and $u''(0)=-f(\al)/n<0$. Hence
$u'(r)<0$ for sufficiently small $r>0$. {\clb Should $u$ cease to be decreasing at some point in $(0,\infty)$, then there would exist} $\widetilde c\in (0,\infty)$ such that $u'(r)<0$ in $(0, \widetilde c)$, $u'(\widetilde c)=0$, and $u''(\widetilde c)\ge 0$.
However, $E(\widetilde c)>E_\infty=0$ implies $F(u(\widetilde c))>0$, leading to $u(\widetilde c)>\alpha_*$, and in turn, $u''(\widetilde c)=-f(u(\widetilde c))<0$. This gives a desired contradiction.

(iii) Let $u$ be a nodal solution with zeros $z_1<z_2<\cdots <z_k$.
{\clb By slightly modifying the proof of (ii), we can show that} $u'<0$ in $(0, z_1)$. Let $c_1$ be a critical point in $(z_1, z_{2})$. Then $u(c_1)<0$ and $F(u(c_1))=E(c_1)>E(z_2)>0$. Consequently, $|u(c_1)|>\alpha_*$ and $u''(c_1)=-f(u(c_1))>0$. Hence $u(c_1)$ is a local minimum and {\clb $c_1$ is the only critical point on $[z_1, z_2]$}. As $F(|u(c_1)|)=E(c_1)<F(\al)$,
it follows that $|u(c_1)|<u(0)$. {\clb The remaining statements can be proved by the same reasoning}.

(iv) We follow the elegant proof of Lemma 5 in Peletier and Serrin \cite{ps}.
{\clb Write $c_0=0$ if $u$ is a ground state.} Then $\widehat u(r):=-{u'(r)}/{u(r)}>0$
for $r>c_k$. {\clb According to (C1) and (C3), we have $f(u)/u\to f'(0)=-\zeta<0$ as $u\to 0$. We may thus fix $\hat c>c_k$ such that $f(u(r))/u(r)\ge -1-\zeta$ for $r>\hat c$. Now it follows from \equ{ble3} that}
\bl{psl5} \widehat u'=\widehat u^2-\frac{n-1}r\, \widehat u+\frac{f(u)}{u}\ge {\clb \widehat u^2-\frac{n-1}r\, \widehat u-1-\zeta,\quad r>\hat c}. \ee
Should there be any {\clb $\hat r>\max \{2n, \hat c\}$ such that $\widehat u(\hat r)\ge 2+2\zeta$,} then
\ben \displaystyle 2\widehat u'\ge  \widehat u^2+ {\clb \left( \widehat u^2-\widehat u-2-2\zeta\right)}\ge  \widehat u^2,\quad r=\hat r. \een
These relations would remain valid for $r\ge \hat r$ until $\widehat u$ blows up at a finite later value of $r$, which is impossible.
{\clb Consequently, $\widehat u<2+2\zeta$ for all $r> \max \{2n, \hat c\}$.} Since both $u$ and $u'$ approach zero as $r\to\infty$, we may use ${\rm L'H\widehat opital's}$ rule {\clb and (C3)} to derive
\ben \lim_{r\to\infty}\left(\frac{u'}u\right)^2=\lim_{r\to\infty}\frac{u''}u=
\lim_{r\to\infty} {\clb\left(\frac{n-1}r\, \widehat u -\frac{f(u)}u\right)=\zeta}. \een
{\clb Hence $\widehat u(r)\to \sqrt \zeta$ as $r\to \infty$. For any given $\epsilon\in (0, \zeta)$, we have
$\widehat u(r)>\sqrt{\zeta-\epsilon}$} provided that $r$ is sufficiently large. By integration we can derive the inequality easily.
\end{proof}

\begin{prop}\label{basic12} {\clb Suppose (C1), (C2), and (C4) hold and $u(r)\not\equiv 1$ is the solution of \equ{ble3}.} If $E(\overline r)\le 0$ and $u(\overline r)>0$ at some $\overline r\ge 0$, then $u(r)\in (0, \al_*)$ and it oscillates about $u\equiv 1$ in $(\bar r, \infty)$: There are a sequence of critical points of $u$, $\widetilde c_1<\widetilde c_2<\cdots$, along which $\al_*>u(\widetilde c_1)>u(\widetilde c_3)>\cdots >1$, and $0<u(\widetilde c_2)<u(\widetilde c_4)<\cdots<1$.

If $E(\overline r)\le 0$ and $u(\overline r)<0$ at some $\overline r\ge 0$, then the same can be said for $-u$.
\end{prop}

\begin{proof} Assume that $E(\overline r)\le 0$ and $u(\overline r)>0$. Then $E(r)< 0$ for all $r\in (\overline r, \infty)$, over which $u$ must stay within $(0, \al_*)$ because $E(r)\ge 0$ whenever $u(r)=0$ or $u\ge \al_*$. To verify its oscillatory behavior, we follow the approach of Berestycki, Lions and Peletier \cite{blp} and apply the
transformation $\widetilde u(r)=r^{\frac{n-1}2}[u(r)-1]$. Then{\footnote{ {\clb Notably, \equ{blp2} simplifies significantly when $n=3$}, which partially explains why the classic approach of Coffman \cite{coffman} for $n=3$ was not easily extended to higher dimensions $n>3$.}}
\bl{blp2} \widetilde u''= -\,\left\{\frac{f(u)}{u-1}-\frac{(n-1)(n-3)}{4r^2}\right\}\widetilde u.\ee

Suppose to the contrary that there exists $\overline r_1>\overline r$ such that $u(r)\in (0, 1)$ for all $r> \overline r_1$. Since $f(u)<0$ for $u\in (0, 1)$, $u(r)$ could have at most one critical point, necessarily a local minimum, in $(\overline r_1, \infty)$. By {\clb increasing} $\overline r_1$ if necessary, we may assume that $u$ is monotone for $r\ge \overline r_1$ and converges to a limit $u^*\in [0, 1]$. {\clb Then $u', u'' \to 0$ as $r\tin$ and \equ{ble3} implies} $f(u^*)=0$. Since $E_\infty<0$, $u^*\ne 0$. Hence $u(r) \uparrow u^*=1$ over $(\overline r_1, \infty)$. Note that $\widetilde u(r)<0$ for $r>\overline r_1$ and {\clb $f(u)/(u-1)\to f'(1)>0$ as $u\to 1$ by (C1) and (C4)}. We may further increase $\overline r_1$ and apply \equ{blp2} to find that $\widetilde u''>0$ for $r>\overline r_1$.
Thus $\widetilde u'(r)\uparrow u_p^*$ as $r\to\infty$ for some $u_p^*\le \infty$. Now if $u_p^*>0$, including $u_p^*=\infty$, then $\widetilde u(r)\to \infty$, which contradicts $\widetilde u(r)<0$ in $(\overline r_1, \infty)$. If $u_p^* \le 0$, then $\widetilde u(r)\le \widetilde u(\overline r_1)<0$ for $r\ge \overline r_1$, {\clb and
$\widetilde u''(r)\ge -f'(1)\widetilde u(\overline r_1)/2>0$ for $r$ sufficiently large. This leads to $\widetilde u'(r) \uparrow\infty$ as $r\to\infty$, which contradicts our assumption that $u_p^* \le 0$}.

Now we see that $u(r)$ cannot stay entirely in $(0, 1)$ for all $r> \overline r_1$. By a similar argument, we can show that it cannot stay entirely in $(1, \al_*)$ either.
Hence $u(r)$ switches between $(0, 1)$ and $(1, \al_*)$ infinitely often by crossing the line $u=1$, with $|u'|>0$ at each crossing point. There is exactly one critical point between two consecutive crossings. By labeling these critical points appropriately and using the strict monotonicity of $E(r)$, we can make $\al_*>u(\widetilde c_1)>u(\widetilde c_3)>\cdots >1$ and  $0<u(\widetilde c_2)<u(\widetilde c_4)<\cdots<1$.

The last statement follows by reflection since $f(u)$ is an odd function.
\end{proof}

{\clb In contrast to the oscillation around $1$ or $-1$ just described, our next result shows that $u$ has at most finitely many zeros.}

\begin{prop}\label{basic13} {\clb Assume that (C1)--(C4) hold and let $u(r)$ be the solution of \equ{ble3}.

(i) If $\al\le \al_*$ and $\al\ne 1$, then $u$ is positive and oscillates around $1$ in $(0,\infty)$.

(ii) A nodal solution $u$ has only finitely many sign changes.

(iii) If $u$ is a ground state or a bound state, then $E(r)>0$ for all $r>0$.

(iv) A positive solution $u$ is either a ground state with $u\downarrow 0$ exponentially as $r\tin$, or an oscillatory function that oscillates about $1$ behind its last zero.
A nodal solution $u$ is either a bound state with $|u|\downarrow 0$ exponentially as $r\tin$, or an oscillatory function that oscillates about $1$ or $-1$ behind its last zero.}
\end{prop}

\begin{proof} {\clb (i) This follows immediately from Prop.\,\ref{basic12}, since $E(0)\le 0$. }

(ii) {\clb Since $E(r)$ decreases strictly in $(0,\infty)$, and $F(1)<0$ is the absolute minimum of $F$, it follows that ${u^{\prime 2}(r)}<2\left[E(0)-F(u(r))\right]$ and}
\bl{upb} |u'(r)|<D_{\alpha}:=\sqrt{2\left[F(\alpha)-F(1)\right]}, \qquad r>0.\ee
 Suppose for contradiction that $u(r)$ has infinitely many zeros. Between any pair of these zeros, there must be a critical point at which $|u|>\al_*$. {\clb We see from \equ{upb} that these zeros} do not accumulate at any finite limit. We may thus list these zeros in a sequence
$z_1<z_2<\cdots$ with $\lim_{i\tin} z_i=\infty$ and
\bl{eib} E(r)>E_{\infty}= \lim_{r\to\infty}E(r)=\lim_{i\to\infty}E(z_i)\ge 0,\qquad r>0.\ee
{\clb As in Prop.\,\ref{basic11} (iii), $u$ has a unique critical point, denoted by $c_i$, in $(z_i, z_{i+1})$, $i\ge 1$.}

Assume first that $E_{\infty}>0$. We adapt the approach of Dolbeault, Garc\'ia-Huidobro and Man\'asevich (\cite{dgm}, Prop. 3.2). A major step is to show that {\clb $c_i-c_{i-1}$ is bounded above by a constant depending on $\al$. Denote by $b_i, \overline b_i \in (c_{i-1}, c_{i})$ the unique numbers such that $b_i<\overline b_i$ and $|u(b_i)|=|u(\overline b_i)|=\alpha_*$. Then $F(u(r))\le 0$ and}
\bl{upl0} |u^{\prime }(r)|=\sqrt{2\left[E(r)-F(u)\right]}\ge \sqrt{2E(r)}>\sqrt{2E_\infty}>0,
\quad r\in [b_i, \overline b_i].  \ee
The mean value theorem gives
\bl{upl}  \frac{2\alpha_*}{\overline b_i-b_i}=\frac{|u(\overline b_i)-u(b_i)|}{\overline b_i-b_i}
>\sqrt{2E_\infty}\quad \Rightarrow \quad \overline b_i-b_i<\frac{\sqrt{2}\alpha_*}{\sqrt{E_\infty}} .\ee
 To estimate $b_i-c_{i-1}$ and $c_i-\overline b_i$, we choose a large integer $i_0$ such that
\bl{ci1b} c_{i-1}>z_{i-1}>\frac{2(n-1)D_{\alpha}}{f(\alpha_*)}>0,\qquad i\ge i_0.\ee
For each $r\in (c_{i-1}, b_i]\cup [\overline b_i, c_i]$, we have $|f(u(r))|\ge f(\alpha_*)$ and
\ben |u''(r)|=\left|f(u)+\frac {n-1}ru'\right|>f(\alpha_*)-\frac {n-1}{c_{i-1}}D_{\alpha} >\frac{f(\alpha_*)}2,\qquad i\ge i_0. \een
{\clb Applying the mean value theorem to $u'(r)$ on the interval
$[c_{i-1}, b_i]$ , we deduce that}
\bl{bicib} \frac {|u'(b_i)|}{b_i-c_{i-1}}=\frac {|u'(b_i)-u'(c_{i-1})|}{b_i-c_{i-1}}
>\frac{f(\alpha_*)}2\quad \Rightarrow \quad b_i-c_{i-1}<\frac {2D_\alpha}{f(\alpha_*)}, \,\, i\ge i_0. \ee
{\clb A similar argument over $[\overline b_i, c_i]$ leads to}
\bl{bicib2}  {\clb c_i-\overline b_i<\frac {2D_\alpha}{f(\alpha_*)},} \qquad i\ge i_0.\ee
Combining these estimates with \equ{upl}, we conclude that, for $i\ge i_0$,
\ben c_i-c_{i-1}=(c_i-\overline b_i)+(\overline b_i-b_i)+(b_i-c_{i-1})<
C_\alpha:=\frac {4D_\alpha}{f(\alpha_*)}+ \frac{\sqrt{2}\alpha_*}{\sqrt{E_\infty}}.\een
 It implies that $\overline b_i\le \overline b_{i_0}+2(i-i_0)C_\alpha$ for $i> i_0$. By integration,
\ben  \frac {E(0)}{n-1}&>& \frac {E(c_{i_0-1})-E_{\infty}}{n-1}=\int_{c_{i_0-1}}^\infty \frac {u^{\prime 2}(r)}r\,dr
\qquad \quad {\rm from} \,\,\equ{ep}
\\
&>& \sum_{i=i_0}^\infty \int_{b_{i}}^{\overline b_i} \frac {u^{\prime 2}(r)}r\,dr
>\sqrt{2E_\infty}\sum_{i=i_0}^\infty \frac 1 {\overline b_i}\int_{b_{i}}^{\overline b_i}  |u'(r)|\,dr  \quad {\rm from} \,\,\equ{upl0} \\
&=&  {2\alpha_*}\sqrt{2E_\infty}\sum_{i=i_0}^\infty \frac 1{\overline b_i}
\ge {2\alpha_*}\sqrt{2E_\infty}\sum_{j=0}^\infty \frac 1{\overline b_{i_0}+2jC_\alpha}.\een
This leads to a contradiction since the last series diverges to $\infty$.

{\clb Assume next that $E_{\infty}=0$. Then $E(r)\downarrow 0$ as $r\to \infty$ and
$\lim_{i\to\infty}F(u(c_i))=0$.} Following
Cort\'azar, Garc\'ia-Huidobro and Herreros (\cite{cgh}, Appendix A), {\clb we use
\bl{H} \widehat E'(r)=2(n-1)r^{2n-3}F(u), \qquad {\rm where}\quad \widehat E(r):=r^{2(n-1)}E(r),\ee
to estimate $\widehat E(z_i)-\widehat E(z_{i-1})$. Denote by $r_i,  r_{1i}\in (c_{i-1}, z_{i})$ the unique numbers such that
$|u(r_i)|=1$ and $|u(r_{1i})|=1/2$. Then $z_{i-1}<\overline b_{i-1}<c_{i-1}<b_i<r_i<r_{1i}<z_i$ for $i>1$.} Within the interval $(z_{i-1}, z_{i})$, {\clb we have $F(u(r))>0$ for $r\in (\overline b_{i-1}, b_i)$ and $F(u(r))\le 0$ otherwise. Integrating \equ{H} gives}
\ben \frac{\widehat E(z_i)-\widehat E(z_{i-1})}{ 2(n-1)}
< \int_{\overline b_{i-1}}^{b_i}r^{2n-3}F(u(r))\, dr+\int_{r_i}^{r_{1i}}r^{2n-3}F(u(r))\, dr.
\een
The mean value theorem and \equ{upb} imply that $r_{1i}-r_i>1/(2D_\alpha)$. {\clb Now we let $i$ be sufficiently large.} Then \equ{bicib} and \equ{bicib2} imply that
\ben b_i-\overline b_{i-1}=(b_i-c_{i-1})+(c_{i-1}-\overline b_{i-1})< {4D_\alpha}/{f(\alpha_*)}.\een
{\clb Note that $F(u(r))$ is maximized at $c_{i-1}\in (\overline b_{i-1}, b_i)$,} $\max_{s\in[1/2, 1]}F(s)=F(1/2)<0$, and $\lim_{i\to\infty} F(u(c_i))=0$. We continue to estimate
\ben \frac{\widehat E(z_i)-\widehat E(z_{i-1})}{ 2(n-1)}
&<& {r_i}^{2n-3}\left[(b_i-\overline b_{i-1})F(u(c_{i-1}))+(r_{1i}-r_i)F(1/2)\right] \\
&<& {r_i}^{2n-3}\left[ \frac {4D_\alpha}{f(\alpha_*)}F(u(c_{i-1}))+\frac{F(1/2)}{2D_\alpha}\right]
< {r_i}^{2n-3}\cdot \frac{F(1/2)}{4D_\alpha}.
\een
This leads to $\lim_{i\to\infty}\widehat E(z_i)=-\infty$, which contradicts $\widehat E(r)>0$ in $(0,\infty)$.

{\clb (iii) This follows immediately from (ii), Prop.\,\ref{basic1} (i), and Prop.\,\ref{basic11} (iv).}

(iv) {\clb Assume first that $u$ is a nodal solution. According to (ii), we may let $z_k>0$ be the largest zero of a nodal solution $u$.}
If $E(r)\le 0$ at some $r>z_k$, then Prop.\,\ref{basic12} implies that $u$ oscillates about $1$ or $-1$ in $(z_k, \infty)$. If $E(r)>0$ for all $r> z_k$, then $u(r)$ has no critical points at which $|u|\le \al_*$. Thus $u$ has at most one critical point {\clb in $(z_k, \infty)$}, and is monotone for sufficiently large $r$. Clearly, $u_\infty=\lim_{r\tin} u(r)$ is finite in view of Prop.\,\ref{basic1}, and $f(u_\infty)=0$. Since $E(r)>0$ for all $r>0$, $u_\infty$ is neither $1$ nor $-1$. Thus $u_\infty=0$ and Prop.\,\ref{basic11} (iv) implies that $u$ is a bound state with $|u|\downarrow 0$ exponentially as $r\tin$. {\clb By replacing $z_k$ with $0$ and repeating the proof, we obtain the conclusion for a positive solution $u$}. \end{proof}

\subsection{Phases and labels}
 Assume {\clb that (C1)--(C2) hold and that $u$ is a bound state
with exactly $k$ zeros $z_1<\cdots <z_k$}. If follows from Prop.\,\ref{basic11}\,(iii) that $u$
has $k+1$ critical points $0=c_0<c_1<\cdots<c_k$. We decompose $(0, \infty)$ into $k$ phases and a semi-tail phase separated by these critical points:
\bl{ph1} (0,\infty)=\underbrace{(0,c_1]}_{{\rm Phase}\, 1}\cup\overbrace{(c_1, c_2]}^{{\rm Phase}\, 2}\cup\cdots\cup\underbrace{(c_{i-1}, c_i]}_{{\rm Phase}\, i}\cup\cdots\cup\overbrace{(c_{k-1}, c_k]}^{{\rm Phase}\, k}\cup \underbrace{(c_{k}, \infty)}_{\rm Semi-tail\,\, phase}. \ee
{\clb If $u$ is a nodal solution with exactly $k$ zeros and oscillates in $(z_k, \infty)$}, then
we make a similar decomposition that differs from \equ{ph1} only in the last two intervals:
\bl{ph2} (0,\infty)=\underbrace{(0,c_1]}_{\rm Phase\, 1}\cup\overbrace{(c_1, c_2]}^{\rm Phase\, 2}\cup\cdots\cup\underbrace{(c_{i-1}, c_i]}_{{\rm Phase}\, i}\cup\cdots\cup\overbrace{(c_{k-1}, z_k]}^{{\rm Semi-phase}\, k}\cup \underbrace{(z_{k}, \infty)}_{\rm Tail\,\, phase}. \ee
 The function $u$ decreases in Phase $i$ if $i$ is odd, and increases in Phase $i$ if $i$ is even.

Within a single phase $(c_{i-1}, c_i]$, $u$, $f(u)$, and $F(u)$ all change signs. We label the points where $F(u)=0$ or $f(u)=0$ besides $r=z_i$ as follows:
\begin{equation}\label{biri}
\begin{cases}
|u(b_i)|=|u(\overline b_i)|=\alpha_*, \quad |u(r_i)|=|u(\overline r_i)|=1, \\
c_{i-1}<b_i<r_i<z_i<\overline r_i<\overline b_i<c_i.
\end{cases}
\end{equation}

We use these numbers to decompose Phase $i$ into subsets in which $u$, $f(u)$, and $F(u)$ remain the same sign. For an odd number $i$,
\bl{oi} (c_{i-1}, c_i]&=&\overbrace{\underbrace{(c_{i-1},b_i)}_{f,\, F>0}\cup\underbrace{\{b_i\}}_{ F=0}\cup\underbrace{(b_{i},r_i)}_{f>0,\, F<0}\cup\underbrace{\{r_i\}}_{f=0}\cup\underbrace{(r_i,z_i)}_{f,\, F<0}}^{u>0}\cup\underbrace{\{z_i\}}_{u, f, F=0}\nonumber \\
&& \cup\overbrace{\underbrace{(z_i,\overline r_i)}_{f>0,\, F<0}\cup\underbrace{\{\overline r_i\}}_{ f=0}\cup\underbrace{(\overline r_i, \overline b_i)}_{f,\, F<0}\cup\underbrace{\{\overline b_i\}}_{F=0}\cup\underbrace{(\overline b_i, c_i]}_{f<0,\, F>0}}^{u<0},\ee
and for an even number $i$,
\bl{ei} (c_{i-1}, c_i]&=&\overbrace{\underbrace{(c_{i-1},b_i)}_{f<0,\, F>0}\cup\underbrace{\{b_i\}}_{ F=0}\cup\underbrace{(b_{i},r_i)}_{f,\, F<0}\cup\underbrace{\{r_i\}}_{f=0}\cup\underbrace{(r_i,z_i)}_{f>0,\, F<0}}^{u<0}\cup\underbrace{\{z_i\}}_{u, f, F=0}\nonumber \\
&& \cup\overbrace{\underbrace{(z_i,\overline r_i)}_{f,\,F<0}\cup\underbrace{\{\overline r_i\}}_{ f=0}\cup\underbrace{(\overline r_i, \overline b_i)}_{f>0,\, F<0}\cup\underbrace{\{\overline b_i\}}_{F=0}\cup\underbrace{(\overline b_i, c_i]}_{f,\, F>0}}^{u>0}.\ee

\subsection{Positivity of energy-type functions} The Pohozaev function associated with equation \equ{ble3} is defined by
\bl{P} P(r):=2r^nE(r)+(n-2)r^{n-1}uu'=r^n\left[u^{\prime 2}+2F(u)\right]+(n-2)r^{n-1}uu'.\ee
{\clb We now establish the positivity of $P(r)$ and two related functions}
\bl{P1} P_1(r):=r^n\left[u^{\prime 2}+uf(u)\right]+(n-2)r^{n-1}uu',\ee
and
\bl{P2} P_2(r):= r^n\left[u^{\prime 2}+\frac {n-2}n uf(u)\right]+(n-2)r^{n-1}uu'.\ee
{\clb Beyond (C1)--(C4), we introduce the following further assumptions on $f$:

 (C5) There is $\al^*>0$ such that $h(u):=2nF(u)-(n-2)uf(u)<0$ for $u\in (0, \al^*)$, and
 $h(u)>0$ for $u>\al^*$;

 (C6) $h_1(u):=uf(u)-2F(u)\ge 0$;

 (C7) there is $\widetilde \al>0$ such that $h_2(u):=(n+2)f(u)-(n-2)uf'(u)<0$ for $u\in (0, \widetilde \al)$, and $h_2(u)>0$ for $u>\widetilde \al$;
 
{\clr (C8) $f(u)/(uf'(u))$ is increasing for $u >1$.}
 
Under the basic assumptions (C1) and (C2), we have $h(\al_*)<0$. Thus (C5) implies
 the important relation $\al^*>\al_*$. Note also that $f'(0) < 0$ implies $h_2(u)<0$ for small $u >0$, indicating that conditions (C3) and (C7) are consistent.}  {\clr If (C1) holds and $f'(u)> 0$ in $(1, \infty)$, then (C8) is equivalent to the Serrin-Tang condition that $uf'(u)/f(u)$ is non-increasing in $(1, \infty$).}

\begin{prop}\label{basic2}  {\clb Assume that (C1)--(C3) hold and let $u(r)$ be the solution of \equ{ble3}. Then the following statements are valid for $r\in (0, z_k]$ if $u$ is a nodal solution with the largest zero $z_k>0$, or for all $r>0$ if $u$ is a ground state or a bound state.

(i) Assume that (C5) holds.  Then $P(r)>0$; in addition, if $u$ is a ground state or a nodal solution, then $u(0)>\al^*$.

(ii) If (C5) and (C6) hold, then $P_1(r)>0$, and $\omega'(r)>0$ as long as $u(r)\ne 0$, where
$\omega(r):=-ru'(r)/u(r)$.

(iii) If (C7) holds, then $P_2(r)>0$ and $-[P(r)/r^n]'>0$.}
\end{prop}
\begin{proof} {\clb We first note that if $u$ is a nodal solution with zeros $z_1<\cdots<z_k$, then $P(z_i)=P_1(z_i)=P_2(z_i)=z_i^nu'^2(z_i)>0$, $1\le i\le k$. Moreover, if $u$ is a ground state or a bound state, then $|u(r)|$ decays exponentially, while $P(r)$, $P_1(r)$, and $P_2(r)$ all approach zero as $r\to \infty$.}

(i) {\clb Assume that (C5) holds. We first show that if $u$ is a ground state or a nodal solution, then $u(0)>\al^*$.} Indeed, if $u$ is a ground state {\clb but} $u(0)\le \al^*$, then $u(r)<\al^*$ for all $r>0$ by Prop.\,\ref{basic11} (ii). Consequently,
{\clb $h(u(r))<0$ and identity \equ{Pr} in the appendix implies that
$P'(r)=r^{n-1}h(u)<0$ for all $r>0$}, which contradicts $P(0)=0$ and $\lim_{r\to \infty}P(r)=0$.
Similarly, if $u$ is a nodal solution {\clb but} $u(0)\le \al^*$, then
$P'(r)<0$ in $(0, z_1)$, which contradicts $P(0)=0$ and $P(z_1)>0$.

Corresponding to a ground state that starts from $u(0)>\al^*$, $P(r)$ increases from $P(0)=0$ until reaching the absolute maximum value where $u=\al^*$, and then decays to zero as $r\tin$. Hence $P(r)>0$ in $(0, \infty)$. Similarly, if $u$ is a nodal solution, then $P(r)$ increases initially and then decreases to $P(z_1)>0$, so $P(r)>0$ in $(0, z_1]$. On $[z_1, c_1]$, $uu'\ge 0$ so \equ{P} gives $P(r)\ge 2r^nE(r)>0$; {\clb see Prop.\,\ref{basic13} (iii)}. In $(c_1, z_2)$, $P(r)$ either decreases when $|u(c_1)|\le \al^*$, or increases first and then decreases when $|u(c_1)|> \al^*$; in both cases $P(r)>0$ since it is positive at the two ends $c_1$ and $z_2$. On $[z_2, c_2]$, the discussion is the same as on $[z_1, c_1]$ because $uu'\ge 0$ again.
The proof that $P(r)>0$ in the remaining intervals can be completed using nearly identical arguments.

(ii) {\clb If (C5)-(C6) hold, then $P_1(r)=P(r)+r^nh_1(u)\ge P(r)>0$. It follows from
\equ{omegap} that $\omega'(r)={P_1(r)}/(r^{n-1}u^{2})>0$ as long as $u(r)\ne 0$.}

(iii) {\clb Assume that (C7) holds. If $u$ is a ground state but $u(0)\le \widetilde \al$, then the identity $nP_2'(r)=-\,{r^nu'}h_2(u)$ in \equ{mmyp2r} implies that $P_2'(r)<0$ in $(0, \infty)$, which contradicts $P_2(0)=0$ and $\lim_{r\to \infty}P_2(r)=0$. Hence $u(0)>\widetilde \al$, and $P_2(r)$ increases from $P_2(0)=0$ until reaching its unique maximum where $u=\widetilde \al$, and then decays to zero as $r\to\infty$, which confirms $P_2>0$ in $(0,\infty)$. Similarly, if $u$ is a nodal solution, then $u(0)>\widetilde \al$ and $P_2(r)>0$ in $(0, z_1]$. To continue, we note that $nP_2(c_i)=(n-2)c_i^n u(c_i)f(u(c_i))>0$ at each critical point $c_i$.} In $(z_1, c_1)$, $uu'>0$; thus $P_2$ increases
in this interval when {\clb $|u(c_1)|\le \widetilde \al$}, or increases first and then decreases to $P_2(c_1)>0$ otherwise. In $(c_1, z_2)$, $P_2$ decreases from $P_2(c_1)>0$ to $P_2(z_2)>0$
when {\clb $|u(c_1)|\le \widetilde \al$}, or increases first and then decreases to
{\clb $P_2(z_2)>0$ otherwise}. Thus $P_2>0$ on $[z_1, z_2]$. Repeating these arguments shows that $P_2(r)>0$ in the remaining intervals. {\clb Finally,
it follows directly from \equ{prn} that $-[P(r)/r^n]'= n P_{2}(r)/r^{n+1} >0$.}  \end{proof}

\subsection{The concavity}

\begin{prop}\label{basic21} Assume that (C1)--(C6) {\clr and (C8) hold.} Then a ground state $u$ of \equ{ble3} changes concavity at a unique point of inflection in $(0, \infty)$. For a nodal solution $u$ of \equ{ble3}, there is a unique point of inflection between each critical point $c_{i-1}$ and its nearest zero $z_i>c_{i-1}$.
\end{prop}

\begin{proof} From (C1), (C4), and {\clr (C8)}, it follows that $f(u)/(uf'(u))$ increases with $|u|$ and $f'(u)>0$ when $|u|> 1$.

Let $u$ be a ground state. {\clb Then $u>1$ on $[0, r_1)$,  $u(r_1)=1$, and $u<1$ for $r>r_1$. From \equ{ble3} we find that $u''(0)=-f(\al)/n<0$ and $u''(r)\ge -(n-1)u'/r>0$ for $r\ge r_1$. Thus $u$ changes concavity in $(0, r_1)$.} Define $r_1^*\in (0, r_1)$ by $u''(r_1^*)=0$ and $u''(r)<0$ in $(0, r_1^*)$.
At $r=r_1^*$, {\clb we have}
 \bl{r1s1} -\frac {n-1}{r_1^*}u'=f(u),\quad u'''(r_1^*)=\frac{n-1}{r_1^{*2}}u'-f'(u)u'\ge 0\,\,\, \Rightarrow\,\,\,  \frac{n-1}{r_1^{*2}}\le f'(u),\ee
where the last inequality {\clb holds since} $u'(r)<0$ for all $r>0$. Consequently,
 \bl{r1s2} \omega(r_1^*)=\, -\frac{r_1^*u'}{u}=-\frac {n-1}{r_1^*}u'\cdot \frac 1 u \cdot \frac {r_1^{*2}}{n-1}\ge \frac{f(u)}{uf'(u)}.\ee
 Because $f(u)/(uf'(u))$ increases in $u$ when $u>1$, it decreases in $r$ for $r\in (0, r_1)$. As Prop.\,\ref{basic2} (ii) indicates that $\omega'(r)>0$ in $(0, \infty)$,
we obtain
 \bl{r1s} \omega(r)>{\clb \varpi(r):=} \frac{f(u(r))}{u(r)f'(u(r))}, \qquad r\in (r_1^*, r_1).\ee
{\clb This implies, in particular, that if $u''(r) = 0$ for some $r \in (r_1^*, r_1)$, then
$u''' (r)>0$. If $u'''(r_1^*)>0$, then $u''$ changes sign from negative to positive at $r_1^*$, and stays positive in $(r_1^*, r_1)$. Since $u''(r)>0$ for $r\ge r_1$, we reach
our conclusion in this case.}

{\clb We need} to prove that $u''(r)>0$ in $(r_1^*, r_1)$ {\clb for the degenerate case $u'''(r_1^*)=0$. By the statement following \equ{r1s}, we deduce that $u''(r)\not\equiv 0$ on $[r_1^*, r_1^*+\epsilon]$ for any $\epsilon>0$, and that $u''$ cannot have a sequence of zeros accumulating at $r_1^*$. It remains to rule out the final scenario that}
\ben u''(r_1^*)=u'''(r_1^*)=0,\,\,\,{\rm and}\,\,\, u''<0\,\,\, {\rm in}\,\,\, (r_1^*, r_1^*+\epsilon]\,\,\, {\rm for \,\, a \,\, small}\,\,\,  \epsilon>0.\een
Should this occur, $u''$ would have a zero in $(r_1^*, r_1)$ because
$u''(r_1)>0$. Between $r_1^*$ and this zero, $u''$ has a negative minimum value. Hence there exists $\widetilde r_1^*\in (r_1^*, r_1)$ such that $u''(\widetilde r_1^*)<0$ and $u'''(\widetilde r_1^*)=0$. From equation \equ{ble3} we find that
\ben -\frac {n-1}{\widetilde r_1^*}u'-f(u)=u''(\widetilde r_1^*)<0, \qquad \frac{n-1}{\widetilde r_1^{*2}}u'-f'(u)u'< u'''(\widetilde r_1^*)=0. \een
Since $u'(\widetilde r_1^*)<0$, we obtain $-(n-1)u'/\widetilde r_1^*<f(u)$, \,$(n-1)/\widetilde r_1^{*2}> f'(u)$, and
\ben \omega(\widetilde r_1^*)=-\frac {n-1}{\widetilde r_1^*}u'\cdot \frac 1 u \cdot \frac {\widetilde r_1^{*2}}{n-1}< {\clb \varpi (\widetilde r_1^*)} \een
that contradicts \equ{r1s}. This completes the proof for a ground state.

{\clb The proof above also shows that a nodal solution $u(r)$ possesses a unique inflection point in $(0, z_1)$.} The discussion in $(c_{i-1}, z_i)$ for $i>1$ is very similar, especially when $i$ is odd.
We present a brief proof for an even number $i$, paying particular attentions to the signs of relevant quantities. The phase decomposition \equ{ei} indicates that, when $i$ is even, $u(r)<0$ and $u'(r)>0$ in $(c_{i-1}, z_i)$, and
\bl{cizif} f(u(r))<0, \,\, f'(u(r))>0,\,\, r\in [c_{i-1}, r_i); \quad f(u)>0,\,\, r\in (r_i, z_i).\ee
 Hence $u''(c_{i-1})=-f(u(c_{i-1}))>0$, {\clb $u''(r)<-f(u)< 0$ in $(r_i, z_i)$}, and $u''$ has a zero in $(c_{i-1}, r_i)$. Let $r^*_i\in (c_{i-1}, r_i)$ be such that $u''>0$ in $(c_{i-1}, r^*_i)$ and $u''(r^*_i)=0$. {\clb Then
\ben -\frac {n-1}{r_i^*}u'=f(u),\quad u'''(r_i^*)=\frac{n-1}{r_i^{*2}}u'-f'(u)u'\le 0\,\,\, \Rightarrow\,\,\,  \frac{n-1}{r_i^{*2}}\le f'(u).\een
Consequently, $\omega(r_i^*)\ge \varpi(r_i^*)$, and $\omega(r)>\varpi(r)$
for $r\in (r_i^*, r_i)$. The remaining part of the proof that $u''(r)<0$ in $(r_i^*, r_i)$
is analogous to the ground state case. We therefore provide only a detailed discussion to exclude the degenerate case}
\ben u''(r_i^*)=u'''(r_i^*)=0,\,\,\,{\rm and}\,\,\, u''>0\,\,\, {\rm in}\,\,\, (r_i^*, r_i^*+\epsilon]\,\,\, {\rm for \,\, a \,\, small}\,\,\,  \epsilon>0.\een
 Should it occur, $u''(r_i)<0$ would force $u''$ to take a positive maximum value, say, at $\widetilde r_i^*\in (r_i^*, r_i)$. Hence $u''(\widetilde r_i^*)>0$, $u'''(\widetilde r_i^*)=0$,
\ben -\frac {n-1}{\widetilde r_i^*}u'-f(u)=u''(\widetilde r_i^*)>0,\quad{\rm and}\quad
\frac{n-1}{\widetilde r_i^{*2}}u'-f'(u)u'> u'''(\widetilde r_i^*)=0. \een
Since $u(\widetilde r_i^*)<0$ and $u'(\widetilde r_i^*)>0$, by using \equ{cizif} we obtain
$(n-1)u'/\widetilde r_i^*<-f(u)$, $(n-1)/\widetilde r_i^{*2}> f'(u)$, and
\ben \omega(\widetilde r_i^*)=\frac {n-1}{\widetilde r_i^*}u'\cdot \frac 1 {-u} \cdot \frac {\widetilde r_i^{*2}}{n-1}< \varpi(\widetilde r_i^*)\een
that again leads to a contradiction. The proof is completed. \end{proof}

{\clb We note that, for a  nodal solution $u$ with zeros $z_1<\cdots<z_k$, its
convexity in the intervals
$(z_i, c_i)$ for $1\le i\le k-1$ and $(z_k, \infty)$ is not addressed in this proposition.}

 \section{{\clb Outline of the Proofs of Theorems \ref{hh1} and \ref{hh2}}}

{\clb Our proofs of Theorems \ref{hh1} and \ref{hh2} are based on the shooting method,
a technique introduced by Kolodner \cite{kol}, by analyzing the behavior of the variation
\ben v(r, \alpha)={\p u(r, \alpha)}/{\p\alpha}  \een
corresponding to the solution $u=u(r, \alpha)$ of \equ{ble3}.
Given that $u$ is bounded for all $r \ge 0$, standard
results from ODE theory (see Chapter 7 of \cite{br} or Chapter 2 of \cite{cl}) ensure that the variation $v(r)$ is well-defined for all $r\ge 0$ and satisfies}
\bl{ev} v''+\frac {n-1}rv'+f'(u)v=0, \quad v(0)=1,\quad v'(0)=0. \ee
{\clb By the uniqueness theorem for ODEs, $v$ cannot have a double zero; that is,
whenever $v$ vanishes, $v'\ne 0$ and $v$ changes sign near its zero.}

\begin{lemma}\label{tail0} Let {\clb $u(r, \al)$ be the solution of \equ{eu} and
$v(r, \al)={\p u(r, \alpha)}/{\p\alpha}$. Suppose that, for any $\al>\al^*$, the variation $v(r,\al)$ satisfies the following conditions:

(a) If $u(r)=u(r, \al)$} has exactly $k\ge 1$ zeros $z_1<\cdots<z_k$,
then $v(r)=v(r, \al)$ has precisely $k$ zeros $\tau_1<\cdots<\tau_k$ on $[0, z_k]$, where $\tau_1\in (0, z_1)$, and $\tau_i\in (z_{i-1}, z_i)$ for $2\le i\le k$. Moreover, $u'v'>0$ whenever $v=0$.

(b) {\clb If $u(r)$ is a ground state, then $v(r)$ has exactly one zero $\tau_1\in (0, \infty)$. If $u(r)$ is a $k$-node bound state,} then $v(r)$ has precisely one additional zero $\tau_{k+1}$ in $[z_k, \infty)$ and $\tau_{k+1}>c_k>z_k$, where $c_k$ is the largest critical point of $u$. {\clb In either case, $v$} is strictly monotone for sufficiently large $r$ and $\lim_{r\to\infty} |v(r)|=\infty$.

{\clb Then Theorems \ref{hh1} and \ref{hh2} hold}.
\end{lemma}

\begin{proof} {\clb We begin with the proof showing that Theorem \ref{hh2} follows from (a) and (b), and conclude with a simple verification of Theorem \ref{hh1}. The model nonlinearity $f$ in \equ{eu} clearly satisfies conditions (C1)-(C5) stated in Section 2, with $\zeta=-f'(0)=1$ in (C3), while the constants
 $\al_*>1$ in (C2) and $\al^*>\al_*$ in (C5) are given by \equ{als} in terms of $n$ and $p$.
Propositions \ref{basic1}--\ref{basic2}\,(i) will be invoked in the proof below.}

Denote by $\mathcal{N}(\al)$ the number of zeros of $u(r, \al)$ over $(0,\infty)$.
 {\clb Then $\mathcal{N}(\al)=0$ for all $\al\le \al^*$ by Prop.\,\ref{basic2} (i), and
$\mathcal{N}(\al)$ is finite for each $\al>\al^*$ by Prop.\,\ref{basic13} (ii).
We now employ condition (a) to show that $\mathcal{N}(\al)$ is non-decreasing.  For a given $\al>\al^*$, let $z(\al)$ be a zero of a nodal solution $u=u(r,\al)$. By positioning the zeros of $v$ relative to those of $u$ as described by condition (a), we find that $u'v>0$ at $z(\al)$.} Differentiating $u(z(\al), \al)=0$ with respect to $\al$ gives
\bl{zpal} u'(z(\al), \al)z'(\al)+v(z(\al), \al)=0\quad \Rightarrow \quad z'(\al)=-\frac{v(z(\al), \al)}{u'(z(\al), \al)}<0.\ee
Let $u(r,\overline\al)$ be a nodal solution with exactly $j$ zeros $z_1(\overline\al)<\cdots<z_j(\overline\al)$ and assume that $\al>\overline\al$. Then $u(r,\al)$ has at least $j$ zeros with $z_1(\al)<z_1(\overline\al), \cdots, z_j(\al)<z_j(\overline\al)$.
{\clb Clearly, $\mathcal{N}(\al)\ge \mathcal{N}(\overline\al)$, which confirms the monotonicity property of $\mathcal{N}(\al)$.

Next we establish a pausing property of $\mathcal{N}(\al)$: If $\overline u(r)=u(r,\overline\al)$, $\overline\al>\al^*$, is an oscillatory solution}, then $\mathcal{N}(\al)$ remains a constant in a neighborhood of $\overline\al$. To show this, {\clb we select $\overline r$ to be a critical point of $\overline u$ at which $|\overline u|<1$. Then $\overline r$ lies behind the last zero of
$\overline u$ if it is a nodal solution}, and the energy function is negative at $\overline r$. If $\al$ is sufficiently close to $\overline \al$, then {\clb $u$ has the same number of zeros as $\overline u$ does} in $(0, \overline r]$, and the energy $E(r)$ for $u(r,\al)$ is also negative at $\overline r$ by continuity. Prop.\,\ref{basic1} (i) implies that $E(r)<0$ in $[\overline r, \infty)$. Thus
$u$ has no zeros in $(\overline r, \infty)$ and $\mathcal{N}(\al)=\mathcal{N}(\overline\al)$.

{\clb Moreover, we employ condition (b) to establish a jumping property of $\mathcal{N}(\al)$: If $u(r,\overline\al)$ is not an oscillatory solution, i.e., $\overline u$ is either a ground state or a bound state according to Prop.\,\ref{basic13} (iii), then
 $\liminf_{\al\to \overline\al^+}\mathcal{N}(\al)>\mathcal{N}(\overline\al)$.} We prove the jumping property by modifying McLeod's approach in \cite{mc}. Assume first that $j=\mathcal{N}(\overline\al)$ is odd. Then $\overline u$ is eventually negative and $\overline u \uparrow 0$ as $r\tin$. Denote by $\overline v$ the solution of \equ{ev} corresponding to $\overline u$. {\clb In view of condition (b),} $\overline v,\, \overline v'>0$ for $r$ sufficiently large and $\lim_{r\to\infty} \overline v(r)=\infty$. Fix $\overline r$ so large that $\overline v,\, \overline v'>0$ at $\overline r$, and for all $r\ge \overline r$, {\clb $f'(\overline u(r))<0$}.
 Clearly, $\overline r$ is behind the last critical point of $\overline u$, so $\overline u \uparrow 0$ in $(\overline r, \infty)$. If $\al>\overline \al$ and $\al$ is sufficiently close to $\overline \al$, then $w(r):=u(r)-\overline u(r)$ and $w'(r)$ are both positive at
$\overline r$, and $u(\overline r)<0$. Of course, $u$ has exactly $j$ zeros in $(0, \overline r)$ {\clb and $\mathcal{N}(\al)\ge j$}. Suppose for contradiction that $u$ has no zeros in $(\overline r, \infty)$. Then $u<0$ in $(\overline r, \infty)$.
If $u$ is also a bound state, then $w\to 0$ as $r\tin$; if $u$ is oscillatory, then $u$ oscillates about $-1$ in $(\overline r, \infty)$ to force $w=0$ at a subsequent point. In either case, if follows from $w(\overline r)>0$ and $w'(\overline r)>0$ that
$w$ would take a positive maximum at some $\widehat r>\overline r$, where
{\clb $f'(u(\widehat r))<0$ and $f'(\overline u(\widehat r))<0$}. Thus $w(\widehat r)>0$, $w'(\widehat r)=0$, and $w''(\widehat r)\le 0$. However, by using \equ{eu} we find that, for some $\widetilde u(r)$ between $\overline u(r)$ and $u(r)$,
\ben w''+ \frac{n-1}{r}w' + f'(\widetilde u(r))w = 0 \quad \Rightarrow \quad
w''(\widehat r)=-f'(\widetilde u(\widehat r))w(\widehat r)>0.
 \een
It gives a contradiction and confirms {\clb the jumping property} when $j=\mathcal{N}(\overline\al)$ is odd. If $j$ is even {\clb(including zero)}, then $\overline u>0$ and $v<0$ ultimately with $\lim_{r\to\infty} v(r)=-\infty$. The proof can be completed by reflecting the discussions above.

{\clb Now we apply the non-decreasing, pausing, and jumping property of $\mathcal{N}(\al)$ to prove Theorem \ref{hh2}. Let $\al_0>\al^*$ be such that $u_0=u(r, \al_0)$ is a ground state;
the existence of $\al_0$ is ensured} by the existence theorem of Berestycki and Lions \cite{bl1, bl2}. {\clb Then $\mathcal{N}(\al)\ge 1$ for $\al>\al_0$, and $\mathcal{N}(\al)=0$ for
$\al<\al_0$. Should there exist $\overline \al<\al_0$ such that $u(r, \overline \al)$ is another ground state, the jumping property would yield $\mathcal{N}(\al_0)>0$ and a contradiction. Therefore, $u_0$ is the only ground state, which verifies part (i) of Theorem \ref{hh2}.
If $\al<\al_0$ and $\al\ne 1$, then $u$ is positive in $(0, \infty)$ and oscillates
about $u\equiv 1$; furthermore, we see from Prop.\,\ref{basic12} that $\inf u>0$ and confirm
part (iii) of Theorem \ref{hh2}.

Similarly, let $\al_k>\al_0$, $k\ge 1$, be such that $u_k=u(r, \al_k)$ is a $k$-node bound state; the existence of $\al_k$ is ensured} by the existence of solutions of \equ{ble1} with any prescribed number of zeros (
see Jones and K\"upper \cite{jk}, or McLeod, Troy and Weissler \cite{mtw}).
{\clb Then $\mathcal{N}(\al)\ge k+1$ for $\al>\al_k$, and $\mathcal{N}(\al)\le k$ for
$\al<\al_k$, implying that $u_k$ is the unique $k$-node bound state and}
$\al_0<\al_1<\al_2<\cdots$.  Because $\mathcal{N}(\al)$ is finite for any $\al>0$, it holds that $\lim_{k\tin} \al_k=\infty$.

{\clb We now verify the remaining statements of Theorem \ref{hh2} (ii) and (iv):} For the $k$-node bound state $u_k$, the existence of a unique critical point between each pair of its two consecutive zeros and a unique critical point behind its last zero is asserted in Prop.\,\ref{basic11} (iii). The two limiting properties in \equ{lim0} are given in Prop.\,\ref{basic11} (iv) {\clb with $\zeta=1$. In addition, Prop.\,\ref{basic11} (iii) ensures that $|u_k|>\al_*$ at each of its critical points}. This completes the proof of Theorem
\ref{hh2} (ii). For each $\al\in (\al_k, \al_{k+1})$, the jumping property of $\mathcal{N}(\al)$ implies that $u(r)$ cannot be a bound state, and Prop.\,\ref{basic13} (iv) indicates that
$u(r)$ is oscillatory with $\mathcal{N}(\al)=k+1$. According to Prop.\,\ref{basic12}, it oscillates about $u\equiv 1$ or $u\equiv -1$ behind its last zero. Theorem \ref{hh2} (iv) is thus confirmed.

{\clb To conclude, we present a simple verification of Theorem \ref{hh1}. Its first part, concerning the existence and uniqueness of a $k$-node bound state for any $k\ge 1$, follows directly from Theorem \ref{hh2}. It implies further that $\lim_{\al\to\infty}\mathcal{N}(\al)=\infty$. Consequently, for a given integer $i>0$,
the $i$th zero $z_i(\al)$ of $u(r,\al)$ exists for all sufficiently large $\al$, and  $z_i'(\al)<0$ by \equ{zpal}.} It is well-known that $\lim_{\al\tin}z_i(\al)=0$; see (4.1) in \cite{mtw}, which can be derived from Theorem 2.2 in \cite{n} or Corollary 6.7 in \cite{nn}. Consequently, it is evident that, on a given finite ball in $\mrn$, there exists a unique radial {\clb solution of \equ{ble1} that takes a positive value at its center and has} a prescribed number of sign changes within this ball. The second part of Theorem \ref{hh1} is thus confirmed. \end{proof}

{\clb The remainder of this work is devoted to the verification of conditions (a) and (b) in this lemma. These conditions suffice not only to ensure the uniqueness of the $k$-node bound state and the classification of radial solutions of (1.2), but also to elucidate how the variation $v$ behaves relative to $u$.}{\footnote{{\clb Conditions (a) and (b)} reveal that, hopefully, $u$ and $v$ are not twisted in a tangle, but sway together gracefully like dancing a tango. In an animated version, we may restate {\clb (a) and (b)} by the rules of tango:
(i) Land one foot between the partner's feet; $v=0$ exactly once between two consecutive zeros of $u$. (ii) When landing, they face the same direction; $u'v'>0$ whenever $v=0$. (iii) Never step on the partner's feet; $v\ne 0$ whenever $u=0$. (iv) To conclude the dance, jump higher; $\lim_{r\to\infty} |v(r)|=\infty$.}} {\clb In addition, the divergent behavior exhibited by
the variation $v$ of a bound state $u$, $\lim_{r\to\infty} |v(r)|=\infty$, is of fundamental importance in the analysis of related nonlinear wave equations and Schr\"odinger equations.}
A solution $u$ of \equ{ble} is called non-degenerate if the kernel of the linearization $\Delta +f'(u)$, considered as a self-adjoint operator in $L^2(\mrn)$, is spanned by $\partial u/\partial x_1, \cdots, \partial u/\partial x_n$ \cite{awy, bas, len, nt}. It was noted by Weinstein \cite{we2} that the non-degeneracy of the unique ground state $u$ follows essentially from the divergence of $v$ \cite{fra}.

\section{ {\clb Analysis in Phase 1}}

{\clb In this section, we investigate the behavior of the variation $v(r, \alpha)={\p u(r, \alpha)}/{\p\alpha}$ on the first-phase interval $(0, c_1]$, where $u$ is the solution of
\equ{ble3} with a general nonlinearity $f$ satisfying some conditions to be specified later.}

\subsection {The tendency to change sign} 

We {\clr present a result based on the condition 

(C9) $f''(u)>0$ for $u>1$.}

\begin{lemma}\label{cs} {\clb Assume that (C1) and {\clr (C9) hold}. Let $u=u(r, \alpha)$ be a ground state or a nodal solution of \equ{ble3}. Then we have

(i) $v(r)$ changes sign in $(0, r_1)$.

(ii) If $u(r)$ has $k\ge 2$ zeros and $2\le i\le k$, then $v(r)$ changes sign in $(\overline r_{i-1}, r_i)$. Moreover, if $v(c_{i-1})\cdot v'(c_{i-1})<0$, then $v(r)$ changes sign in $(c_{i-1}, r_i)$. }
\end{lemma}

\begin{proof} (i) Suppose to the contrary that $v(r)$ does not change sign in $(0, r_1)$. Then $v>0$ in $(0, r_1)$. {\clb Consider the Wronskian of $u'$ and $v'$:
\bl{rho}  \varrho(r):=r^{n-1} \left(u''v'-u'v''\right)=r^{n-1} \left[f'(u)u'v-f(u)v'\right]. \ee
Since $u>1$ in $(0, r_1)$, we find from \equ{rhop} that
$ \varrho'(r)= r^{n-1}f''(u)u^{\prime 2}v>0$ in $(0, r_1)$ and $\varrho(r_1)>\varrho(0)=0$.
However, substituting $f(u(r_1))=f(1)=0$, $f'(u(r_1))=f'(1)\ge 0$, $u'(r_1)<0$, and $v(r_1)\ge 0$ into \equ{rho} yields $\varrho(r_1)\le 0$. This leads to a contradiction.}

(ii) {\clb Since $f$ is odd by (C1), it follows that $f'(1)=f'(-1)\ge 0$, and that $f''(u)<0$ for $u<-1$. By using \equ{rho} together with $|u(\overline r_{i-1})|=|u(r_{i})|=1$, we deduce that} \bl{rhor} \varrho(\overline r_{i-1})=\overline r_{i-1}^{n-1}{\clb f'(1)}u'(\overline r_{i-1})v(\overline r_{i-1}),\quad \varrho(r_i)=r_i^{n-1}{\clb f'(1)}u'(r_i)v(r_i).\ee
Suppose for contradiction that $v>0$ in $(\overline r_{i-1}, r_i)$. If $i$ is even,
then $u'(\overline r_{i-1})<0$, $u'(r_i)>0$, and $u<-1$ in $(\overline r_{i-1}, r_i)$.
As a result, \equ{rhor} yields $\varrho(\overline r_{i-1})\le 0$ and $\varrho(r_i)\ge 0$,
but \equ{rhop} gives $\varrho'(r)<0$ in $(\overline r_{i-1}, r_i)$ and a contradiction.
If $i$ is odd, then $u'(\overline r_{i-1})>0$, $u'(r_i)<0$, and $u>1$ in $(\overline r_{i-1}, r_i)$. Consequently, \equ{rhor} yields $\varrho(\overline r_{i-1})\ge 0$ and $\varrho(r_i)\le 0$,
while \equ{rhop} gives $\varrho'(r)>0$ in $(\overline r_{i-1}, r_i)$ and again
a contradiction. The same reasoning shows that it is also impossible to have $v<0$ in $(\overline r_{i-1}, r_i)$.

Suppose $v(c_{i-1})\cdot v'(c_{i-1})<0$. If $v<0$ in $(c_{i-1}, r_i)$, then $v(c_{i-1})<0$ and $v'(c_{i-1})>0$. When $i$ is even, we find $\dsp \varrho(c_{i-1})=-c_{i-1}^{n-1}f(u)v'>0$ from  \equ{rho} and $\dsp \varrho(r_i)\le 0$ from \equ{rhor}, which contradict $\varrho'(r)>0$ in $(c_{i-1}, r_i)$ derived from \equ{rhop}. When $i$ is odd, \equ{rho} yields $\dsp \varrho(c_{i-1})<0$ and \equ{rhor} yields $\dsp \varrho(r_i)\ge 0$, which contradict $\varrho'(r)<0$ in $(c_{i-1}, r_i)$. A similar argument rules out the possibility that $v>0$ in $(c_{i-1}, r_i)$. \end{proof}

\subsection {The resistance to {\clb further sign changes}} For a ground state $u$, {\clb it is rather technical to show that $v$ does not change sign a second time \cite{coffman, fra, k, mc, tao}. Most of our proof below involves determining the signs of }
\bl{qm} Q(r):=r^n\left[u'v'+f(u)v\right]+(n-2)r^{n-1}u'v, \quad M(r):=r^{n-1}(u'v-uv'),\ee
and the following function introduced in \cite{tang03}:
\bl{t1g1} T_1(r):=Q(r)-g_1(u(r))M(r), \quad g_1(u):=2f(u)/[uf'(u)-f(u)].\ee

\begin{lemma}\label{rcs}  {\clb Assume that (C1), (C3), (C5)--(C6), and (C9) hold. In addition, suppose that $uf'(u)-f(u),\, g_1'(u)>0$ for $u>0$.}

(i) If $u(r)$ is a nodal solution of \equ{ble3}, then there exists $\tau_1\in (0, r_1)$ such that $v(r)>0$ in $(0, \tau_1)$, $v(\tau_1)=0$, and $v(r)<0$ in $(\tau_1, z_1]$. {\clb In addition,}
\bl{qmt1++} Q(r), \,\, M(r)>0, \,\,\, r\in (0, z_1]; \quad T_1(r), \,\, T_1'(r)>0, \,\,\, r\in (0, z_1).\ee

(ii) If $u(r)$ is a ground state of \equ{ble3}, then there exists $\tau_1\in (0, r_1)$ such that
\bl{rcs1} v(r)>0 \,\,\,{\rm in}\,\,\, (0, \tau_1), \,\,\, v(\tau_1)=0, \,\,\, v(r)<0 \,\,\,{\rm in}\,\,\, (\tau_1, \infty), \,\,\, \lim_{r\to\infty} v(r)=-\infty.\ee
\end{lemma}
 \begin{proof} (i) According to Lemma \ref{cs}\,(i), there exists $\tau_1 \in (0, r_1)$ such that $u>1$ and $v>0$ in $(0, \tau_1)$, $v(\tau_1)=0$ and $v'(\tau_1)<0$. {\clb From the appendix, we have
 \ben Q'(r)=2r^{n-1}f(u)v,\quad M'(r)=r^{n-1}[uf'(u)-f(u)]v,\quad T_1'(r)= -g_1'(u)u'\cdot M(r).\een
 The assumptions in this lemma, together with $Q=M=T_1=0$ at $r=0$, yield}
\bl{qmt1+} Q'(r),\, M'(r)>0,\,\,\, r\in(0, \tau_1);\quad  Q(r),\, M(r), \, T_1'(r),\, T_1(r)>0, \,\,\, r\in (0, \tau_1].\ee
{\clb Should $v$ have a second zero on $(0, z_1]$, then $M\le 0$ at the second zero. To eliminate this possibility,} it suffices to show that
 \bl{m+} M(r)>0, \qquad r\in (0, z_1].\ee
If \equ{m+} were false, then there would be $\tau_{1m}\in (\tau_1, z_1]$ such that
\bl{t1s}  M(r)>0 \,\,\, {\rm in}\,\,\, (0, \tau_{1m}),\,\,\, M(\tau_{1m})=0; \quad v(r)<0\,\,\, {\rm in}\,\,\, (\tau_1, \tau_{1m}).\ee
Hence {\clb $T_1'(r)>0$ in $(0, \tau_{1m})$, and \equ{t1g1} implies that}
\ben Q(\tau_{1m})=T_1(\tau_{1m})+g_1(u(\tau_{1m}))M(\tau_{1m})=T_1(\tau_{1m})>T_1(0)=0.\een
Thus $\tau_{1m}$ cannot be a zero of $v$ and $v(\tau_{1m})<0$. We then derive
from $M(\tau_{1m})=0$ that $uv'=u'v>0$ and $u'/v'=u/v<0$ at $r=\tau_{1m}$, and for $P_1$ defined in \equ{P1},
\bl{p1tm} P_1(\tau_{1m})&=& \tau_{1m}^n\left[u'v'\cdot \frac{u'}{v'}+f(u)v\cdot\frac{u}v\right]+(n-2)\tau_{1m}^{n-1}uv'\cdot \frac{u'}{v'}\nonumber\\
&=& \left\{\tau_{1m}^n\left[u'v'+f(u)v\right]+(n-2)\tau_{1m}^{n-1}uv'\right\} \cdot \frac{u}{v} \nonumber\\
&=& Q(\tau_{1m})\cdot \frac{u(\tau_{1m})}{v(\tau_{1m})}<0. \ee
This contradicts Prop.\,\ref{basic2} (ii) and confirms \equ{m+}.

{\clb We now verify \equ{qmt1++}. Since $M'(r)>0$ in $(0, \tau_1)$ and $M'(r)<0$ in $(\tau_1, z_1)$, we deduce from $M(0)=0$ and $M(z_1)=z_1^{n-1}u'v>0$ that $M>0$ on $(0, z_1]$. Thus $T_1'(r)>0$ and $T_1(r)>0$ for $r\in (0, z_1)$. As $g_1(u(r_1))=g_1(1)=0$, \equ{t1g1} gives
$Q(r_1)=M(r_1)>0$. Since $Q$ increases from $Q(0)=0$ in $(0,\tau_1)$, decreases in $(\tau_1, r_1)$ with $Q(r_1)>0$, and then increases in $(r_1, z_1)$, we conclude with $Q>0$ in $(0, z_1]$.}

(ii) Let $u(r)$ be a ground state. {\clb By replacing $z_1$ with $\infty$ and repeating the first paragraph in the proof of Part (i), we deduce that $v(r)$ has a unique zero $\tau_1\in (0,\infty)$, with $\tau_1\in (0, r_1)$. Thus} $(r^{n-1}v')'=-r^{n-1}f'(u)v<0$ for $r$ sufficiently large, and $v$ is ultimately monotone. Should $v(r)$ have a finite limit as $r\tin$, {\clb then we would have $\lim_{r\to\infty}Q(r)=\lim_{r\to\infty}M(r)=0$ due to the exponential decay of $u$ at infinity. The second paragraph in the proof of Part (i) shows that $Q(r),\, M(r)>0$ in $(0,\infty)$. Consequently, $Q'(r)>0$ in $(r_1, \infty)$ and
$\lim_{r\to\infty}Q(r)>Q(r_1)>0$.
This leads to a contradiction; hence we must have $\lim_{r\to\infty} v(r)=-\infty$.} \end{proof}

\subsection{A bridge crossing the river} In the {\clb second part $(z_1, c_1]$ of Phase 1, $u$ becomes negative and decreases to its absolute minimum at $c_1$. To elucidate the behavior of $v$ in $(z_1, c_1]$, we focus primarily on determining the signs of $Q$ and $M$, along with the $Q$-family functions $Q_i(r)=Q(r)+ir^{n-1}u'v$ and the ``bridging function"
\bl{ba} B_a(r):=Q(r)-aM(r)-2F_a(u(r))\cdot \frac{r^{n-1}v}{u'},\ee
where $a$ is a real constant and
\bl{Fau} F_a(u):=F(u)-\frac a2\left[uf(u)-2F(u)\right]. \ee
In Phase 1, our analysis needs only the case $a = 0$. }

\begin{lemma}\label{phasel1}  {\clb Assume that all conditions of Lemma \ref{rcs} hold, and that
$u(r)$ is either a solution of \equ{ble3} with multiple zeros or a bound state.} Then there is $\tau_1\in (0, r_1)$ such that $v(r)>0$ in $(0, \tau_1)$, $v(\tau_1)=0$, and $v(r)<0$ in $(\tau_1, c_1]$.{\footnote{This result reveals {\clb the} counterintuitive, somewhat surprising, fact that $v$ stays {\clb entirely negative} when $u$ rolls down from $u(z_1)=0$ to its absolute minimum $u(c_1)$. By symmetry, one might expect {\clb $v$ to change sign exactly} once in $(z_1, c_1)$ as it does in $(0, z_1)$.}} Moreover
\bl{mq1q2} Q(r)>0 \,\, \, {\rm in}\, \, \, (0, z_1]\cup [\overline b_1, c_1], \quad M(r),
\, Q_1(r),\, Q_2(r)>0\,\, \, {\rm in}\, \, \, (0, c_1].\ee
\end{lemma}

 \begin{proof} According to Lemma \ref{rcs}, there exists $\tau_1\in (0, r_1)$ such that $v(r)>0$ in $(0, \tau_1)$, $v(\tau_1)=0$, and $v(r)<0$ in $(\tau_1, z_1]$.

{\clb We first show that $v<0$ on $[z_1, \overline b_1]$. It suffices to prove that
\bl{q1+} Q_1(r)=Q(r)+r^{n-1}u'v>0,\qquad r\in [z_1, \overline b_1],\ee
because $Q_1=r^n u'v'<0$ at the first presumed zero of $v$ in $(z_1, \overline b_1]$.
 As $u'v>0$ at $z_1$, it follows from \equ{qmt1++} that $Q_1(z_1)>Q(z_1)>0$.} If \equ{q1+} were false, then there would be $\widetilde q_1\in (z_1, \overline b_1]$ such that, for $Q_n(r)=Q(r)+nr^{n-1}u'v$,
\bl{tq1} v<0\,\,\,{\rm on}\,\,\, [z_1, \widetilde q_1], \quad
Q_n(r)>Q_1(r)>0\,\,\,{\rm in}\,\,\, [z_1, \widetilde q_1), \quad Q_1(\widetilde q_1)=0.\ee
{\clb With $a=0$, we have $F_0(u)=F(u)$, $B_0(z_1)=Q(z_1)$, and $B'_0(r)=-2F(u)Q_n(r)/(ru'^2)$ by \equ{bap}.} As $F(u)<0$ in $(z_1, \overline b_1)$, {\clb $B_0'(r)>0$ in $(z_1, \widetilde q_1)$. It follows} that
\bl{bqz} B_0(\widetilde q_1)>B_0(z_1)=Q(z_1)>0.\ee
However, if we evaluate $B_0(\widetilde q_1)$ by using \equ{ba} directly, then we find that
\ben B_0(\widetilde q_1)
&=& -\widetilde q_1^{n-1}u'v-2F(u)\cdot \frac{\widetilde q_1^{n-1}v}{u'} \quad \quad\qquad\qquad{\rm by}  \,\equ{tq1}, \nonumber\\&=& -\frac{2\widetilde q_1^{n-1}v(\widetilde q_1)}{u'(\widetilde q_1)}\cdot E(\widetilde q_1)<0, \quad
{\rm by}\,\,\equ{ef}, \,\, \equ{tq1},\,\, {\rm Prop.}\,\, \ref{basic1}\,{\rm (i)}, \een
which contradicts \equ{bqz} and confirms \equ{q1+}. Furthermore, it implies that
$v<0$ and $Q_n(r)>Q_1(r)>0$ on $[z_1, \overline b_1]$, and $B'_0(r)>0$ in $(z_1, \overline b_1)$.
By substituting $F(u(\overline b_1))=0$ and $F(u(z_1))=0$ into \equ{ba}, we also obtain
\bl{qb1z1} Q(\overline b_1)=B_0(\overline b_1)>B_0(z_1)=Q(z_1)>0. \ee

With $Q(\overline b_1)>0$ just {\clb established, it is straightforward} to show that $v<0$ on $[\overline b_1, c_1]$. On this interval, we have $u(r)\le -\alpha_*$, $f(u)<0$, and $Q'(r)=2r^{n-1}f(u)v>0$ as long as $v(r)<0$. Should $v$ stop being negative the first time at $\widetilde \tau\in (\overline b_1, c_1]$, then $v(\widetilde \tau)=0$, $v'(\widetilde \tau)>0$, and
$Q(\widetilde \tau)>Q(\overline b_1)>0$, which contradict $u'(\widetilde \tau)\le 0$.

To conclude, we verify \equ{mq1q2}. Recall from \equ{qmt1++} that $Q, \, M>0$ in $(0, z_1]$. Since $Q(\overline b_1)>0$ and $Q'(r)>0$ on $[\overline b_1, c_1]$, the first part of \equ{mq1q2} is verified. In $(z_1, c_1]$, {\clb we have $uf'(u)-f(u)<0$ and so $M'(r)>0$; thus $M(r)>M(z_1)>0$.} For $r\in (0, \tau_1)$, we have $f(u),\,f'(u),\,v>0$; hence $v'(r)<0$ by \equ{rnvp}, and $Q_1', \, Q_2'>0$ by \equ{q1} and \equ{q2}; thus $Q_1,\, Q_2>0$. For $r\in (\tau_1, z_1]\cup [\overline b_1, c_1)$, $u'v>0$ and $Q>0$ imply that $Q_2>Q_1>Q>0$. Also $Q_2(c_1)=Q_1(c_1)=Q(c_1)>0$. Finally, we see from \equ{q1+} that $Q_2>Q_1>0$ on $[z_1, \overline b_1]$. Thus \equ{mq1q2} is completely verified.{\footnote{Within $(b_1, \overline b_1)$,
$u$, $f(u)$, and $f'(u)$ all change {\clb sign} and $Q(r)$ may decrease, {\clb making it uncertain} whether $Q(r)>0$ for all $r\in (b_1, \overline b_1)$. {\clb The bridging function $B_0(r)$ serves as a bridge, enabling us} to drive $Q$ from $Q(b_1)>0$ to $Q(\overline b_1)>Q(b_1)>0$ without swimming in the ``river" $(b_1, \overline b_1)$. The remaining section $[\overline b_1, c_1]$ is an uphill street for $Q$, along which we readily see that $Q$ {\clb stays positive}.}} \end{proof}

\section{{\clb Transition to later phases}}

{\clr We now restrict our attention to equation \equ{eu} with $f(u)=-u+|u|^{p-1}u$ and $n\ge 3$.
From Table 1 in the appendix, we see that all the assumptions specified in Sections 2 and 4 are satisfied.}

\subsection{{\clb The phase transition lemma}} Inspired by {\clb our proof that the variation $v(r)$ corresponding to a nodal solution $u(r)$ vanishes exactly once in Phase 1, we attempt to prove the same for} each subsequent phase. Our major tool is the function
{\clb \bl{g2} T_2(r):=Q(r)-g_2(u(r))M(r), \quad g_2(u):=2F(u)/[uf(u)-2F(u)]. \ee
For the model nonlinearity $f$, there is a simple relation $g_1(u)=g_2(u)+|u|^{1-p}$, from which
we derive the crucial identity for $T_1$ and $T_2$:
\bl{t2} T_2(r)=T_1(r)+|u|^{1-p}M(r).\ee}
{\clb The function $T_2$, together with $Q$ and $M$,} serves our purpose for two reasons:

(i) {\it Sufficiency}: The positivity of $Q$, $M$, and $T_2$ at $c_{i-1}$,  the starting point of Phase $i$ interval $(c_{i-1}, c_i]$, is sufficient to ensure that $v$ changes sign exactly once in Phase $i$.

(ii) {\it Renewability}: Their positivity will be regained at the right end $c_i$.

Now we state the principle result of this section:

\begin{lemma}\label{ptl} {\bf (Phase Transition Lemma)} Let $u$ be a nodal solution of \equ{eu} and $k\ge 1$. Assume that $u$ has $k$ zeros if $u$ {\clb is a bound state}, or $k+1$ zeros if $u$ {\clb is not a bound state. Then $Q(c_{1}), M(c_{1}), T_2(c_{1})>0$. Furthermore, we have

(i) If $i\in \{2,  \cdots, k\}$ and $Q,\, M,\, T_2>0$ at $c_{i-1}$, then
$v$ has a unique zero $\tau_i$ in $[c_{i-1}, c_i]$ with $\tau_i\in (c_{i-1}, r_i)$, and
$Q,\, M,\, T_2>0$ at $c_{i}$.

(ii) If $u$ is not a bound state and $Q(c_{k}), M(c_{k}), T_2(c_{k})>0$, then
$v$ has a unique zero $\tau_{k+1}$ on $[c_{k}, z_{k+1}]$ with $\tau_{k+1}\in (c_{k}, r_{k+1})$.

(iii) If $u$ is a bound state and $Q(c_{k}), M(c_{k}), T_2(c_{k})>0$, then $v$ has a unique zero $\tau_{k+1}$ in $[c_k, \infty)$ with $\tau_{k+1}\in (c_{k}, r_{k+1})$. Moreover, $v$ is strictly monotone for $r$ sufficiently large, and $\lim_{r\to\infty} |v(r)|=\infty$.}
\end{lemma}

{\clb We devote most of this section to proving this lemma. Throughout the process, we maintain its basic assumption that $u$ has $k\ge 1$ zeros if $u$ is a bound state of \equ{eu}, or $k+1$ zeros if $u$ is not a bound state. In either case, $c_k$ is well-defined and $|u(c_k)|>\al_*$. The primary challenge is to establish the positivity of $T_2(c_1)$, and the positivity of $T_2(c_i)$ for $i\in \{2,  \cdots, k\}$ if $Q,\, M,\, T_2>0$ at $c_{i-1}$. Our proof requires crucial assistances from the bridging function $B_a(r)$, and also} from the following connection identity
\bl{conn} Q(r)-P(r)\cdot \frac vu=
\omega(r) \left[{M(r)}- \va(r)\right], \quad \va(r):=\frac{p-1}{p+1}\cdot \frac {r^{n-1}v}{u'}\cdot {|u|^{p+1}},  \ee
{\clb which connects $Q$, $M$, $P$, $\omega$, and $\va$. Through the important equalities}
 \be T_2'(r)&=& (p-1)r^{n-1}uv-\frac {(p+1)uu'}{|u|^{p+1}}\cdot M(r) \label{t2p}\\
&=& -(p+1)\cdot \frac{uu'}{|u|^{p+1}}\cdot \left[{M(r)}-\va(r) \right], \label{t2p2} \ee
we can connect $T_2$ with $Q$, $M$, and $P$. {\clb In Lemma \ref{b11},
we show that $v$ changes sign in Phase $i$ with its first zero at $\tau_i\in (c_{i-1}, r_i)$.
When $\tau_i\ge b_i$, a direct calculation using \equ{t2p} yields $T_2>0$ on $[c_{i-1}, b_i]$.
The same is true when $\tau_i< b_i$, but its proof requires a tremendous effort, beginning with the crucial step to show that $v$ has no zeros in $(\tau_i, b_i]$ by using the connection identity. Assume that $T_2$ attains its minimum on $[\overline b_i, c_i]$ at $\overline t_i$, and define $t_i\in (c_{i-1}, b_i]$ by $|u(t_i)|=|u(\overline t_i)|$.
To reach our goal $T_2(c_i)>0$, it suffices to show that $T_2(\overline t_i)\ge T_2(t_i)$.
As $T_2$ is undefined at
$z_i\in (t_i, \overline t_i)$, we transform $T_2(\overline t_i)\ge T_2(t_i)$ into an integral condition of $B'_a(r)$ on $[t_i, \overline t_i]$. With additional conditions imposed, Lemmas \ref{b14}-\ref{qrn} provide several delicate estimates that facilitate our validation of the integral condition in Lemma \ref{t2c1}. Finally, we obtain $T_2(c_i)>0$ by verifying all the conditions required in Lemmas \ref{b14}-\ref{t2c1}.}

\subsection{Locating the zeros of $v(r)$ and keeping a record of signs}

\begin{lemma}\label{b11} Suppose $Q(c_{i-1}), \, M(c_{i-1})>0$ for a number $i\in \{2,  \cdots, k, {\clb k+1}\}$. Then $v(r)$ changes sign in $(c_{i-1}, r_i)$. Let $\tau_i$ be its first zero in $(c_{i-1}, r_i)$. Then
\bl{cip} uf(u),\, f'(u),\, uv,\, u'v'>0, \quad uu',\, uv',\, u'v,\, vv'<0, \quad  r\in (c_{i-1}, \tau_i).\ee
In addition, \equ{cip} holds for $i=1$ {\clb with $c_0=0$}.
\end{lemma}

\begin{proof} Lemma \ref{phasel1} indicates that $v(r)$ has a unique zero $\tau_1$ in the first phase $[0, c_1]$ and $\tau_1\in (0, r_1)$. For $r\in (c_0, \tau_1)=(0, \tau_1)$, we have $u>1$, $u'<0$, $v>0$, and also \equ{rnvp} implies that $v'<0$. It is then clear that \equ{cip} holds for $i=1$.

Let $i\in \{2,  \cdots, k, {\clb k+1}\}$. {\clb At $c_{i-1}$, we have $u'=0$ and $|u|>\alpha_*$.  It follows from $Q(c_{i-1})>0$ and $M(c_{i-1})>0$} that
\bl{ci1p}  f(u)v>0, \quad uv'<0, \quad \Rightarrow \quad uv>0, \quad vv'<0, \quad {\rm at}\quad r=c_{i-1}.  \ee
According to Lemma \ref{cs}\,(ii), {\clb $v(r)$ changes sign in $(c_{i-1}, r_i)$ if $i\in \{2,  \cdots, k\}$, or $i=k+1$ and $u$ is not a bound state. If $i=k+1$ and $u$ is a bound state, then $u$ has only $k$ zeros and Lemma \ref{cs} does not apply directly; nevertheless, there is a unique $r_{k+1}>c_k>z_k$ such that $|u(r_{k+1})|=1$. By computing $\varrho(c_k)$ and $\varrho(r_{k+1})$ and evaluating $\varrho'$ in $(c_k, r_{k+1})$, it is seen that $v(r)$ must change sign in $(c_{k}, r_{k+1})$.} Let $r\in (c_{i-1}, \tau_i)$. Then $|u|>1$, {\clb which in turn implies} $uf(u), \, f'(u)>0$; {\clb moreover, $uu'<0$ since
$u<0$ and $u'>0$ for even $i$, and $u>0$ and $u'<0$ for  odd $i$.}

Note that $u$, $u'$, and $v$ do not change sign in $(c_{i-1}, \tau_i)$. We show that the same is true for $v'$. Indeed, if $v'(c_{i-1})>0$, then $v(c_{i-1})<0$ by \equ{ci1p} and so $v(r)<0$ in $(c_{i-1}, \tau_i)$. It follows from \equ{rnvp} that $r^{n-1}v'$ increases and so $v'(r)>0$ in this interval. Similarly, if $v'(c_{i-1})<0$, then $v(r)>0$ and \equ{rnvp} gives $v'(r)<0$ in $(c_{i-1}, \tau_i)$. Consequently, the {\clb assertions} for $uv$, $uv'$, and $vv'$ in \equ{cip} follow from \equ{ci1p}.  Finally, $u'v'>0$ in $(c_{i-1}, \tau_i)$ since $u^2(u'v')=(uu')(uv')>0$, and $u'v<0$ since $(uv)(u'v)=(uu')v^2<0$.  \end{proof}

\subsection{Positivity of $Q_n(r)$ {\clb conditional on the positivity of $Q_1(r)$}}
\begin{lemma}\label{b14} {\clb Assume that}
\bl{qmcq1} Q(c_{i-1}), \, M(c_{i-1})>0;\quad Q_1(r)>0,\,\,r\in [c_{i-1}, c_i],\quad i\in \{2,  \cdots, k\}. \ee
{\clb Then $v(r)$ has a unique zero $\tau_i$ on $[c_{i-1}, c_i]$ and $\tau_i\in (c_{i-1}, r_i)$.
Moreover, if $n=3$ and $p\in [2, 5)$, or $n\ge 4$, then $Q_n(r)=Q(r)+r^{n-1}u'v>0$ on $[b_i, c_i]$ for $i\in \{1, 2, \cdots, k\}$.}{\footnote{We see from \equ{qn} that $Q_n(r)<0$ for small $r>0$; hence $[b_i, c_i]$ may not be extended to the whole phase $(c_{i-1}, c_i]$ in general.}} \end{lemma}

\begin{proof} {\clb By Lemma \ref{b11}, $v(r)$ changes sign in $(c_{i-1}, r_i)$
and $u'v'>0$ at its first zero $\tau_i$ in $(c_{i-1}, r_i)$.} If $v$ has another zero, say $\widetilde \tau_i$, next to $\tau_i$ in $(\tau_i, c_i]$, then $u'(\widetilde \tau_i)\cdot v'(\widetilde \tau_i)< 0$ {\clb and $Q(\widetilde \tau_i)<0$ if $\widetilde \tau_i<c_i$, or $u'(\widetilde \tau_i)\cdot v'(\widetilde \tau_i)= 0$ and $Q(\widetilde \tau_i)=0$ if $\widetilde \tau_i=c_i$, either of which contradicts \equ{qmcq1}. Thus $v(r)$ has a unique zero $\tau_i$ on $[c_{i-1}, c_i]$.

When $\tau_i \le b_i$, it is straightforward to prove $Q_n(r)>0$ on $[b_i, c_i]$. Indeed,}
because $v$ changes sign exactly once at $\tau_i$ and $u'$ does not change sign in $(c_{i-1}, c_i)$, it follows from \equ{cip} that $u'v>0$ in $(\tau_i, c_i)$.
{\clb Consequently, $\dsp Q_n(r)\ge Q_1(r)>0$ on $[\tau_i, c_i]\supset [b_i, c_i]$.}

We now deal with the subtle case that $b_i<\tau_i$. Define
 \bl{db1} s_i=\inf \{s:\, s\in [c_{i-1}, c_i],\,\,\, Q_n(r)>0\, \, {\rm for}\,\, r\in (s, c_i]\}.\ee
Then $s_i<\tau_i$. In Phase 1, we have $s_1>c_0=0$. When $i>1$, it may happen that $Q_n(r)>0$ for all $r\in [c_{i-1}, c_i]$ because $Q_n(c_{i-1})=Q(c_{i-1})>0$. Of course, if this does happen, then $s_i=c_{i-1}$ and we have nothing to prove.

In what follows, we only need to consider the case $s_i>c_{i-1}$. {\clb By definition \equ{db1},}  $Q_n(s_i)=0$ and $Q_n(r)>0$ in $(s_i, c_i]$. We claim that
\bl{ffp} \frac {2(n-1)f(u)}{uf'(u)}|_{u=u(s_i)}>\omega(s_i)=-\frac{s_iu'}{u}.\ee
As $s_i\in (c_{i-1}, \tau_i)$, we can use \equ{cip} to check that both quantities in the two sides of \equ{ffp} are positive. As $Q_n(s_i)=0$, from the definition of $Q_n$ we find
\ben s_i[u'v'+f(u)v]=-2(n-1)u'v \quad \Rightarrow \quad -u'v|_{r=s_i}=\frac {s_i[u'v'+f(u)v]}{2(n-1)}.\een
{\clb Note that $\rho(0)=0$, and $\rho(c_{i-1})=-c_{i-1}^{n-1}f(u)v'>0$ for $i>1$ by \equ{rho}.
Since \equ{rhop} implies $\dsp \rho'(r)=p(p-1)r^{n-1}u|u|^{p-3}u^{\prime 2}v>0$ in $(c_{i-1}, \tau_i)$, we obtain}
\bl{rsi} \rho(s_i)=s_i^{n-1} \left[ f'(u)u'v-f(u)v'\right]>\rho(c_{i-1})\ge 0. \ee
With a careful examination on the signs of relevant quantities in \equ{cip}, we then find
\ben -\,\frac {f(u)v'}{f'(u)}|_{r=s_i}>-u'v|_{r=s_i}=\frac {s_i[u'v'+f(u)v]}{2(n-1)}>\frac {s_iu'v'}{2(n-1)}.\een
{\clb Multiplying these terms by $-\,2(n-1)/(uv')>0$ yields \equ{ffp}.}

Recall from Prop.\,\ref{basic2}\,(i) that $P(r)$ is universally positive in $(0, c_k]$. Since $uu'<0$ at $r=b_i$, we obtain an important {\clb inequality} $\omega(b_i)>n-2$ from
\bl{pb1n2} 0<P(b_i)=b_i^n u'^2+(n-2)b_i^{n-1}uu'=b_i^{n-1}|uu'|\cdot [\omega(b_i)-(n-2)].\ee

Apparently, $Q_n(r)>0$ on $[b_i, c_i]$ if and only if $s_i<b_i$. Suppose for contradiction that $s_i\ge b_i$. Then $\omega(s_i)\ge \omega(b_i)>n-2$ according to
Prop.\,\ref{basic2}\,(ii). By using \equ{ffp} and the exact forms of $f(u)$, we derive
\bl{b1} \frac {2(n-1)f(u)}{uf'(u)}|_{u=u(s_i)}>n-2 \quad \Rightarrow \quad [2(n-1)-p(n-2)]|u(s_i)|^{p-1}>n. \ee
If $n\ge 4$ and $1<p<(n+2)/(n-2)$, then
$2(n-1)-p(n-2)>0$. Since $|u(r)|$ decreases in $(c_{i-1}, \tau_i)$ and $s_i\in [b_i, \tau_i)$, it holds that $|u(s_i)|\le |u(b_i)|=\alpha_*$. The second inequality in \equ{b1} and $\al_*^{p-1}=(p+1)/2$ then lead to
\ben (p+1)[2(n-1)-p(n-2)]>2n \quad \Rightarrow \quad  [(n-2)p-2](p-1)<0, \een
which is apparently impossible if $p>1$ and $n\ge 4$. When $n=3$, \equ{b1} becomes $\dsp (4-p)|u(s_i)|^{p-1}>3$, which cannot be true if $p\ge 4$. If $2\le p<4$, then
substituting $|u(s_i)|\le \alpha_*$ into \equ{b1} leads to
\ben  (4-p)(p+1)>6\quad \Rightarrow \quad p^2-3p+2=(p-1)(p-2)<0.\een
It contradicts $p\ge 2$ and completes our proof.  \end{proof}

\subsection{{\clb Positivity of an integral} for the residual case $n=3$ and $p\in (1, 2)$}
It is unclear if the conclusion that $Q_n(r)>0$ on $[b_i, c_i]$ in Lemma \ref{b14}
 {\clb can be extended to} this residual case, because no contradictions {\clb follow immediately} from \equ{b1}. {\clb By the definition of $s_i$, this conclusion may fail only if}
\bl{sbt1} b_i<s_i<\tau_i.\ee
{\clb To pave the way for our later proof that $T_2(c_i)>0$, we establish the positivity of an integral involving $Q_3(r)$ over $[b_i, \tau_i]$.} Our proof {\clb relies on} a variety of delicate estimates such as $3(\sqrt e-1)\approx 1.9462<2$ and reveals more subtle properties of $u$ and $v$.

\begin{lemma}\label{b13} Let $n=3$ and $p\in (1, 2)$. Suppose \equ{qmcq1} holds and
$\tau_i>b_i$. Then
\bl{I13} I_{i}:=\int_{b_i}^{\tau_i}{u^2(r)} \left(1-\left|\frac {u(r)}{\widetilde u}\right|^{p-1}\right) \frac{Q_3(r)}{ru'^2(r)}\, dr>0\ee
for any number $\widetilde u \ge |u(b_i)|=\al_*$. \end{lemma}

\begin{proof} {\clb By the definition of $s_i$, $Q_3(r)>0$ in $(b_i, \tau_i)$ and \equ{I13} holds trivially if $s_i\le b_i$.} Hence we only need to deal with the case when \equ{sbt1} holds.
We first present several supporting observations over $(b_i, \tau_i)$, listed in {\it S1 -- S8}:

\noindent
{\it S1: $Q(r),\, M(r)>0$ and $\dsp -rv'(r)/v(r)>1$ in $(b_i, \tau_i)$}. Lemma \ref{phasel1} assures that $Q,\,M>0$ in $(b_1, \tau_1)$. For $i>1$, \equ{qmcq1} gives $Q(c_{i-1}),\, M(c_{i-1})>0$. By checking $Q'$ and $M'$, we see that $Q$ and $M$ increase in $(c_{i-1}, \tau_i)$. Hence $Q,\, M>0$ in $(b_i, \tau_i)\subset
(c_{i-1}, \tau_i)$. {\clb Since the universal positivity of $P(r)$ implies $\omega(b_i)>n-2=1$; see \equ{pb1n2}, the monotonicity of $\omega$ further implies that $\omega(r)>\omega(b_i)>1$ for $r\in (b_i, \tau_i)$.} Thus $M(r)>0$ leads to $\dsp u'v>uv'$, and in turn,
$\dsp -rv'(r)/v(r)>\omega(r)>1$ {\clb for $r\in (b_i, \tau_i)$}.

\noindent
{\it S2: The ratio $f(u)v/(u'v')$ {\clb decreases strictly with $r$} in $(b_i, \tau_i)$}.
Let $p\in (1, 2)$, and $r\in (b_i, \tau_i)$. Then $\omega(r)>1$, $\dsp |u(r)|^{p-1}< |u(b_i)|^{p-1}=(p+1)/2<2/(3-p)$, and
\ben  \frac{3f(u)}{uf'(u)}-\omega(r)<\frac{3f(u)}{uf'(u)}-1
=\frac{3-p}{f'(u)}\left(|u|^{p-1}-\frac2{3-p}\right)<0. \een
{\clb With the aid of {\it S1}, a direct computation yields}
\ben u'v'\cdot\frac{d}{dr}\left(\frac{f(u)v}{u'v'}\right)&=&
f'(u)u'v+f(u)v'+\frac{f(u)v}{u'v'}\left(\frac{4}ru'v'+f(u)v'+f'(u)u'v\right)\\
&<& f'(u)u'v+f(u)v'+{4}f(u)v/r\\
&=& \frac{f(u)v}{r}\left(\frac{rv'}v+1\right)+\frac 1rf'(u)uv\left(\frac{3f(u)}{uf'(u)}-\omega(r)\right)<0.\een

\noindent
{\it S3: $Q'_3(r)>0$ in $(b_i, \tau_i)$}.  {\clb Since}
\bl{vnp} Q'_3(r)=r^{2} \left[3u'v'-f(u)v\right]=r^{2}u'v' \left[3-{f(u)v}/(u'v')\right],\ee
$Q'_3(\tau_i)=3\tau_i^{2}u'v'>0$. If {\it S3} were false, then there would be $r_*\in (b_i, \tau_i)$ such that $Q'_3(r_*)=0$. Hence $f(u)v/(u'v')=3$ at $r_*$, and {\it S2} implies that $Q_3(r)$ takes the absolute minimum value on $[b_i, \tau_i]$ at $r_*$. However, {\clb the definition of $Q_n(r)$ gives}
\ben {Q_3(r_*)}/{r_*^{2}}=r_*[u'v'+f(u)v]+4u'v=4r_*u'v'+4u'v=4u'v(1+r_*v'/v)>0. \een
It results in $Q_3(r)\ge Q_3(r_*)>0$ on $[b_i, \tau_i]$, which violates \equ{sbt1} and confirms {\it S3}.

\noindent
{\it S4: $u'^2(r)$ decreases (strictly) in $(b_i, \tau_i)$}. Prop.\,\ref{basic21} shows that $u''$ changes sign exactly once in $(c_{i-1}, z_i)$. Since $u'(z_i)u''(z_i)=-2u'^2(z_i)/z_i<0$
and $u'$ does not change sign in $(c_{i-1}, z_i)$, there exists $r_i^*\in (c_{i-1}, z_i)$ such that $u'u''>0$ in $(c_{i-1}, r_i^*)$, and $u'u''<0$ in $(r_i^*, z_i]$. If {\it S4} were not true, then $b_i< r_i^*$ and
\ben u''(b_i)v(b_i)=-\frac{2u'v}{b_i}-f(u)v< 0 \quad \Rightarrow \quad -\frac{2u'v}{b_i}< f(u)v|_{r=b_i}. \een
As {\it S3} gives {\clb $Q_3(b_i)< Q_3(s_i)=0$, we deduce} that, at $r=b_i$,
\bl{s1up} u'v'+f(u)v<-\frac{4u'v}{b_i}< 2f(u)v \quad \Rightarrow\quad 2u'v'<u'v'+f(u)v.\ee
Since {\clb $\rho(c_{i-1})\ge 0$ and $\dsp \rho'(r)>0$ in $(c_{i-1}, \tau_i)$,}
$\dsp \rho(b_i)=b_i^{n-1} \left[ f'(u)u'v-f(u)v'\right]>0$, which implies that, with the help of the first inequality in \equ{s1up},
\ben -\,f(u)v'|_{r=b_i}>\frac{b_if'(u)}{4}\cdot \left(\frac{-4u'v}{b_i}\right)
>\frac{b_if'(u)}{4}\cdot [u'v'+f(u)v].\een
{\clb Together with the second part of \equ{s1up}, this implies} that, again at $r=b_i$,
\ben 2b_iu'v'< b_i[u'v'+f(u)v] < -\frac{4f(u)v'}{f'(u)} \quad\Rightarrow\quad
 \frac{2f(u)}{uf'(u)}>
-\frac {b_iu'}u=\omega(b_i)>1. \een
However, continuing from the last inequality leads to
\ben 1<\frac{2f(u)}{uf'(u)}|_{r=b_i}=\frac {2(-1+|u(b_i)|^{p-1})}{-1+p|u(b_i)|^{p-1}}=\frac{p-1}{p(p+1)/2-1}, \een
which is invalid for any $p>1$. Thus {\it S4} is verified.

\noindent
{\clb {\it S5: $0<|u(b_i)|-|u(s_i)|<\sqrt e-\sqrt[3]{e}\approx 0.2531$.} Since $Q_3(s_i)=0$, $\rho(s_i)>\rho(b_i)>0$, and $\omega(s_i)>\omega(b_i)>1$, it follows that $\dsp |u(s_i)|>\left[3/(4-p)\right]^{1/(p-1)}$; see also \equ{b1}. Observing that
$\dsp |u(b_i)|=\left[(p+1)/2\right]^{1/(p-1)}\to \sqrt e$ as $p\downarrow 1$ and decreases in $p>1$, and $\left[3/(4-p)\right]^{1/(p-1)}\to \sqrt[3]{e}$ as $p\downarrow 1$ and increases in $p\in (1, 2)$, we obtain {\it S4}.}

\noindent
{\it S6:} $\dsp 3|u(s_i)-u(\tau_i)|>\tau_i|u'(\tau_i)|$. {\clb Note that $f'(u)>0$ and $vv'<0$ in $(c_{i-1}, \tau_i)$. We have $\dsp (r^{2}|v'|)'=r^{2}f'(u)v>0$ if
$v>0$ and $v'<0$, or $\dsp (r^{2}|v'|)'=-r^{2}f'(u)v>0$ if
$v<0$ and $v'>0$ in $(c_{i-1}, \tau_i)$. By integrating \equ{vnp}, we estimate}
 \ben && Q_3(\tau_i)=3\int_{s_i}^{\tau_i} r^{2} u'v'\, dr-\int_{s_i}^{\tau_i} r^{2}f(u)v\, dr
 < 3\int_{s_i}^{\tau_i} r^{2} u'v'\, dr\\
 &<&3\tau_i^{2}|v'(\tau_i)|\int_{s_i}^{\tau_i}|u'|\, dr=
 3\tau_i^{2}|v'(\tau_i)|\cdot|u(s_i)-u(\tau_i)|. \een
{\clb Given that $Q_3(\tau_i)=\tau_i^3|u'(\tau_i)||v'(\tau_i)|$, {\it S6} follows at once.}

\noindent
{\it S7: The differences $u(\tau_i)-u(s_i)$ and $u(s_i)-u(b_i)$ obey}
\bl{btsb} \frac{u(\tau_i)-u(s_i)}{u(s_i)-u(b_i)}>\omega^2_i, \quad{\rm where}\quad
\omega_i:=\frac {\omega(\tau_i)} {\omega(b_i)}>1.\qquad \ee
To {\clb prove this, we use {\it S6}, $|u(\tau_i)|>1$, and $\omega(b_i)>1$ to estimate}
\ben 3|u(s_i)-u(\tau_i)|>\tau_i|u'(\tau_i)|=|u(\tau_i)|\cdot \omega (\tau_i)>{\omega(\tau_i)}={\omega_i\omega(b_i)}>{\omega_i}.\een
Since $1<|u(\tau_i)|<|u(s_i)|<|u(b_i)|<\sqrt e$ (see {\it S5}), we obtain
\ben  |u(b_i)-u(s_i)|<\sqrt e-1-|u(s_i)-u(\tau_i)|<\sqrt e-1-{\omega_i}/{3},\een
and also $\dsp \omega_i<3(\sqrt e-1)\approx 1.9462<2.$ We thus conclude
\ben \frac{u(\tau_i)-u(s_i)}{u(s_i)-u(b_i)}=\frac{|u(s_i)-u(\tau_i)|}{|u(b_i)-u(s_i)|}>\frac{\omega_i}{3(\sqrt e-1)-\omega_i}>\frac{\omega_i}{2-\omega_i}>\omega_i^2. \een

\noindent
{\it S8: For each $s^{-}_i\in [b_i, s_i)$, define $s_i^+\in (s_i, \tau_i)$ as the unique number in $(s_i, \tau_i)$ such that $u(s^{-}_i)-u(s_i)=u(s_i)-u(s_i^+)$.
Then}
 \bl{r1r1} Q_3(s_i^+)>|Q_3(s^{-}_i)|.\ee
The existence of $s_i^+$ follows from {\it S7}, and the uniqueness follows from the monotonicity of $u$ in $(b_i, \tau_i)$, in which $f(u)v/(u'v')<3$ and decreases by {\it S2} and {\it S3}, and
 $r^2|v'|$ increases by {\it S6}.
By integrating \equ{vnp} over $(s^{-}_i, s_i)$, we derive
\ben |Q_3(s^{-}_i)|&=& \int_{s^{-}_i}^{s_i}r^2[3u'v'-f(u)v]\, dr
<(3-\varepsilon_i)\int_{s^{-}_i}^{s_i}r^2u'v'\, dr,\quad  \varepsilon_i=\frac{f(u)v}{u'v'}|_{r=s_i}<3\\
&<& (3-\varepsilon_i)s_i^2|v'(s_i)|\int_{s^{-}_i}^{s_i}|u'|\, dr=(3-\varepsilon_i)s_i^2|v'(s_i)|\cdot|u(s^{-}_i)-u(s_i)|.\een
A similar calculation over $(s_i, s_i^+)$ leads to the confirmation of \equ{r1r1}:
\ben Q_3(s_i^+)&=&\int_{s_i}^{s_i^+}r^2[3u'v'-f(u)v]\, dr>(3-\varepsilon_i)\int_{s_i}^{s_i^+}r^2u'v'\, dr\\
&>& (3-\varepsilon_i)s_i^2|v'(s_i)|\cdot |u(s_i)-u(s_i^+)|>|Q_3(s^{-}_i)|.\een

We now work with $I_i$ directly. As $|u(r)|$ decreases and $\widetilde u \ge |u(b_i)|>|u(r)|$ in $(b_i, \tau_i)$, by using the definition of $\omega(r)$ {\clb and {\it S3}}, we find that
\ben I_i=\int_{b_i}^{\tau_i} \left(1-\left|\frac {u(r)}{\widetilde u}\right|^{p-1}\right)\frac{rQ_3(r)}{\omega^2(r)}   \, dr
> s_i\left(1-\left|\frac {u(s_i)}{\widetilde u}\right|^{p-1}\right)\int_{b_i}^{\tau_i} \frac{Q_3(r)}{\omega^2(r)}\, dr. \een
As $\omega(r)>1$ increases in $(s_i, \tau_i)$, we use $\omega_i={\omega(\tau_i)}/{\omega(b_i)}$ to derive
\bl{j13} && \int_{b_i}^{\tau_i} \frac{Q_3(r)}{\omega^2(r)}\, dr=
\int_{s_i}^{\tau_i} \frac{Q_3(r)}{\omega^2(r)}\, dr-\int_{b_i}^{s_i}\frac{|Q_3(r)|}{\omega^2(r)}\, dr
\nonumber \\
&>& 
\frac 1 {\omega^2(\tau_i)}\underbrace{\left[\int_{s_i}^{\tau_i} Q_3(r)\, dr   -\omega^2_i\int_{b_i}^{s_i}|Q_3(r)|\, dr\right]}_{J_{i}}. \ee

It suffices to prove $J_{i}>0$. As $u(r)$ is monotone in $(b_i, \tau_i)$, its inverse,
denoted by $r_u$ with $u(r_u)=u$, is well-defined. As $|u'(r)|$ decreases in $(b_i, \tau_i)$ by {\it S4}, we find
\bl{bsq} \int_{b_i}^{s_i} |Q_3(r)|\, dr< \frac 1{u'(s_i)}  \int_{b_i}^{s_i} |Q_3(r)|\cdot u'(r)\, dr
=\frac 1{u'(s_i)} \int_{u(b_i)}^{u(s_i)}{|Q_3(r_u)|}\, du.
\ee
For each $\mu$ between $u(s_i)$ and $u(b_i)$, there is a unique $x\in [0, 1]$ such that
\bl{mux} \mu=u(r_\mu)=u(s_i)+[u(b_i)-u(s_i)]x. \ee
It defines a bijection between $[0, 1]$ and $[u(b_i), u(s_i)]$ when $i$ is even and $u'>0$ in $(b_i, \tau_i)$, or $[u(s_i), u(b_i)]$ when $i$ is odd and $u'<0$ in $(b_i, \tau_i)$.
For $x$ and $\mu$ satisfying \equ{mux}, we define $\widetilde r_x=r_\mu$.
Then $\widetilde r_0=s_i$, $\widetilde r_1=b_i$, and
\bl{sbtb} \int_{b_i}^{s_i} |Q_3(r)|\, dr<\frac 1{u'(s_i)} \int_{u(b_i)}^{u(s_i)}{|Q_3(r_u)|}\, du=\frac{u(s_i)-u(b_i)}{u'(s_i)}{\int_0^1 {|Q_3(\widetilde r_x)|}\, dx}.\ee
Similarly, for each $\mu$ between $u(s_i)$ and $u(\tau_i)$, there is a unique $x\in [0, 1]$ such that
\bl{mux2} \mu=u(r_\mu)=u(s_i)+[u(\tau_i)-u(s_i)]x. \ee
For $x$ and $\mu$ satisfying \equ{mux2}, we define $\widehat r_x=r_\mu$. Then
\bl{sbtb2} \int_{s_i}^{\tau_i} Q_3(r)\, dr>\frac 1{u'(s_i)} \int_{u(s_i)}^{u(\tau_i)}{Q_3(r_u)}\, du=\frac{u(\tau_i)-u(s_i)}{u'(s_i)}\int_0^1 {Q_3(\widehat r_{x})}\, dx.
\ee
For a given $x\in (0, 1]$, write $s^{-}_i=\widetilde r_x\in [b_i, s_i)$. For the corresponding $s_i^+\in (s_i, \tau_i)$ defined in {\it S8}, it follows from \equ{mux}, \equ{mux2}, and {\it S7} that \ben && |u(s_i)-u(s_i^+)|=|u(s^{-}_i)-u(s_i)|=|u(\widetilde r_x)-u(s_i)|\\
&& =|u(b_i)-u(s_i)|x<|u(s_i)-u(\tau_i)|x=|u(s_i)-u(\widehat r_x)|.\een
Hence $\widehat r_x>s_i^+$. As $Q_3(r)$ increases in $(b_i, \tau_i)$ by {\it S3}, applying {\it S8} leads to
\ben Q_3(\widehat r_x)>Q_3(s_i^+)>|Q_3(s^{-}_i)|=|Q_3(\widetilde r_x)|,\qquad x\in (0, 1].\een
By combining this with \equ{sbtb}, \equ{sbtb2}, and {\it S7}, we finally derive
\ben J_{i} &=& \int_{s_i}^{\tau_i} Q_3(r)\, dr-\omega^2_i\int_{b_i}^{s_i}|Q_3(r)|\, dr\\
&>& \frac{u(\tau_i)-u(s_i)}{u'(s_i)}\int_0^1 {Q_3(\widehat r_{x})}\, dx-\omega^2_i\cdot\frac{u(s_i)-u(b_i)}{u'(s_i)}{\int_0^1 {|Q_3(\widetilde r_x)|}\, dx}\\
&>& \frac{u(s_i)-u(b_i)}{u'(s_i)}\left(\frac{u(\tau_i)-u(s_i)}{u(s_i)-u(b_i)}-\omega^2_i\right)
\int_0^1 {Q_3(\widehat r_{x})}\, dx>0.\een
This completes the proof of the lemma. \end{proof}

\subsection{A comparison lemma}
\begin{lemma}\label{qrn} Assume that \equ{qmcq1} holds and that, for a number $i\in \{2, \cdots, k\}$,
\bl{qrna} Q(r)>0, \,\,\, r\in (c_{i-1}, b_i]\cup [\overline b_i, c_i]; \qquad Q(\overline b_i)>Q(b_i).\ee
Define
$r_{i\mu}\in (c_{i-1}, b_i]$ and $\overline r_{i\mu}\in [\overline b_i, c_i)$ by $|u( r_{i\mu})|=| u(\overline r_{i\mu})|=\mu$, $\mu \in [\alpha_*, |u(c_i)|)$. Then
\bl{var} \varphi(\overline r_{i\mu})>\varphi(r_{i\mu}), \qquad   \varphi(r):=\frac{Q(r)}{r^{n-1}|u'(r)|}.\ee
 In particular, {\clb both \equ{qrna} and \equ{var} hold for $i=1$.}
 \end{lemma}

\begin{proof} {\clb By Lemma \ref{phasel1}, the first part of \equ{qrna} holds in Phase 1. We claim that the second part, $Q(\overline b_1)>Q(b_1)$, also holds in Phase 1.}
If {\clb $b_1 \ge s_1$, then $Q_n(r)>0$ in $(b_1, \overline b_1)$ by the definition of $s_i$. Consequently, $B_0'(r)=-2F(u)Q_n(r)/(ru'^2)>0$} in $(b_1, \overline b_1)$, {\clb and $\dsp Q(\overline b_1)=B_0(\overline b_1)>B_0(b_1)=Q(b_1)$. If $b_1<s_1$, then $b_1<\tau_1$; thus
$Q_n(r)>Q_1(r)>0$ by Lemma \ref{phasel1}, and so $B_0'(r)>0$ in $(\tau_1, \overline b_1)$. As
$Q'(r)=2r^{n-1}f(u)v>0$ in $(b_1, \tau_1)$, we obtain
$\dsp Q(\overline b_1)=B_0(\overline b_1)>B_0(\tau_1)=Q(\tau_1)>Q(b_1)$ and confirm the claim.}

Since $u(r)$ is monotone in $(c_{i-1}, c_i)$ and $|u(c_{i-1})|>|u(c_i)|>\alpha_*$ by
 Prop.\,\ref{basic11} (iii), $r_{i\mu}$ and $\overline r_{i\mu}$ are well-defined. {\clb As $\mu \uparrow |u(c_i)|$,} one has $\overline r_{i\mu}\uparrow c_i$ and $\varphi(\overline r_{i\mu}) \to \infty$. Hence \equ{var} {\clb holds trivially for $\mu$} sufficiently close to $|u(c_i)|$. When $\mu=\alpha_*$, we have
$r_{i\mu}=r_{i\alpha_*}=b_i$ and $\overline r_{i\mu}=\overline r_{i\alpha_*}=\overline b_i$.
From \equ{H} we see that $\widehat E(r)$ decreases in $(b_i, \overline b_i)$ and $\widehat E(\overline b_i)<\widehat E(b_i)$.
Because $F(u(b_i))=F(u(\overline b_i))=0$, we obtain $\overline b_i^{n-1}|u'(\overline b_i)|<{b_i^{n-1}|u'(b_i)|}$. By combining this
{\clb with \equ{qrna}}, we see clearly that
$\varphi(\overline b_i)>\varphi(b_i)$ and verify \equ{var} at $\mu=\alpha_*$.

{\clb To proceed, we} claim that, for $\mu \in [\alpha_*, |u(c_i)|)$,
\bl{var2}  \varphi(\overline r_{i\mu})\ge \varphi(r_{i\mu})
\quad \Rightarrow \quad \varphi_2(\overline r_{i\mu})> \varphi_2(r_{i\mu}), \quad \varphi_2(r):=\frac {Q_2(r)}{r^{n-1}{|u'^3(r)}|}.\ee
To prove this, assume first that $r_{i\mu} \le \tau_i$, {\clb where $\tau_i\in (c_{i-1}, r_i)$ is the unique zero of $v$ on $[c_{i-1}, c_i]$ assured by Lemma \ref{b14}.} Then
$\dsp Q_2(r_{i\mu})=Q(r_{i\mu})+2r_{i\mu}^{n-1}u'v\le Q(r_{i\mu})$, but $\dsp Q_2(\overline r_{i\mu})> Q(\overline r_{i\mu})$ because $\overline r_{i\mu}>\tau_i$. By using the basic property $|u'(\overline r_{i\mu})|<|u'(r_{i\mu})|$ from Prop.\,\ref{basic1}\,(ii), together with $\varphi(\overline r_{i\mu})\ge \varphi(r_{i\mu})$, we derive
\ben \varphi_2(\overline r_{i\mu})&=&\frac {Q_2(\overline r_{i\mu})}{\overline r_{i\mu}^{n-1}{|u'^3(\overline r_{i\mu})}|}> \frac {Q(\overline r_{i\mu})}{\overline r_{i\mu}^{n-1}{|u'^3(\overline r_{i\mu})}|}=\frac{\varphi(\overline r_{i\mu})}{u'^2(\overline r_{i\mu})}
\\&>&\frac{\varphi(r_{i\mu})}{u'^2(r_{i\mu})}
= \frac {Q(r_{i\mu})}{r_{i\mu}^{n-1}{|u'^3(r_{i\mu})}|}
\ge \frac {Q_2(r_{i\mu})}{r_{i\mu}^{n-1}{|u'^3(r_{i\mu})}|}=\varphi_2(r_{i\mu})\een
and verify \equ{var2}. If $r_{i\mu}> \tau_i$, then $v/u'>0$ in $[r_{i\mu}, c_i)$. {\clb Since
$Q_1(r)>0$ in $(c_{i-1}, c_i)$,
\equ{dvup} yields  $(v/u')'=Q_1(r)/(r^nu'^2)>0$,} which implies that, in particular,
\ben \frac {v(\overline r_{i\mu})}{u'(\overline r_{i\mu})}>\frac {v(r_{i\mu})}{u'(r_{i\mu})}\quad{\rm and}\quad \frac {|v(\overline r_{i\mu})|}{u'^2(\overline r_{i\mu})}>\frac {|v(r_{i\mu})|}{u'^2(r_{i\mu})}.\een
{\clb Therefore, \equ{var2} is confirmed again through the following calculation:}
\ben \varphi_2(\overline r_{i\mu})&=&\frac {Q_2(\overline r_{i\mu})}{\overline r_{i\mu}^{n-1}{|u'^3(\overline r_{i\mu})}|}= \frac {Q(\overline r_{i\mu})}{\overline r_{i\mu}^{n-1}{|u'^3(\overline r_{i\mu})}|}+\frac{2|v(\overline r_{i\mu})|}{{u'^2(\overline r_{i\mu})}}\\
 &>& \frac {Q(r_{i\mu})}{r_{iu}^{n-1}{|u'^3(r_{i\mu})}|}+\frac{2|v(r_{
 i\mu})|}{{u'^2(r_{i\mu})}}=\varphi_2(r_{i\mu}).\een

{\clb We use the identity $\left[Q(r)/(r^{n-1}u')\right]'=f(u)Q_2(r)/(r^{n-1}u'^2)$ in \equ{varp} to complete the proof of \equ{var}. If $i$ is odd,} then $u,\, u',\, f(u)<0$ in $(\overline b_i, c_i)$.
By applying \equ{varp} and noticing that $f(-u)=-f(u)$,
we find that, for each $\mu \in [\alpha_*, |u(c_i)|)$,
\bl{varpo} \varphi(\overline r_{i\mu})-\varphi(\overline b_i)&=& -\int_{\overline b_i}^{\overline r_{i\mu}}f(u)\cdot\frac {Q_2(r)}{r^{n-1}{u'^2(r)}}\, dr
= -\int_{-\alpha_*}^{-\mu}f(u)\cdot\frac {Q_2(\overline r_{i(-u)})}{\overline r_{i(-u)}^{n-1}{u'^3(\overline r_{i(-u)})}}\, du  \nonumber \\
&=& -\int_{\alpha_*}^{\mu}f(-u)\cdot\frac {Q_2(\overline r_{iu})}{\overline r_{iu}^{n-1}|u'^3(\overline r_{iu})|}\, du=\int_{\alpha_*}^{\mu}f(u)\cdot \varphi_2(\overline r_{iu})\, du. \ee
If $i$ is even, then $u,\, u',\, f(u)>0$ in $(\overline b_i, c_i)$. By repeating the calculation above, it is simpler to see that $\varphi(\overline r_{i\mu})-\varphi(\overline b_i)$ can also be expressed by the last integral in \equ{varpo}. Over the interval $(c_{i-1}, b_i)$, we have $u,\, f(u)>0$ and $u'<0$ for $i$ odd, and $u,\, f(u)<0$ and $u'>0$ for $i$ even.
By using a calculation similar to \equ{varpo}, we derive
\ben \varphi(r_{i\mu})-\varphi(b_i)=\int_{\alpha_*}^{\mu}f(u)\varphi_2(r_{iu})\, du \een
By subtraction, we conclude {\clb that}
\bl{vava} \varphi(\overline r_{i\mu})-\varphi(r_{i\mu})=\varphi(\overline b_i)-\varphi(b_i)+\int_{\alpha_*}^{\mu}f(u)\left[\varphi_2(\overline r_{iu})-\varphi_2(r_{iu})\right]\, du.\ee
Since $\varphi(\overline b_i)>\varphi(b_i)$, we find that $\varphi(\overline r_{i\mu})>\varphi(r_{i\mu})$ by continuity when $\mu$ is only slightly larger than $\alpha_*$. {\clb Since $f(u)>0$} for $u>\alpha_*$, it is clear from \equ{var2} and \equ{vava} that \equ{var} can be extended to all $\mu \in [\alpha_*, |u(c_i)|)$. \end{proof}

\subsection{{\clb Positivity of $T_2(c_1)$ and conditional positivity of
$T_2(c_i)$ in later phases}}

\begin{lemma}\label{t2c1} For a given $1<i\le k$, if \equ{qmcq1} and \equ{qrna} hold and
\bl{qrnb} T_2(r)>0\quad {\rm on}\quad [c_{i-1}, b_i]; \qquad Q_2(r)>0 \quad {\rm on}\quad [c_{i-1}, c_i],\quad i>1,  \ee
then $T_2(r)>0$ on $[\overline b_i, c_i]$. In particular, $T_2(r)>0$ on $[\overline b_1, c_1]$.
\end{lemma}

\begin{proof} {\clb By Lemma \ref{rcs} and the expression of $T_2$ in \equ{t2},} we see that
$T_2(r)>T_1(r)>0$ for all $r\in (0, z_1)$. By combining this observation with Lemma \ref{phasel1} and Lemma \ref{qrn}, {\clb we conclude that all assumptions for $i>1$ in this lemma} remain valid for $i=1$, except that $M$, $Q$, $Q_1$, $Q_2$, and $T_2$ {\clb all vanish at $c_0=0$, whereas they are assumed to be positive at $c_{i-1}$ when $i>1$}. {\clb We present a unified approach to establish the conditional positivity of $T_2(r)$ on $[\overline b_i, c_i]$ for $1<i\le k$, and its unconditional positivity on $[\overline b_1, c_1]$.  As we proceed, it will become clear that the minor discrepancies at $c_{i-1}>0$ for $i>1$ and at $c_0=0$ do not affect our unified approach.}

{\clb Lemma \ref{b14} shows that $v(r)$ has a unique zero $\tau_i\in (c_{i-1}, r_i)$ in Phase $i$.
Consequently, $uv>0$ in $(z_i, c_i]$. Using \equ{t2p}, we then obtain $T_2'(c_i)>0$.}
Since $F(u(\overline b_i))=0$, we have $T_2(\overline b_i)=Q(\overline b_i)>0$. If $T_2$ has no critical points in $(\overline b_i, c_i)$, then $T_2'(r)>0$ in $(\overline b_i, c_i)$ and this lemma follows from $T_2(\overline b_i)>0$ at once.

It remains to deal with the {\clb difficult} case that $T_2$ has critical points in $(\overline b_i, c_i)$. {\clb We first show that $T_2$ can have at most one critical point in this interval. To prove this, we make use of the following identity established in \equ{t22A}:}
\bl{t22} T_2(r)=B_0(r)+\frac {2F(u)}{(p-1)uu'}\cdot T_2'(r).\ee
{\clb Suppose that $\overline t_i$ is an arbitrary} critical point of $T_2$ in
$(\overline b_i, c_i)$. {\clb By differentiating \equ{t22} at $r=\overline t_i$ and using  $B'_0(r)=-2F(u)\cdot \varphi_n(r)$, we obtain}
\bl{fpu} \frac {2F(u)}{(p-1)uu'}\cdot T_2''(\overline t_i)=2F(u)\cdot \varphi_n(\overline t_i)
\quad \Rightarrow \quad T_2''(\overline t_i)=(p-1)uu'\cdot \varphi_n(\overline t_i). \ee
{\clb Since $u'v>0$ in $(\overline b_i, c_i)$, we have $Q_n(\overline t_i)>Q_1(\overline t_i)>0$
and $\varphi_n(\overline t_i)=Q_n(\overline t_i)/(\overline t_i u'^2)>0$. This, together with
$uu'>0$ in $(\overline b_i, c_i)$, implies that $\dsp T_2''(\overline t_i)>0$.} Thus $T_2$ minimizes at $\overline t_i$, and $\overline t_i$ is
its only critical point in $(\overline b_i, c_i)$.

It suffices to prove $T_2(\overline t_i)>0$ under the condition that
 \bl{ot1} T_2'(r)<0\,\,\,{\rm in}\,\,\, (\overline b_i, \overline t_i), \,\,\, T_2'(\overline t_i)=0, \,\,\,{\rm and}\,\,\,
T_2'(r)>0\,\,\, {\rm in}\,\,\, (\overline t_i, c_i).\ee
Since $|u(c_{i-1})|>|u(c_i)|$ and $u$ is monotone in Phase $i$, there is a unique $t_i\in (c_{i-1}, b_i)$ such that $u(t_i)=-u(\overline t_i)$. {\clb As $T_2(t_i)>0$, our goal is achieved by proving $T_2(\overline t_i)>T_2(t_i)$. Because $T_2(r)$ is undefined at $z_i$, we turn once more to the bridging function $B_a(r)$ for assistance.}
Corresponding to $g_2(u)$ defined in \equ{g2}, let
\bl{a1} a_i=g_2(u(t_i))=g_2(u(\overline t_i)){\clb =\frac{2F(u( t_i))}{h_1(u( t_i))}}
=\frac {p+1}{p-1}\cdot \frac {2F(u( t_i))}{|u(t_i)|^{p+1}}>0.\ee
{\clb By the definition $F_a(u)=F(u)- a h_1(u)/2$, we find easily that
\ben F_{a_i}(u(t_i))= F(u(t_i))- a_i h_1(u(t_i))/2=0  \een
and $F_{a_i}(u(\overline t_i))=0$ as well. Consequently, by the definitions \equ{ba} and \equ{g2},
\bl{wt2f} T_2( t_i)=B_{a_i}(t_i)=Q(t_i)-a_iM(t_i),\ee
and the same holds after replacing $t_i$ by $\overline t_i$.
For our model case, we have
\bl{fau} F_a(u):=-\frac{u^2}2+\frac 1{p+1}\left({1-\frac{p-1}2\, a}\right)|u|^{p+1}.\ee
Substituting $a=a_i$ into this exact form yields
\bl{fai} H_i(u):=-2F_{a_i}(u)={u^2}\left(1-\left|\frac {u}{u(t_i)}\right|^{p-1}\right). \ee}
Integrating \equ{bap} over
$(t_i, \overline t_i)$ and applying \equ{wt2f} and \equ{fai}, we obtain
\bl{t2t12} T_2(\overline t_i)=T_2(t_i)+\int_{t_i}^{\overline t_i} H_i(u(r))\cdot \varphi_n(r) \, dr>\int_{t_i}^{\overline t_i} H_i(u(r))\cdot \varphi_n(r) \, dr. \ee

{\clb It remains to establish the positivity of the integral in \equ{t2t12}.}
If $t_i\ge s_i$, then $Q_n, \,\varphi_n>0$ in $(t_i, \overline t_i)$ {\clb by
\equ{db1}, and the positivity of the integral follows} at once. {\clb Since
$s_i<\tau_i$ always holds} and \equ{ot1} specifies $\overline t_i>\overline b_i$, or equivalently $t_i<b_i$, {\clb our subsequent considerations are restricted to the case}
\bl{condtsb} t_i<s_i<\tau_i, \quad t_i<b_i. \ee
Denote by $r_{i\mu}\in (t_i, b_i]$ and $\overline r_{i\mu}\in [\overline b_i, \overline t_i)$ the numbers satisfying $|u( r_{i\mu})|=| u(\overline r_{i\mu})|=\mu$ for each $\mu \in [\alpha_*, |u(t_i)|)$. Recall from Prop.\,\ref{basic2}\,(iii) that $P(r)/r^{n}>0$ is universally decreasing. Applying Lemma \ref{qrn} leads to, {\clb for each $\mu\in (\alpha_*, |u(t_i)|)$,}
 \bl{sc31} && \frac{n \overline r^{n-2}_{i\mu}Q(\overline r_{i\mu})}{P(\overline r_{i\mu})u'^2(\overline r_{i\mu})}
=\frac{n\overline r_{i\mu}^{n-3}\varphi(\overline r_{i\mu})}{|u'(\overline r_{i\mu})|}\cdot \frac {\overline r_{i\mu}^{n}}{P(\overline r_{i\mu})} \nonumber\\
&>&\frac{(n-2)r_{i\mu}^{n-3}\varphi(r_{i\mu})}{|u'(r_{i\mu})|}\cdot \frac {r_{i\mu}^{n}}{P(r_{i\mu})} =\frac{(n-2)r_{i\mu}^{n-2}Q(r_{i\mu})}{P(r_{i\mu})u'^2(r_{i\mu})}. \ee
As $Q_2(r)>0$ in $(c_{i-1}, c_i]$ by \equ{qrnb},
we find that
\bl{up1}  -Q_n(r)=-Q_2(r)-(n-2)r^{n-1}u'v\le (n-2)r^{n-1}|u'v|,\quad r\in (t_i, s_i).\ee
As $Q(r)>0$ on $[\overline b_i, c_i]$ by \equ{qrna} and $u'v>0$ in $(\tau_i, c_i)$, {\clb we have}
\bl{up2} Q_n(r)=Q(r)+nr^{n-1}u'v>nr^{n-1}u'v=nr^{n-1}|u'v|,\quad r\in (\overline b_i, \overline t_i).\ee
{\clb In addition, we use the connection identity \equ{conn} to show that}
\bl{vqp} |v|<\frac {Q(r)}{P(r)}\cdot |u|\quad{\rm for}\quad r\in (t_i, \tau_i), \qquad
|v|>\frac {Q(r)}{P(r)}\cdot |u|\quad{\rm for}\quad r\in (\overline b_i, \overline t_i). \ee
Indeed, {\clb since $u\cdot v,\, \omega,\, M>0$ and $\va<0$ in $(c_{i-1}, \tau_i)$, \equ{conn} implies $Q(r)>P(r)v/u$ in $(c_{i-1}, \tau_i)$ and confirms the first part of \equ{vqp}.
Over the interval $(\overline b_i, \overline t_i)$, we have} $uv, \, uu'>0$ and $\omega<0$, {\clb and $T_2'(r)<0$ by \equ{ot1}; hence \equ{t2p2} implies $M(r)>\va(r)$,} and \equ{conn} implies $Q(r)<P(r)v/u$, by which the second part of \equ{vqp} is confirmed.

{\clb In the rest}, we consider the cases $s_i\le b_i$ and $s_i>b_i$ separately. {\clb Define $\overline s_i\in (z_i, c_i)$ by} $u(\overline s_i)=-u(s_i)$. Then we have either $\overline s_i\ge \overline b_i$ or $\overline s_i< \overline b_i$ in the two cases. If $s_i\le b_i$, then \equ{condtsb} implies that $t_i<s_i\le b_i$ and $\overline b_i\le \overline s_i<\overline t_i$. For each $\mu \in (|u(s_i)|, |u(t_i)|)$, we apply the estimates in \equ{sc31}--\equ{vqp} to derive
\ben && \frac{\varphi_n(\overline r_{i\mu})}{|u'(\overline r_{i\mu})|}=\frac{Q_n(\overline r_{i\mu})}{\overline r_{i\mu}|u'(\overline r_{i\mu})|^3}
> \frac{n \overline r^{n-2}_{i\mu}|v(\overline r_{i\mu})|}{u'^2(\overline r_{i\mu})} \qquad {\rm by} \,\, \equ{dvarn},\,\equ{up2} \nonumber \\
&>&\frac{n \mu \overline r^{n-2}_{i\mu}Q(\overline r_{i\mu})}{P(\overline r_{i\mu})u'^2(\overline r_{i\mu})}
>\frac{(n-2)\mu r_{i\mu}^{n-2}Q(r_{i\mu})}{P(r_{i\mu})u'^2(r_{i\mu})} \,\qquad \quad{\rm by} \,\, \equ{sc31},\,\equ{vqp}   \nonumber \\
&>& \frac{(n-2)r_{i\mu}^{n-2}|v(r_{i\mu})|}{u'^2(r_{i\mu})}
>-\frac{\varphi_n(r_{i\mu})}{|u'(r_{i\mu})|} \,\,\qquad \qquad \quad  {\rm by} \,\, \equ{up1},\,\equ{vqp}. \nonumber \een
Note carefully that \equ{sc31}--\equ{vqp} are applicable in these steps because $\mu \in (|u(s_i)|, |u(t_i)|)$ implies that $r_{i\mu}\in (t_i, s_i)$ and $\overline r_{i\mu}\in (\overline s_i, \overline t_i)$, whereas
$\dsp (t_i, s_i)\subseteq (t_i, b_i)$, $\dsp (t_i, s_i)\subset(t_i, \tau_i)$, and $\dsp (\overline s_i, \overline t_i)\subseteq (\overline b_i, \overline t_i)$. {\clb Therefore,
 for each $\mu \in (|u(s_i)|, |u(t_i)|)$, we obtain}
 \bl{Psi} \Psi_i(\mu):=\frac{\varphi_n(\overline r_{i\mu})}{|u'(\overline r_{i\mu})|}+\frac{\varphi_n(r_{i\mu})}{|u'(r_{i\mu})|}>0.\ee
{\clb If $i$ odd,} then $u>0,\, u'<0$ in $(t_i, s_i)$, and $u,\, u'<0$ in $(\overline s_i, \overline t_i)$. As $\varphi_n(r)>0$ in $(s_i, c_i]$, {\clb and $H_i(u)\ge 0$
for $u\in (-|u(t_i)|, |u(t_i)|)$,} we continue from \equ{t2t12} to {\clb estimate}
\bl{t2ovt1} T_2(\overline t_i)&>& \int_{t_i}^{s_i}H_i(u(r))\cdot \varphi_n(r) \, dr+\int_{\overline s_i}^{\overline t_i} H_i(u(r))\cdot \varphi_n(r) \, dr   \nonumber \\
&=& \int_{u(t_i)}^{u(s_i)}H_i(u)\cdot \frac{\varphi_n(r_{iu})}{u'(r_{iu})}\, du
+\int_{u(\overline s_i)}^{ u(\overline t_i)} H_i(u)\cdot \frac{\varphi_n(\overline r_{i(-u)})}{u'(\overline r_{i(-u)})}\, du \nonumber\\
&=& \int_{|u(s_i)|}^{|u(t_i)|}H_i(u)\cdot \frac{\varphi_n(r_{iu})}{|u'(r_{iu})|}\, du
+ \int_{|u(s_i)|}^{|u(t_i)|} H_i(u)\cdot \frac{\varphi_n(\overline r_{iu})}{|u'(\overline r_{iu})|}\, du \nonumber\\
&=& \int_{|u(s_i)|}^{|u(t_i)|} H_i(u) \cdot \Psi_i(u)\, du>0.\ee
If $i$ is even, then $u<0,\, u'>0$ in $(t_i, s_i)$, and $u,\, u'>0$ in $(\overline s_i, \overline t_i)$. By repeating the calculations in \equ{t2ovt1}, with the only minor modification of replacing $r_{iu}$ by $r_{i(-u)}$ and $\overline r_{i(-u)}$ by $\overline r_{iu}$ in the second line, {\clb we derive again that $T_2(\overline t_i)>0$.}

 {\clb There remains the case $s_i>b_i$. By Lemmas \ref {b14}, this} may occur only {\clb if} $n=3$ and $1<p<2$. {\clb Since $|u(s_i)|<\alpha_*$, while Lemma \ref{qrn} requires $\mu\ge \alpha_*$, we can claim  $\Psi_i(\mu)>0$ for $\mu\in [|u(b_i)|, |u(t_i)|]$, but not for $\mu \in (|u(s_i)|, |u(b_i)|)$.} Write $[t_i, \overline t_i]=[t_i, b_i]\cup (b_i, \tau_i) \cup [\tau_i, \overline b_i)\cup [\overline b_i, \overline t_i]$. It is easy to see that $H_i(u(r)), \, \varphi_n(r)>0$ in $[\tau_i, \overline t_i]$ {\clb (except $H_i(u(z_i))=0$)}. According to Lemmas \ref {b13},
$H_i(u(r))\cdot \varphi_n(r)$ has a positive integral over $(b_i, \tau_i)$. As a result,
replacing $s_i$ by $b_i$ and $\overline s_i$ by $\overline b_i$ in the process leading to \equ{t2ovt1} yields
{\clb $T_2(\overline t_i)>\int_{|u(b_i)|}^{|u(t_i)|}\, H_i(u)\cdot \Psi_i(u)\, du>0$, which
completes our proof.} \end{proof}

\subsection{Proof of the phase transition lemma}
In this {\clb subsection}, we complete the proof {\clb Lemma \ref{ptl}.
In view of Lemmas \ref{phasel1} and \ref{t2c1}, we see that $Q(c_{1}), M(c_{1}), T_2(c_{1})>0$. We need to confirm the statements (i), (ii), and (iii) under the assumption that
\bl{qmt2i1} Q(c_{i-1}), \, M(c_{i-1}),\, T_2(c_{i-1})>0\,\,{\rm for\,\, some}\,\,
i\in \{2, \cdots, k+1\}.
\ee
According to Lemma \ref{b11}, $v(r)$ changes sign in $(c_{i-1}, r_i)$. Let $\tau_i$ be its first zero in this interval. Then the signs of various quantities in $(c_{i-1}, \tau_i)$ are fixed in \equ{cip}. In particular, $uv,\, f(u)v>0$ in $(c_{i-1}, \tau_i)$. Thus $Q'(r), M'(r)>0$ and}
\bl{c42} Q(r)>Q(c_{i-1})>0,\quad M(r)>M(c_{i-1})>0, \quad r\in (c_{i-1}, \tau_i].\ee
We continue our proof in {\clb five} steps.

{\it Step 1}. As the first major step, we use the connection identity to prove
\bl{tmono} T_2'(r),\, T_2(r)>0, \quad r\in [c_{i-1}, \max\{\tau_i, b_i\}], \quad {\clb i\in \{2, \cdots, k+1\}.}\ee
{\clb When $i=k+1$ and $u$ is a bound state, $b_{k+1}$ is uniquely determined by $b_{k+1}>c_k$ and $|u(b_{k+1})|=\al_*$.} Substituting $u'(c_{i-1})=0$ into \equ{t2p} gives {\clb $T_2'(c_{i-1})>0$.}
{\clb Since $uv\ge 0$, $M>0$, and $uu'<0$ in $(c_{i-1}, \tau_i]$, applying \equ{t2p} once again we conclude that $T_2'(r)>0$, and $T_2(r)\ge T_2(c_{i-1})>0$, for all $r\in [c_{i-1}, \tau_i]$.
This confirms \equ{tmono} when $\tau_i\ge b_i$.}

It is highly non-trivial to prove \equ{tmono} in the other case that $\tau_i< b_i$. The hardest part is to rule out the possibility that $v$ has zeros within $(\tau_i, b_i]$. Because $uv<0$ in a right neighborhood of $\tau_i$, $v$ will not have any zero in $(\tau_i, b_i]$ if we can prove that
\bl{stv} \tau_i<b_i \qquad \Rightarrow \qquad uv<0\quad{\rm in}\quad (\tau_i, b_i].\ee
Suppose for contradiction that \equ{stv} is not true. Then there exists $\widetilde \tau_i$
such that
\bl{vt4} \widetilde \tau_i \in (\tau_i, b_i], \qquad v(\tau_i)=v(\widetilde \tau_i)=0, \qquad uv<0\quad {\rm in}\quad (\tau_i, \widetilde \tau_i).\ee
By using the definitions of $Q(r)$, $M(r)$, $T_2(r)$, and $\omega(r)$, we derive the expression
\bl{wt2} T_2(\tau_i)&=& Q(\tau_i)-g_2(u(\tau_i))M(\tau_i)
= \tau_i^n u'v'+\tau_i^{n-1} g_2(u(\tau_i))uv'  \nonumber \\
&=& \tau_i^{n-1} u(\tau_i)v'(\tau_i)\left[g_2(u(\tau_i))-\omega(\tau_i)\right]
\ee
and the same for $T_2(\widetilde\tau_i)$ by replacing $\tau_i$ with $\widetilde\tau_i$. As {\clb $u(\tau_i)v'(\tau_i)<0$ and $T_2(\tau_i)>0$, \equ{wt2} implies}
$g_2(u(\tau_i))<\omega(\tau_i)$. {\clb Since $g_2(u(r))$ decreases in $(\tau_i, \widetilde\tau_i)$ by \equ{g2}, whereas $\omega(r)$ increases in this interval by Prop.\,\ref{basic2}\,(ii),} it follows that $g_2(u(\widetilde\tau_i))<\omega(\widetilde\tau_i)$.
{\clb As $v'$ assumes} opposite signs at $\tau_i$ and $\widetilde \tau_i$, we obtain
$u( \widetilde\tau_i)v'( \widetilde\tau_i)>0$. {\clb By replacing $\tau_i$ with $\widetilde \tau_i$ in
\equ{wt2}, we find} $\dsp T_2(\widetilde \tau_i)<0$. Since $T_2,\, T_2'>0$ on $[c_{i-1}, \tau_i]$, $T_2$ has a critical point within $(\tau_i, \widetilde\tau_i)$, and there exists $\tau_i^*$ such that
\bl{t2s} \tau_i^*\in (\tau_i, \widetilde\tau_i), \quad T_2(r),\, T_2'(r)>0 \,\, {\rm in}\,\, [c_{i-1}, \tau_i^*), \quad T_2'(\tau_i^*)=0.\ee
{\clb Since $u'(\tau_i^*)v(\tau_i^*)>0$, it follows from \equ{conn} that} $\va(\tau_i^*)>0$. Consequently, \equ{t2p2} gives $M(\tau_i^*)=\va(\tau_i^*)>0$. Furthermore, since $uv<0$ at $\tau_i^*$ and $P(r)$ is always positive, the connection identity \equ{conn} leads to the striking relation
\bl{qta4} Q(\tau_i^*)=P(\tau_i^*)\cdot {v(\tau_i^*)}/{u(\tau_i^*)}<0.\ee
As \equ{t2s} clearly implies that $T_2(\tau_i^*)>0$, we deduce from $M(\tau_i^*)>0$ and
\equ{g2} that
\bl{g2ta4} g_2(u(\tau_i^*))M(\tau_i^*)=Q(\tau_i^*)-T_2(\tau_i^*)<0 \quad \Rightarrow\quad g_2(u(\tau_i^*))<0.\ee
{\clb Thus $F(u(\tau_i^*))<0$ by the definition of $g_2$}. However,
{\clb $\tau_i^*\in (\tau_i, b_i)$ by our selection of $\tau_i^*$. It leads to
$F(u(\tau_i^*))>0$ and a contradiction. This confirms \equ{stv}.}

{\clb Now we prove \equ{tmono}. If it were not true, then there would be} a critical point of $T_2$,  again denoted by $\tau_i^*$, such that
\bl{t2si} \tau_i^*\in (\tau_i, b_i], \quad T_2,\, T_2'>0 \,\, {\rm in}\,\, [c_{i-1}, \tau_i^*), \quad T_2'(\tau_i^*)=0, \quad uv<0 \,\, {\rm in}\,\, (\tau_i, \tau_i^*].\ee
{\clb Since $u'(\tau_i^*)v(\tau_i^*)>0$, we find again that $\va(\tau_i^*)>0$,
$M(\tau_i^*)=\va(\tau_i^*)>0$, and $Q(\tau_i^*)<0$. Thus, \equ{g2ta4} gives
$g_2(u(\tau_i^*))<0$, which in turn implies} $F(u(\tau_i^*))<0$. It contradicts $\tau_i^*\in (\tau_i, b_i]$ and completes {\clb our} proof of \equ{tmono}.

{\it Step 2}. {\clb Based on \equ{tmono} and \equ{stv}, we provide a straightforward proof of}
\bl{qq12} Q(r), \,\, Q_1(r),\, \, Q_2(r)>0, \quad r\in [c_{i-1}, \max\{b_i, \tau_i\}], \quad {\clb i\in \{2, \cdots, k+1\}}.\ee
At $r=c_{i-1}$, the three $Q$ family functions are {\clb equal and positive. Moreover,  they all increase} in $(c_{i-1}, \tau_i)$; {\clb thus \equ{qq12} is verified when $b_i\le \tau_i$.} If $b_i> \tau_i$, {\clb then \equ{stv} holds, which implies that} $f(u)v,\, Q'(r)<0$ and $u'v>0$ in $(\tau_i, b_i]$. By using  \equ{tmono} and the definitions of $g_2$ and $T_2$, we derive
\bl{qs4} Q(b_i)=T_2(b_i)+g_2(u(b_i))M(b_i)=T_2(b_i)>0.\ee
Over the interval $[c_{i-1}, b_i]$, $Q(r)$ increases from $Q(c_{i-1})>0$ to reach a local maximum at $\tau_i$ and then {\clb decreases} to $Q(b_i)>0$. As a result, we find that $Q_2>Q_1>Q>0$ in $(\tau_i, b_i]$ and complete the proof of \equ{qq12}.

{\it Step 3}. {\clb We make use of the bridging function $B_0(r)$ to prove that
\begin{equation}\label{qi+}
\begin{cases} Q_1>0\,\, {\rm on}\,\, [\tau_i, \overline b_i] \,\, {\rm if}\,\, i\in \{2, \cdots, k\};\\
\,  Q_1>0\,\, {\rm on}\,\, [\tau_{k+1}, z_{k+1}]\,\,{\rm if}\,\,i=k+1\,\,{\rm and}\,\, u{\rm \,\, is\,\, not\,\, a\,\, bound\,\, state;}\\
\, Q_1>0\,\, {\rm in}\,\, [\tau_{k+1}, \infty) \,\,{\rm if}\,\,i=k+1\,\,{\rm and}\,\, u{\rm \,\, is\,\, a\,\, bound\,\, state.}
\end{cases}
\end{equation}
Since $u'v'<0$ and $Q_1<0$ at the first possible zero of $v$ in $(\tau_i, \overline b_i]$ for
$i\in \{2, \cdots, k\}$, and in $(\tau_{k+1}, z_{k+1}]$ or $(\tau_{k+1}, \infty)$ for $i=k+1$,
\equ{qi+} implies that $v$ has no zeros in these intervals.} Suppose for contradiction that \equ{qi+} is false and let $\widetilde q_i$ be the first zero of $Q_1$ {\clb in each of the specified intervals.  Then
$\widetilde q_i>\max\{b_i, \tau_i\}$ by {\it Step 2} and}
\bl{tqi} u'v>0\,\,\,{\rm in}\,\,\, (\tau_i, \widetilde q_i], \quad
Q_n(r)\ge Q_1(r)>0\,\,\,{\rm in}\,\,\, [\tau_i, \widetilde q_i), \quad Q_1(\widetilde q_i)=0.\ee
If $b_i\le \tau_i$, {\clb then $F(u(r))\cdot Q_n(r)<0$ and $B'_0(r)>0$} in $(\tau_i, \widetilde q_i)$. {\clb Using \equ{qq12}, we obtain}
\bl{bqzi} B_0(\widetilde q_i)>B_0(\tau_i)=Q(\tau_i)>0.\ee
Similarly, if $b_i>\tau_i$, then {\clb $B'_0(r)>0$} in $(b_i, \widetilde q_i)$ and
{\clb \equ{qq12} gives}
\bl{bqzii} B_0(\widetilde q_i)>B_0(b_i)=Q(b_i)>0.\ee
However, {\clb by using \equ{tqi} and \equ{bobo}, we find
\ben B_0(\widetilde q_i)
=-\,\widetilde q_i^{n-1}u'v-2F(u)\cdot \frac{\widetilde q_i^{n-1}v}{u'}= -\frac{2\widetilde q_i^{n-1}v(\widetilde q_i)}{u'(\widetilde q_i)}\cdot E(\widetilde q_i)<0, \een}
which contradicts \equ{bqzi} and \equ{bqzii}. {\clb This confirms \equ{qi+}}.

{\it Step 4}. We continue from \equ{qi+} to verify{\footnote{
Define $r_{i\mu}\in (c_{i-1}, b_i]$ and $\overline r_{i\mu}\in [\overline b_i, c_i)$ by $|u( r_{i\mu})|=| u(\overline r_{i\mu})|=\mu$ as in Lemma \ref{qrn}. Then we may rewrite \equ{qb4p} as
$\dsp Q(\overline r_{i\alpha_*})>Q(r_{i\alpha_*})$. Define $\dsp \overline B_{\mu}(r):=Q(r)-2\left[F(u)-F(\mu)\right]\cdot {r^{n-1}v}/{u'}$.
Then $\dsp \overline B'_{\mu}(r)=-2\left[F(u)-F(\mu)\right]\cdot {Q_n(r)}/{ru'^2}$, from which we may deduce that, for each $\mu\in (0, 1]$, $\dsp Q(\overline r_{i\mu})=\overline B_{\mu}(\overline r_{i\mu})<\overline B_{\mu}(r_{i\mu})=Q(r_{i\mu})$. {\clb This stands in sharp contrast to \equ{qb4p} and highlights the subtle nature of $Q(r)$ within the ``river" $(b_i, \overline b_i)$.} }}
\bl{qb4p} Q(\overline b_i)>Q(b_i)>0,\quad {\clb i\in \{2, \cdots, k\}}.\ee
{\clb From {\it Step 3} we see that $u'v>0$ in $(\tau_i, \overline b_i]$. If $b_i\ge \tau_i$, then $Q_n(r)>Q_1(r)>0$ and $B'_0(r)>0$ for $r\in (b_i, \overline b_i]$. As a result, $Q(\overline b_i)=B_0(\overline b_i)>B_0(b_i)=Q(b_i)$.
If $b_i<\tau_i$, then $B'_0(r)>0$ in $(\tau_i, \overline b_i)$ and we find similarly that $Q(\overline b_i)>Q(\tau_i)$. In addition, since $Q'(r)>0$ in $(b_i, \tau_i)$, it follows that $Q(\overline b_i)>Q(\tau_i)>Q(b_i)$. This, together with \equ{qq12}, confirms \equ{qb4p} completely.}

{\it Step 5. Completion.}

{\clb (i) Assume that $i\in \{2,  \cdots, k\}$ and $Q,\, M,\, T_2>0$ at $c_{i-1}$. Then
$v(r)$ takes its first zero on $[c_{i-1}, c_i]$ at $\tau_i\in (c_{i-1}, r_i)$. As shown in {\it Step 3}, $v$ has no zeros in $(\tau_i, \overline b_i]$.} We show that $v$ has no zeros on $[\overline b_i, c_i]$ {\clb as well}. In fact, if this were false, then there would be $\widetilde \tau_i\in (\overline b_i, c_i]$ such that $u'v>0$ in $(\tau_i, \widetilde \tau_i)$ and $v(\widetilde\tau_i)=0$. As $u'v'>0$ at $\tau_i$, we would have $u'v'<0$ at $\widetilde\tau_i$ and  $Q(\widetilde\tau_i)<0$. However, since {\clb $f(u)v,\, Q'(r)>0$ on $[\overline b_i, \widetilde\tau_i]$, and $Q(\overline b_i)>0$ by {\it Step 4}, we derive} $Q(\widetilde \tau_i)>Q(\overline b_i)>0$ {\clb and a contradiction. As a result,
$v$ has a unique zero $\tau_i$ in $[c_{i-1}, c_i]$ with $\tau_i\in (c_{i-1}, r_i)$.}

{\clb It remains to verify $Q,\, M,\, T_2>0$ at $c_{i}$. Clearly, $f(u)v,\, Q'(r)>0$ on $[\overline b_i, c_i]$, and}
\bl{qupv} u'v>0 \,\,\,{\rm in}\,\,\, (\tau_i, c_i], \quad Q(r)>0 \,\,\,{\rm on}\,\,\, [\overline b_i, c_i]. \ee
{\clb According to {\it Steps 2} and {\it 3}, $Q_1>0$ on $[c_{i-1}, \overline b_i]$.
In addition, $Q_1>Q>0$ on $[\overline b_i, c_i]$ by \equ{qupv}. Hence $Q_1>0$ on
$[c_{i-1}, c_i]$. Note that {\it Step 2} also gives
$Q_2>0$ on $[c_{i-1}, \tau_i]$. Since $u'v>0$ in $(\tau_i, c_i]$, we find
$Q_2>Q_1>0$ in $(\tau_i, c_i]$ and so $Q_2>0$ on $[c_{i-1}, c_i]$. As} $M$ increases in $(c_{i-1}, \tau_i)$ from $M(c_{i-1})>0$, decreases in $(\tau_i, z_i)$ until reaching a local minimum value
{\clb $M(z_i)=z_i^{n-1}u'v>0$,}
and then increases in $(z_i, c_i)$, {\clb we obtain $M>0$ on $[c_{i-1}, c_i]$. Combining these observations with {\it Steps 1-4}, we see that all the assumptions imposed in
Lemmas \ref{b14}-\ref{t2c1} are fully covered}. {\clb In particular,
Lemma \ref{t2c1} applies here, from which we obtain} $T_2(r)>0$ on $[\overline b_i, c_i]$.

 {\clb (ii) This is an immediate consequence of {\it Step 3}.

(iii) Assume that $u$ is a bound state and $Q(c_{k}), M(c_{k}), T_2(c_{k})>0$.
It again follows from {\it Step 3} that $v$ has a unique zero $\tau_{k+1}$ in $[c_k, \infty)$ with $\tau_{k+1}\in (c_{k}, r_{k+1})$.} It remains to show that {\clb $v$ is strictly monotone for $r$ sufficiently large, and} $\lim_{r\to\infty} |v(r)|=\infty$. {\clb In view of {\it Steps 2} and {\it 3}}, it is seen that $Q_1(r)>0$ for all $r\ge c_k$.

We {\clb claim} that $\lim_{r\tin}B_0(r)$ is positive (possibly $\infty$). {\clb Indeed,} if $\tau_{k+1}\ge b_{k+1}$, then
\bl{ed1} 0<|u(r)|< \al_*\,\, \Rightarrow \,\, F(u(r))< 0; \quad u'v>0\,\, \Rightarrow \,\,
Q_n(r)\ge Q_1(r)>0 \ee
for all $r\in (\tau_{k+1}, \infty)$.  Hence $B'_0(r)>0$ in $(\tau_{k+1}, \infty)$ and
\ben \lim_{r\tin}B_0(r)>B_0(\tau_{k+1})=Q(\tau_{k+1})>0.\een
If $\tau_{k+1}<b_{k+1}$, then \equ{ed1} holds for all $r>b_{k+1}$.
Thus $B'_0(r)>0$ in $(b_{k+1}, \infty)$ and
\ben \lim_{r\tin}B_0(r)>B_0(b_{k+1})=Q(b_{k+1})>0.\een
{\clb In either case, the claim is confirmed.}

Since $f'(u(r))<0$ and $v$ does not change sign for $r$ sufficiently large, we see from
\equ{rnvp} that $r^{n-1}v'$ is  eventually monotone, and so $v$ must be strictly monotone. Now we suppose for contradiction that $|v|\not\to \infty$ as $r\tin$. Then $|v|$ approaches a finite constant. Recall from Prop.\,\ref{basic11} (iv) that $u/u' \to -1$ and $|u|\downarrow 0$ exponentially as $r\tin$. It follows easily that $\lim_{r\tin}Q(r)=0$. From the expression
\ben B_0(r)=Q(r)-2F(u)\cdot \frac{r^{n-1}v}{u'} =Q(r)-2r^{n-1}\left(-\frac{|u|}2+\frac{|u|^p}{p+1}\right)\cdot \frac {|u|}{u'}\cdot v, \een
we find that $\lim_{r\tin} B_0(r)=0$. It gives a contradiction and completes the proof. \qed

\subsection{{\clb Verification of the conditions} of Lemma \ref{tail0}}

{\clb We verify conditions (a) and (b) of Lemma \ref{tail0}, while maintaining its basic assumptions that $u(r)=u(r, \al)$ is a solution of \equ{eu} and $\al>\al^*$.

(a) Assume that $u(r)=u(r, \al)$ has exactly $k\ge 1$ zeros $z_1<\cdots<z_k$. We need to verify that $v(r)$ has exactly $k$ zeros $\tau_1<\cdots<\tau_k$ on $[0, z_k]$, with $\tau_1\in (0, z_1)$, and $\tau_i\in (z_{i-1}, z_i)$, $2\le i\le k$; moreover, $u'v'>0$ whenever $v=0$.} When $k=1$, this is covered by Lemma \ref{rcs} (i). Next we consider $k=2$. {\clb By Lemma \ref{phasel1}, $v$ has a unique zero $\tau_1$ in the first phase $[0, c_1]$
and $\tau_1\in (0, r_1)$. By the phase transition lemma,
$Q,\, M,\, T_2>0$ at $c_1$, and $v$ has a unique zero $\tau_2$ on $[c_1, z_2]$ with $\tau_2\in (c_1, r_2)$. Since $r_1<z_1<c_1<r_2<z_2$, it follows that $\tau_2$ is the unique zero of $v$ on $[z_1, z_2]$. As $\tau_2$ is located within $(c_1, r_2)$, it is easily seen that $u'v'>0$ at $\tau_2$. By induction, the same reasoning based on the phase transition lemma extends to all $k\ge 3$.

(b) The condition concerning a ground state is validated by Lemma \ref{rcs} (ii). When
$u$ is a $1$-node bound state, the condition is an immediate consequence of part (iii) of Lemma \ref{ptl}. For $k\ge 2$, one applies an induction argument relying on part (i) of Lemma \ref{ptl} to establish $Q(c_{k}), M(c_{k}), T_2(c_{k})>0$. The rest of the condition then follows directly from part (iii) of Lemma \ref{ptl}.}  \qed

\appendix
\section{{\clb Auxiliary functions and identities}}

\subsection{{\clb Auxiliary functions associated with $f(u)$}} {\clb In Table 1, we list the definitions of several functions associated with $f(u)$, along with their explicit expressions for the model nonlinearity. These expressions are particularly useful in confirming that the conditions on $f$ specified in Sections 2 and 4 are indeed satisfied in the model case.}

\begin{table}[h]
\centering
 \begin{tabular}{|c|c|c|c|c|}
 \hline
Name & Definition & \footnotesize For $ f(u)=-u+u|u|^{p-1},\,\, 1<p<\frac{n+2}{n-2}, \,\, n\ge 3$       \\
\hline
$F(u)$& \footnotesize $\int_0^u f(s)\, ds$&\footnotesize $-\frac{u^2}2+\frac{|u|^{p+1}}{p+1}=\frac{u^2}{p+1}\left(|u|^{p-1}-\alpha_*^{p-1}\right)$ \\
\hline
$f'(u)$ &  & \footnotesize $\dsp -1+p|u|^{p-1}$ \\
\hline
$f''(u)$ &  & \footnotesize $\dsp p(p-1)u|u|^{p-3}$ \\
\hline
$h(u)$ &  \footnotesize $2nF(u)-(n-2)uf(u)$   & \footnotesize  $2u^2\left(|u/\al^*|^{p-1}-1\right)$\\
\hline
$h_1(u)$ & \footnotesize $uf(u)-2F(u)$   & \footnotesize $\frac{p-1}{p+1}|u|^{p+1}$  \\
\hline
$h_2(u)$ & \footnotesize $(n+2)f(u)-(n-2)uf'(u)$   &  \footnotesize $-4u+[n+2-p(n-2)]u|u|^{p-1}$\\
\hline
$g_1(u)$ & \footnotesize ${2f(u)}/[uf'(u)-f(u)]$   & \footnotesize $\frac 2 {p-1}\left(1-|u|^{1-p}\right)$ \\
\hline
$g_2(u)$ & \footnotesize $2F(u)/h_1(u)$   & \footnotesize $\frac 2{p-1}\left(1-\frac {p+1}2|u|^{1-p}\right)$ \\
\hline
$F_a(u)$ & \footnotesize $ F(u)-\frac a2 h_1(u)
$ & \footnotesize $ -\frac{u^2}2+\frac 1{p+1}\left({1-\frac{p-1}2\, a}\right)|u|^{p+1}$ \\
\hline
\end{tabular}
\ms
\caption{{\small {\it {\clb List of auxiliary functions associated with $f(u)$.}} }}  \label{tb0}\vspace*{-10pt} \end{table}

{\clb \subsection{Energy-type functions} We carry out routine computations of the derivatives of various functions and summarize the results in Table 2. By differentiating the energy function $E(r)$ defined in \equ{ef} and using \equ{ble3}, we deduce that
\bl{ep} E'(r)=u'u''+f(u)u'=u'[u''+f(u)]=-\,\frac{n-1}ru^{\prime 2}(r).\ee
Consequently, for $\widehat E(r)=r^{2(n-1)}E(r)$, we calculate
\bl{HA} \widehat E'(r)=2(n-1)r^{2n-3}E(r)-(n-1)r^{2n-3} u^{\prime 2}(r)=2(n-1)r^{2n-3}F(u). \ee

To proceed, we rewrite the equation in \equ{ble3} as
\bl{rnup} (r^{n-1}u')'=r^{n-1}\left(u''+\frac {n-1}r u'\right)=\, -r^{n-1}f(u).\ee
With the help of \equ{ep} and \equ{rnup}, we can verify the well-known Pohozaev identity:
\bl{Pr} P'(r)
&=& 2nr^{n-1}E(r)-2(n-1)r^{n-1}u^{\prime 2}+(n-2)r^{n-1}u'^2-(n-2)r^{n-1}uf(u) \nonumber \\
&=& r^{n-1}\left[2nF(u)-(n-2)uf(u)\right]\\
&=& r^{n-1} h(u).  \nonumber \ee
We may rewrite $P_2$ in the form
\ben P_2(r)=P(r)+r^n\left[\frac {n-2}n uf(u)-2F(u)\right]=P(r)-\frac{r^n}n\, h(u).\een
By applying \equ{Pr} and the relation $h_2(u)=h'(u)$, we calculate
\bl{mmyp2r} P_2'(r)=r^{n-1} h(u)-r^{n-1} h(u)-\frac{r^nu'}n\, h'(u) =-\,\frac{r^nu'}n\, h_2(u).\ee
For the important ratio
\bl{omega} \omega(r)=\, -\frac{ru'(r)}{u(r)},\ee
a direct differentiation yields
\bl{omegap}  \omega'(r)=\frac{u[(n-2)u'+rf(u)]+ru'^2}{u^2}=\frac {P_1(r)}{r^{n-1}u^{2}}.\ee
By using \equ{Pr}, we compute
\bl{prn} && \frac{d}{dr}\frac{P(r)}{r^n}= \frac{P'(r)}{r^n}-\frac{nP(r)}{r^{n+1}}=\frac{h(u)}{r}
-\frac{nP(r)}{r^{n+1}} \nonumber \\
&=& \frac 1r \left[2nF(u)-(n-2)uf(u)\right]-\frac nr\left[u^{\prime 2}+2F(u)\right]-\frac{n(n-2)uu'}{r^2} \nonumber \\
&=& -\frac{n}{r^{n+1}}\cdot \left\{r^n\left[u^{\prime 2}+\frac{n-2}nuf(u)\right]+(n-2) r^{n-1}uu'\right\} \nonumber \\
&=&  -\frac{n}{r^{n+1}}\cdot P_{2}(r).
\ee

\begin{table}[h]
\centering
 \begin{tabular}{|c|c|c|c|c|}
 \hline
Name & Definition & Derivative & Labels       \\
\hline
$E(r)$ & \footnotesize $\dsp {u^{\prime 2}(r)}/2+F(u)$ & \footnotesize $\dsp -(n-1)u^{\prime 2}(r)/r$ & \footnotesize \equ{ef}, \equ{ep}\\
\hline
$\widehat E(r)$ & \footnotesize $\dsp r^{2(n-1)}E(r)$ & \footnotesize $\dsp 2(n-1)r^{2n-3}F(u)$ & \footnotesize \equ{H}, \equ{HA}\\
\hline
$P(r)$& \footnotesize $\dsp 2r^nE(r)+(n-2)r^{n-1}uu' $ & \footnotesize $\dsp r^{n-1}h(u)$ & \footnotesize \equ{P}, \equ{Pr}\\
\hline
$P_1(r)$ & \footnotesize $\dsp P(r)+r^nh_1(u)$&  & \footnotesize \equ{P1}\\
\hline
$P_2(r)$ & \footnotesize $\dsp P(r)-r^n h(u)/n $& \footnotesize $ -\,\frac{r^nu'}n \,h_2(u)$ & \footnotesize \equ{P2}, \equ{mmyp2r}\\
\hline
$\omega(r)$ & \footnotesize $\dsp -\,{ru'(r)}/{u(r)}$ & \footnotesize $\dsp {P_1(r)}/(r^{n-1}u^{2})$ & \footnotesize \equ{omega}, \equ{omegap}\\
\hline
$\varrho(r)$ & \footnotesize $\dsp r^{n-1} \left[f'(u)u'v-f(u)v'\right]$&
\footnotesize $\dsp r^{n-1}f''(u)u^{\prime 2}v$ & \footnotesize \equ{rho}, \equ{rhop}\\
\hline
$Q(r)$ & \footnotesize $\dsp r^n\left[u'v'+f(u)v\right]+(n-2)r^{n-1}u'v$ &
\footnotesize $\dsp 2r^{n-1}f(u)v$ &  \footnotesize \equ{qm}, \equ{qp}\\
\hline
$Q_1(r)$ & \footnotesize $\dsp Q(r)+r^{n-1}u'v$ & \footnotesize $\dsp r^{n-1}\left[u'v'+f(u)v\right]$& \footnotesize \equ{q1}\\
\hline
$Q_2(r)$ & \footnotesize $\dsp Q(r)+2r^{n-1}u'v$ & \footnotesize $\dsp 2r^{n-1}u'v'$ &\footnotesize \equ{q2}\\
\hline
$Q_n(r)$ & \footnotesize $\dsp Q(r)+nr^{n-1}u'v$& \footnotesize $\dsp r^{n-1}\left[nu'v'-(n-2)f(u)v\right]$ & \footnotesize \equ{qn}\\
\hline
$M(r) $ & \footnotesize $\dsp r^{n-1}(u'v-uv')$ & \footnotesize $\dsp  r^{n-1}[uf'(u)-f(u)]v$
& \footnotesize \equ{qm}, \equ{mp}\\
\hline
$T_1(r)$ & \footnotesize $\dsp Q(r)- g_1(u)M(r)  $ & \footnotesize $\dsp
- g_1'(u)u'\cdot M(r) $&\footnotesize \equ{t1g1}, \equ{t1p}\\
\hline
$T_2(r)$ & \footnotesize $\dsp Q(r)- g_2(u)M(r)
$&  &\footnotesize \equ{g2} \\
\hline
$B_a(r)$ & \footnotesize $\dsp Q(r)-aM(r)-2F_a(u)\cdot {r^{n-1}v}/{u'}$ & \footnotesize $
-2F_a(u)\cdot \frac{Q_n(r)}{ru'^2}$&\footnotesize \equ{ba}, \equ{bap}\\
\hline
\end{tabular}
\ms
\caption{{\small {\it {\clb List of auxiliary functions and their derivatives.} } }}  \label{tb1}\vspace*{-10pt} \end{table}

\subsection{Functions associated with $u$ and $v$}
In our analysis, we occasionally employ the equation of $v$ in \equ{ev} in the alternative form
\bl{rnvp} (r^{n-1}v')'=r^{n-1}\left(v''+\frac{n-1}rv'\right)=-r^{n-1}f'(u)v.\ee
Define the linear operator
\bl{L} L:=\frac{d^2}{dr^2}+\frac{n-1}r \frac{d}{dr}+f'(u).\ee
Then \equ{ev} takes the simple form $Lv=0$. Through routine calculations, we can check that
$Lu=uf'(u)-f(u)$, $L(u')=(n-1)u'/r^2$, $L(ru')=-2f(u)$, and
\ben L(v')=\frac{n-1}{r^2}\, v'-f''(u)u'v.\een
For arbitrary smooth functions $\mathfrak{g}$ and $\mathfrak{h}$, there is a useful
Wronskian identity given by Tao ((B.16) in \cite{tao}),
\bl{tao} \qquad \frac{d}{dr}\left[(r^{n-1}(\mathfrak{g} \mathfrak{h'}- \mathfrak{h} \mathfrak{g'})\right]
=r^{n-1}(\mathfrak{g} L\mathfrak{h}-\mathfrak{h} L\mathfrak{g}).\ee
For the Wronskian of $u'$ and $v'$,
\bl{rhoA}  \varrho(r)=r^{n-1} \left(u''v'-u'v''\right)=r^{n-1} \left[f'(u)u'v-f(u)v'\right], \ee
applying \equ{tao} gives
\bl{rhop} \varrho'(r)= r^{n-1}\left[v'L(u')-u'L(v')\right]=r^{n-1}f''(u)u^{\prime 2}v.\ee
As $M(r)=r^{n-1}(u'v-uv')$ is the Wronskian of $u$ and $v$, we find
\bl{mp} M'(r)=r^{n-1}vLu=r^{n-1}[uf'(u)-f(u)]v. \ee
Since
\ben Q(r)=r^n\left[u'v'+f(u)v\right]+(n-2)r^{n-1}u'v=r^{n-1}[v'(ru')-v(ru')'],\een
$Q(r)$ is the Wronskian of $v$ and $ru'$, and applying \equ{tao} yields
\bl{qp}  Q'(r)=-r^{n-1}vL(ru')=2r^{n-1}f(u)v. \ee
Write $Q_i(r)=Q(r)+ir^{n-1}u'v$. With the help of \equ{rnup} and \equ{qp}, we find
\ben Q'_i(r)=2r^{n-1}f(u)v-ir^{n-1}f(u)v+ir^{n-1}u'v'=r^{n-1}\left[iu'v'+(2-i)f(u)v\right].  \een
It leads to, for $i=1, 2$ and $n$,
\bl{q1}  Q_1(r)=Q(r)+2r^{n-1}u'v, \quad   Q'_1(r)=r^{n-1}\left[u'v'+f(u)v\right].  \ee
\bl{q2} Q_2(r)=Q(r)+2r^{n-1}u'v, \quad   Q'_2(r)=2r^{n-1}u'v'.\ee
\bl{qn} Q_n(r)=Q(r)+nr^{n-1}u'v, \quad  Q'_n(r)=r^{n-1}\left[nu'v'-(n-2)f(u)v\right].\ee

The following identities are closely related to the $Q$-family functions:
\bl{dvup} && \frac d{dr}\frac {v(r)}{u'(r)}=\frac {u'v'-u''v}{u'^2}  \nonumber \\
&=& \frac {r^nu'v'+r^nf(u)v+(n-1)r^{n-1}u'v}{r^nu'^2}=\frac {Q_1(r)}{r^nu'^2}.\ee
Similarly, we have
\bl{dvarn} && \varphi_n(r):= \frac d{dr}\frac{r^{n-1}v}{u'}= \frac {(n-1)r^{n-2}u'v+r^{n-1}u'v'-r^{n-1}u''v}{u'^2}  \nonumber \\
&=& \frac {r^nu'v'+r^nf(u)v+2(n-1)r^{n-1}u'v}{ru'^2}=\frac{Q_n(r)}{ru'^2}.\ee
With the help of \equ{qp} and \equ{q2}, a direct differentiation also gives
\bl{varp} && \frac{d}{dr}\frac{Q(r)}{r^{n-1}u'}= \frac{2r^{2(n-1)}f(u)u'v+r^{n-1}f(u)Q(r)  }{r^{2(n-1)}u'^2} \nonumber \\
&=& f(u)\cdot \frac{Q(r)+2r^{n-1}u'v}{r^{n-1}u'^2}
= f(u)\cdot\frac {Q_2(r)}{r^{n-1}{u'^2}}. \ee

For the bridging function
\bl{bobo} B_0(r)=Q(r)-2F(u)\cdot \frac{r^{n-1}v}{u'},\ee
we find
\bl{bop} B'_0(r)=\, -2F(u)\cdot \frac d{dr}\frac{r^{n-1}v}{u'}=\,-2F(u)\cdot \varphi_n(r)=\,-2F(u)\cdot \frac{Q_n(r)}{ru'^2}.\ee
For the general bridging function $B_a(r)$, we use
\ben F'_a(u)=f(u)-\frac a2[uf'(u)-f(u)]\een
to calculate
\bl{bap} B'_a(r)&=& Q'(r)-aM'(r)-2r^{n-1}F_a'(u)v-2F_a(u)\cdot \frac d{dr}\frac{r^{n-1}v}{u'}\nonumber\\
&=& r^{n-1}v\left\{2f(u)-a\left[uf'(u)-f(u)\right]-2F_a'(u)\right\}-2F_a(u)\cdot \frac d{dr}\frac{r^{n-1}v}{u'}\nonumber\\
&=& -2F_a(u)\cdot \varphi_n(r)=\,-2F_a(u)\cdot \frac{Q_n(r)}{ru'^2}.\ee
For
\ben T_1(r)=Q(r)-g_1(u(r))M(r), \quad g_1(u)=\frac {2f(u)}{uf'(u)-f(u)},\een
we find similarly that
\bl{t1p}  T_1'(r)&=& 2r^{n-1}f(u)v-2r^{n-1}g_1(u)\left[uf'(u)-f(u)\right]v-g_1'(u)u'\cdot M(r) \nonumber\\
&=& -g_1'(u)u'\cdot M(r).\ee

\subsection{Connection identities} Our final calculations for the connection identities are restricted to the model case $f(u)=-u+u|u|^{p-1}$. Establishing analogous identities for other nonlinearities is, in general, highly nontrivial. To verify \equ{conn}, we begin with the definitions of $P$ and $Q$ to compute
\ben Q(r)-P(r)\cdot \frac vu
&=& r^n\left[u'v'+f(u)v\right]- r^n\left[u^{\prime 2}+2F(u)\right]\cdot \frac vu   \\
&=& \frac{r^nv}u\left(\frac{uu'v'}v-u'^2+uf(u)-2F(u)\right)\\
&=& -\frac{ru'}{u}\left[r^{n-1}(u'v-uv')-\frac{p-1}{p+1}\cdot \frac {r^{n-1}v}{u'}\cdot {|u|^{p+1}}\right].
 \een
By defining
\ben \va(r)=\frac{p-1}{p+1}\cdot \frac {r^{n-1}v}{u'}\cdot {|u|^{p+1}},\een
we obtain
\ben Q(r)-P(r)\cdot \frac vu=\omega(r) \left[{M(r)}- \va(r)\right]\een
and confirm \equ{conn}. By differentiating $T_2=T_1+|u|^{1-p}M$ and apply \equ{t1p}, we calculate
\ben T_2'(r) &=& -\frac{2uu'}{|u|^{p+1}}\cdot M(r) -(p-1)\frac {uu'}{|u|^{p+1}}\cdot M(r)+|u|^{1-p}M'(r)  \\
&=& -\frac {(p+1)uu'}{|u|^{p+1}}\cdot M(r)+|u|^{1-p}\cdot (p-1)u|u|^{p-1}v  \\
&=& -\frac {(p+1)uu'}{|u|^{p+1}}\cdot M(r)+(p-1)r^{n-1}uv \\
&=& -(p+1)\cdot \frac{uu'}{|u|^{p+1}}\cdot \left[{M(r)}-\va(r) \right], \een
where the last two equalities verify \equ{t2p} and \equ{t2p2}, respectively. By \equ{t2p}, we express
\ben M(r)=\frac {|u|^{p+1}}{(p+1)uu'}\left[(p-1)r^{n-1}uv-T_2'(r)\right].\een
By substituting this into the definition of $T_2$, we obtain
\bl{t22A} T_2(r)
&=& Q(r)-\frac {2F(u)}{(p-1)uu'}\cdot \left[(p-1)r^{n-1}uv-T_2'(r)\right] \nonumber \\
&=&Q(r)-2F(u)\cdot\frac {r^{n-1}v}{u'}+\frac {2F(u)}{(p-1)uu'}\cdot T_2'(r)\nonumber\\
&=&B_0(r)+\frac {2F(u)}{(p-1)uu'}\cdot T_2'(r). \ee
This justifies \equ{t22} by connecting $T_2$ with the bridging function $B_0$. }

\section*{Acknowledgements}
  {\clb The author is sincerely grateful to the reviewer for invaluable guidance in improving the clarity, presentation, and organization of this work. In particular, the reviewer suggested conditions (C1)-(C9), which enabled the proofs in Sections 2 and 4 to be extended to more general nonlinearities. The author dedicates this work to the memory of his mentors, Lynn Erbe and James Serrin, and to his father, Honghan Tang.}


\begin{thebibliography}{A}
\bibitem{awy} {\sc W. Ao, J. Wei and W. Yao,}
 Uniqueness and nondegeneracy of sign-changing radial solutions to an almost critical elliptic
problem. {\it Adv. Differential Eqs.} {\bf 21} (2016), 1049--1084.

\bibitem{arw} {\sc D. Aronson and H. Weinberger,}
Multidimensional nonlinear diffusion arising in population genetics.
{\it Adv. in Math.} {\bf 30} (1978), 33--76.


\bibitem{bas} {\sc  P. Bates and J. Shi,} Existence and instability of spike layer solutions to singular perturbation problems. {\it J.
Funct. Anal.} {\bf 196} (2002), 211--264.


\bibitem{bl1} {\sc H. Berestycki and P. L. Lions,} Nonlinear scalar field equations. I.
Existence of a ground state. {\it Arch. Rational
Mech. Anal.} {\bf 82} (1983), 313--345.

\bibitem{bl2} {\sc H. Berestycki and P. L. Lions,} Nonlinear scalar field equations. II.
Existence of infinitely many solutions. {\it Arch. Rational
Mech. Anal.} {\bf 82} (1983), 347--375.

\bibitem{blp} {\sc H. Berestycki, P. L. Lions and L. A. Peletier,}
An ODE approach to the existence of positive solutions for semilinear problems in $R^n$.
{\it Indiana Univ. Math. J.} {\bf 30} (1981), 141--157.

{\clb \bibitem{br} {\sc G. Birkhoff and G. C. Rota,}
{\it Ordinary differential equations,}
Blaisdell Pub. Co., Waltham, Mass, 1969. }


\bibitem{chl} {\sc C. Chen and C. Lin,}
Uniqueness of the ground state solution of $\Delta u+f(u)=0$ in
$\mathbb{R}^n$, $n\ge 3$. {\it Comm. Partial Diff. Eqs.} {\bf 16} (1991), 1549--1572.


\bibitem{cj}
{\sc C. B. Clemons and C. Jones,}
A geometric proof of the Kwong-McLeod uniqueness result.
{\it SIAM J. Math. Anal.} {\bf 24} (1993),  436--443.

{\clb \bibitem{cl} {\sc E. A. Coddington and N. Levinson,}
{\it Theory of ordinary differential equations,}
McGraw-Hill, New York, 1955.}

\bibitem{coffman}
{\sc C. V. Coffman,} Uniqueness of the ground state solution
for $\Delta u-u+u\sp{3}=0$ and a variational characterization
of other solutions.
{\it Arch. Rational Mech. Anal.} {\bf 46} (1972), 81--95.

\bibitem{cls} {\sc A. Cohen, Z. Li and W. Schlag,}
Uniqueness of excited states to $-\Delta u+u-u^3= 0$ in three dimensions.
{\it Analysis \& PDE} {\bf 17} (2024), 1887--1906.


\bibitem{cef}
{\sc C. Cort\'azar, M. Elgueta and P. Felmer,} Uniqueness of positive solutions of
$\Delta u + f(u) = 0$ in $\mathbb{R}^N$, $N\ge 3$.
{\it Arch. Rational Mech. Anal.} {\bf 142} (1998), 127--141.

\bibitem{cgh}
{\sc C. Cort\'azar, M. Garc\'ia-Huidobro and P. Herreros,} Multiplicity results for sign changing bound state solutions of a semilinear equation. {\it J. Diff. Eqs.}
{\bf 259} (2015), 7108--7134.

\bibitem{cgy}
{\sc C. Cort\'azar, M. Garc\'ia-Huidobro and C. S. Yarur,}
On the uniqueness of sign changing bound state solutions
of a semilinear equation.
{\it Ann. Inst. H. Poincar\'e Anal. Non Lin\'eaire} {\bf 28} (2011), 599--621.

\bibitem{cof} {\sc R. C\^ote and X. Friederich,}
 On smoothness and uniqueness of multi-solitons of the nonlinear Schr\"odinger equations.
 {\it Comm. Partial Diff. Eqs.} {\bf 46} (2021), 2325--2385.

\bibitem{ddg}
{\sc J. D\'avila, M. del Pino and I. Guerra,} Non-uniqueness of positive ground states of
non-linear Schr\"odinger equations. {\it Proc. Lond. Math. Soc.}
{\bf 106} (2013), 318--344.

\bibitem{del}
{\sc M. del Pino and J. Dolbeault,}
Best constants for Gagliardo-Nirenberg inequalities and applications to nonlinear
diffusions. {\it J. Math. Pures Appl.} {\bf 81} (2002), 847--875.

\bibitem{dgm}
{\sc J. Dolbeault, M. Garc\'ia-Huidobro and R. Man\'asevich,}
Qualitative properties and existence of sign changing solutions
with compact support for an equation with a p-Laplace operator. {\it Adv. Nonlinear Stud.} {\bf 13} (2013), 149--178.

\bibitem{erbe}
{\sc L. Erbe and M. Tang,}
Uniqueness theorems for positive solutions of quasilinear
elliptic equations in a ball.
{\it J. Diff. Eqs.} {\bf 138} (1997), 351--397.

\bibitem{flr} {\sc R. Finkelstein, R. LeLevier and M. Ruderman,}
 Nonlinear spinor fields. {\it Phys. Rev.} {\bf 83} (1951), 326-332.

\bibitem{fisher} {\sc R. A. Fisher,} The wave of advance of advantageous genes.
{\it Ann. Eugenics} {\bf 7} (1937), 335--369.

\bibitem{fls} {\sc B. Franchi, E. Lanconelli and J. Serrin,}
Existence and uniqueness of non-negative solutions of quasilinear equations in $\mathbb{R}^n$.  {\it  Adv. in Math.} {\bf 118} (1996), 177--243.


\bibitem{fra} {\sc R. L. Frank,} Ground states of semi-linear PDE. Lecture notes from the ``Summer school on Current Topics in Mathematical Physics'', CIRM Marseille, Sept. 2013., 2013.

\bibitem{frl} {\sc R. L. Frank and E. Lenzmann,} Uniqueness of non-linear ground states for fractional Laplacians in $\mathbb{R}$. {\it Acta Math.} {\bf 210} (2013), 261--318.

\bibitem{frls} {\sc R. L. Frank, E. Lenzmann and L. Silvestre,} Uniqueness of radial solutions for the fractional Laplacian. {\it Comm. Pure Appl. Math.} {\bf 69} (2016), 1671--1726.

\bibitem{gnn} {\sc B. Gidas, W. M. Ni and L. Nirenberg}, Symmetry of positive solutions of nonlinear elliptic equations in $\mrn$. {\it Mathematical Analysis and Applications, Part A}, Adv. Math. Supp. Stud.,
{\bf 7a}, 369--402, New York: Academic Press (1981).

\bibitem{ham} {\sc S. P. Hastings and J. B. McLeod,} {\it Classical methods in ordinary differential equations:  With applications to boundary value problems.}
Graduate Studies in Mathematics {\bf 129}. American Mathematical Society, Providence, RI, 2012.

\bibitem{hty} {\sc M. Huang, M. Tang and J. Yu,}
{\it Wolbachia} infection dynamics by reaction-diffusion equations.
{\it Sci. China Math.} {\bf 58} (2015), 77--96.


\bibitem{jk} {\sc C. Jones and T. K\"upper}, On the infinitely many solutions of semilinear elliptic equation.
{\it SIAM J. Math. Anal.} {\bf 17} (1986), 803--835.

\bibitem{kit} {\sc R. Killip, T. Oh, O. Pocovnicu and M. Visan,}
 Solitons and scattering for the cubic-quintic nonlinear Schr\"odinger equation on
 $\mathbb{R}^3$. {\it Arch. Rational Mech. Anal.} {\bf 225} (2017), 469--548.


\bibitem{kol} {\sc I. I. Kolodner,} Heavy rotating string -- a nonlinear eigenvalue problem. {\it Comm. Pure Appl.
Math.} {\bf 8} (1955), 395--408.


\bibitem{k}
{\sc M. K. Kwong,}
Uniqueness of positive solutions of $\Delta u-u+u^p=0$ in $\mathbb{R}^n$.
{\it Arch. Rational Mech. Anal.} {\bf 105} (1989), 243--266.

\bibitem{kl}
{\sc M. K. Kwong and Y. Li,} Uniqueness of radial solutions of semilinear elliptic equations.
{\it Trans. Amer. Math. Soc.} {\bf 333} (1992), 339--363.

\bibitem{kz}
{\sc M. K. Kwong and L. Zhang,}
Uniqueness of the positive solution of $\Delta u+f(u)=0$ in
an annulus. {\it Diff. Int. Eqs.} {\bf 4} (1991), 583--596.

\bibitem{len}{\sc M. Lewin and S. R. Nodari,}
 The double-power nonlinear Schr\"odingger equation and its generalizations: uniqueness, non-degeneracy and applications.
 {\it Calc. Var. Partial Differ. Eqs.} {\bf 59} (2020), 1--49.

\bibitem{ln} {\sc Y. Lou and W.-M. Ni,}
Diffusion, self-diffusion and cross-diffusion
{\it J. Diff. Eqs.} {\bf 131} (1996), 79--131.


\bibitem{mps}{\sc D. W. McLaughlin, G. C. Papanicolaou, C. Sulem and P. L. Sulem,} Focusing singularity of the cubic Schr\"odinger equation. {\it Phys. Rev. A} {\bf 34} (1986), 1200--1210.


\bibitem{mc}
{\sc K. McLeod,}
Uniqueness of positive radial solutions of $\Delta u+f(u)=0$ in $\mathbb{R}^n$, II.
{\it Tran. Amer. Math. Soc.} {\bf 339}  (1993), 495--505.

\bibitem{ms81}
{\sc K. McLeod and J. Serrin,} Uniqueness of solutions of semilinear Poisson equations.
{\it Proc. Nat. Acad. Sci. USA} {\bf 78} (1981), 6592--6595.


\bibitem{ms87}
{\sc K. McLeod and J. Serrin,}
Uniqueness of positive radial solutions of $\Delta u+f(u)=0$ in $\mathbb{R}^n$.
{\it Arch. Rational Mech. Anal.}
{\bf 99} (1987), 115--145.

\bibitem{mtw} {\sc K. McLeod, W. C. Troy and F. B. Weissler,}
Radial solutions of $\Delta u + f(u) = 0$ with
prescribed numbers of zeroes. {\it J. Diff. Eqs.}
{\bf 83} (1990), 368--378.

\bibitem{ne1} {\sc Z. Nehari,} Characteristic values associated with a class of nonlinear second-order differential equations.
    {\it Acta Math.} {\bf 105} (1961), 141--175.

\bibitem{ne2} {\sc Z. Nehari,} On a nonlinear differential equation arising in nuclear physics. {\it Proc. Royal
Irish Academy} {\bf 62} (1963), 117--135.

\bibitem{n}
{\sc W.-M. Ni,}
Uniqueness of solutions of nonlinear Dirichlet problems.
{\it J. Diff. Eqs.} {\bf 50} (1983), 289--304.

\bibitem{nn}
{\sc W.-M. Ni and R. D. Nussbaum,}
Uniqueness and nonuniqueness for positive radial solutions
of $\Delta u+f(u,r)=0$.
{\it Comm. Pure Appl. Math.} {\bf 38} (1985), 67--108.

\bibitem{nt} {\sc W.-M. Ni and I. Takagi,}
Locating the peaks of least-energy solutions to a semilinear Neumann problem.
{\it Duke Math. J.} {\bf 70} (1993),  247--281.

\bibitem{nit} {\sc W.-M. Ni and M. Tang,}
Turing patterns in the Lengyel-Epstein system for the CIMA reaction
{\it Trans. Am. Math. Soc.} {\bf 357} (2005), 3953--3969.

\bibitem{ous} {\sc T. Ouyang and J. Shi,}
Exact multiplicity of positive solutions for a class of semilinear problems: II.
{\it J. Diff. Eqns.} {\bf 158} (1999), 94--151.


\bibitem{ps}
{\sc L. A. Peletier and J. Serrin,}
Uniqueness of positive solutions of semilinear equations in $\mathbb{R}^n$.
{\it Arch. Rational Mech. Anal.} {\bf 81} (1983), 181--197.

\bibitem{poh}
{\sc S. I. Pohozaev,}
Eigenfunctions of the equation $\Delta u + \lambda f(u) = 0$. {\it Soviet Math.}
{\bf 5} (1965), 1408--1411.

\bibitem{pus} {\sc P. Pucci and J. Serrin,}
Uniqueness of ground states for quasilinear elliptic operators.
{\it Indiana Univ. Math. J.} {\bf 47} (1998), 501--528.

\bibitem{rr}
{\sc N. Rosen and H. B. Rosenstock,}
 The forces between particles in a nonlinear field theory.
 {\it Phys. Rev.} {\bf 85} (1952), 257--259.

\bibitem{ryd} {\sc G. H. Ryder,} Boundary value problems for a class of nonlinear
differential equations.
{\it Pac. J. Math.} {\bf 22} (1967), 477--503.

\bibitem{sch}
{\sc W. Schlag,} Stable manifolds for an orbitally unstable nonlinear Schr\"odinger equation.
{\it Ann. of Math.} (2) {\bf 169} (2009), 139--227.

\bibitem{serrin81}
{\sc J. Serrin,} Phase transitions and interfacial layers for van der Waals fluids, in {\it Recent methods in nonlinear analysis and applications}, A. Canfora, S. Rionero, C. Sbordone, and G. Trombetti, eds., Liguori Editore,
Naples, 169--175 (1981).


\bibitem{st}
{\sc J. Serrin and M. Tang,} Uniqueness of ground
states for quasilinear elliptic equations.
{\it Indiana Univ. Math. J.} {\bf 49} (2000), 897--923.

\bibitem{st77} {\sc W. A. Strauss,} Existence of solitary waves in higher dimensions.
{\it Comm. Math. Phys.} {\bf 55} (1977), 149--162.

\bibitem{tang03} {\sc M. Tang,} Uniqueness of positive radial solutions for $\Delta u -u + u^p
= 0$ on an annulus. {\it J. Diff. Eqs.} {\bf 189} (2003), 148--160.


\bibitem{tao} {\sc T. Tao,} {\it Nonlinear dispersive equations. Local and global analysis}. CBMS Regional Conference Series in Mathematics {\bf 106}. Published for the Conference Board of the Mathematical Sciences, Washington, DC; by the American Mathematical Society, Providence, RI, 2006.

\bibitem{we1} {\sc M. I. Weinstein,} Nonlinear Schr\"odinger equations and sharp interpolation estimates. {\it Comm. Math. Phys.} {\bf 87} (1983), 567--576.

\bibitem{we2} {\sc M. I. Weinstein,} Modulational stability of ground states of nonlinear Schr\"odinger equations. {\it SIAM J. Math. Anal.} {\bf 16} (1985), 472--491.


\bibitem{yan} {\sc E. Yanagida,} Uniqueness of positive radial solutions of $\Delta u+g(r)u+h(r)u^p=0$ in $\mathbb{R}^n$.
{\it Arch. Rational Mech. Anal.} {\bf 115} (1991), 257--274.
\end{thebibliography}
\end{document}